\documentclass{article}
\usepackage{latexsym}
\usepackage{amsmath,amsthm}
\usepackage{enumerate}
\usepackage{amssymb}

\usepackage[margin=1.1in,dvips]{geometry}

\def\tr{\mbox{tr}}

\newcommand{\vol}{\textnormal{vol}}

\def\f12{\frac 1 2}

\def\a{\alpha}
\def\b{\beta}
\def\ga{\gamma}
\def\Ga{\Gamma}

\def\De{\Delta}
\def\ep{\epsilon}
\def\la{\lambda}
\def\La{\Lambda}
\def\si{\sigma}
\def\Si{\Sigma}
\def\om{\omega}

\def\th{\theta}
\def\Th{\Theta}

\def\pa{\partial}
\def\les{\lesssim}

\def\f12{\frac 1 2}

\def\e{\text{e}}

\newtheorem{thm}{Theorem}

\newtheorem{prop}{Proposition}

\newtheorem{lem}{Lemma}
\newtheorem{cor}{Corollary}
\newtheorem{remark}{Remark}
\newtheorem{Def}{Definition}
\begin{document}

\title{On the geodesic hypothesis in general relativity}
\date{}
\author{Shiwu Yang \footnote{Department of Mathematics, Princeton University, NJ 08544 USA, Email: shiwuy@math.princeton.edu}}
\maketitle

\begin{abstract}
In this paper, we give a rigorous derivation of Einstein's geodesic
hypothesis in general relativity. We use small material bodies $\phi^\ep$ governed by the nonlinear Klein-Gordon equations to approximate the test particle. Given a
vacuum spacetime $([0, T]\times\mathbb{R}^3, h)$, we consider the initial value problem for the
Einstein-scalar field system. For all sufficiently small $\ep$ and
$\delta\leq \ep^q$, $q>1$, where $\delta$, $\ep$ are the amplitude
and size of the particle, we show the existence of solution $([0,
T]\times\mathbb{R}^3, g, \phi^\ep)$ to the Einstein-scalar field system with the property that the energy of the particle
$\phi^\ep$ is concentrated along a timelike geodesic. Moreover, the
gravitational field produced by $\phi^\ep$ is negligibly small in
$C^1$, that is, the spacetime metric $g$ is $C^1$ close to the given vacuum metric $h$.
These results generalize those obtained by D. Stuart  in
\cite{psedoStuart}, \cite{einsteinStuart}.
\end{abstract}

\section{Introduction}
In general relativity, Einstein's geodesic hypothesis,
which corresponds to Newton's first law of motion in classical
mechanics, states that a free massive test
particle will follow a timelike geodesic in the spacetime, where by
free we mean in the absence of all external forces except
gravitation, which is ascribed to the spacetime curvature instead of
a force. For the concept of test particle, one has to ignore its
internal structure as well as the gravitational field produced by
it. This paper is devoted to a rigorous mathematical derivation of
geodesic hypothesis for the toy model that the particles are stable solitons for a
class of nonlinear Klein-Gordon equations. Since the particles
will interact with the background spacetime, we consider the
following Einstein-scalar field system
\begin{equation}
\label{EQUATION}
\begin{cases}
R_{\mu\nu}-\f12 R g_{\mu\nu}=T_{\mu\nu}(g, \phi^\ep;\mathcal{V}_{\ep, \delta}),\\
\Box_{g}\phi^\ep-\mathcal{V}_{\ep, \delta}'(\phi^\ep)=0.
\end{cases}
\end{equation}
Here $\Box_g$ is the covariant wave operator for the unknown
spacetime metric $g_{\mu\nu}$. $\phi^\ep$ is the complex scalar
field representing the particles. $R_{\mu\nu}$, $R$ are the
Ricci, scalar curvatures of the metric $g$ respectively.
$\mathcal{V}'_{\ep, \delta}$ is the first variation of the potential
$\mathcal{V}_{\ep, \delta}$. $T_{\mu\nu}$ is the energy momentum
tensor for the scalar field $\phi^\ep$ and is given as follows
\begin{equation}
\label{defofT} T_{\mu\nu}(g,\phi^\ep;\mathcal{V}_{\ep,
\delta})=<\pa_{\mu}\phi^\ep, \pa_{\nu}\phi^\ep>-\f12
g_{\mu\nu}\left(<\pa^\ga\phi^\ep, \pa_\ga\phi^\ep>+2\mathcal{V}_{\ep,
\delta}(\phi^\ep)\right),
\end{equation}
where $<, >$ denotes the inner product of two complex numbers, namely, $<a, b>=\f12(a\bar b+\bar a b)$.

We study the motion of the particles dynamically. The test particle is approximated by a sequence of actual material bodies with small size $\ep$ and small amplitude $\delta$. In terms of local coordinate $(t, x)$, the
normalized scalar field $\phi$ takes the following form
\[
 \phi(t, x)=\delta^{-1} \phi^\ep (\ep t, \ep x).
\]
Then the corresponding potential $\mathcal{V}$ for the normalized scalar field $\phi$ is given by $\mathcal{V}_{\ep,
\delta}$ via
\begin{equation}
\label{potenscal}\mathcal{V}(\phi)= \delta^{-2}\ep^{2}\mathcal{V}_{\ep,
\delta}(\phi^\ep).
\end{equation}
In particular, one has $\mathcal{V}_{\ep, \delta}'(\phi^\ep)=\delta\ep^{-2}\mathcal{V}'(\delta^{-1}\phi^\ep)$.
For a solution $(\mathcal{M}, g, \phi^\ep)$ of the above
Einstein-scalar field system, we find that the normalized scalar
field $\phi$ must satisfy the nonlinear wave equation
\begin{equation}
\label{fixbackKGeq}
 \Box_{g^\ep}\phi-\mathcal{V}'(\phi)=0,\quad g^\ep(t, x)=g(\ep t, \ep x).
\end{equation}
As the particle will not change its shape during the motion, we use solitons of the
above wave equation to approximate the normalized scalar field $\phi$. By a
soliton centered at point $P$ in the spacetime, we mean a solitary
solution of the form $\e^{i\om t}f(x)$ to the equation
$\Box_{g^\ep(P)}\phi-\mathcal{V}'(\phi)=0$ on the flat space
$\mathbb{R}^{3+1}$ with metric $g^\ep(P)$. For concreteness, we will
assume that the potential $\mathcal{V}$ takes the following form
\begin{equation}
\label{defofmV} \mathcal{V}(\phi)=\frac{m^2}{2}|\phi|^2-
N(\phi)=\frac{m^2}{2}|\phi|^2-\frac{1}{p+1}|\phi|^{p+1},\quad 1<p<5.
\end{equation}
This type potentials guarantees the existence of solitons, which decays exponentially \cite{exsolitonLions}. If $\phi$ is close to such solitons, then the
particles $\phi^\ep$ is localized to a region of size $\ep$.
Hence solitons centered at the position of the particles can be used
to approximate the test particle.

We consider the initial value problem for the Einstein-scalar field system on a given vacuum
spacetime which is diffeomorphic to $([0, T]\times\mathbb{R}^3, h)$
with coordinate system $(t, x)$.  We assume the vacuum Einstein metric $h$ is $C^4$
and satisfies the estimates
\begin{equation}
\label{condfix}
\begin{split}
 &K_0^{-1} I_{3\times 3}\leq (h^{kl})(t, x)\leq K_0 I_{3\times 3},\quad \forall (t, x)\in[0, T]\times \mathbb{R}^3,\\
& K_0^{-1}\leq -h^{00}\leq K_0,\quad |h^{\mu\nu}|\leq K_0, \quad
\forall \mu, \nu=0, 1, 2, 3,\\
&\|h\|_{C^4([0, T]\times\mathbb{R}^3)}+\||x|^{\f12}\pa h\|_{L^\infty}+\||x|\pa^{2+s}h\|_{L^\infty}\leq K_0,\quad |s|\leq 2
\end{split}
\end{equation}
for some positive constant $K_0$, where $h=h_{\mu\nu}(t, x)dx^\mu
dx^\nu$, $x^0=t$, $h^{\mu\nu}(t, x)= (h^{-1})_{\mu\nu}(t, x)$, $\pa$
is short for $(\pa_t, \pa_{x_1}, \pa_{x_2}, \pa_{x_3})$. The initial data $(\mathbb{R}^3, \bar g, \bar K,
\phi_0^\ep, \phi_1^\ep)$, where $\bar g$ is a Riemannian metric on
$\mathbb{R}^3$, $\bar K$ is a symmetric two tensor, $(\phi_0^\ep,
\phi_1^\ep)$ are the initial data for the particle, have to satisfy the constraint
equations
\begin{equation}
 \label{constrainteq}
\begin{cases}
 \bar R(\bar g)-|\bar K|^2+(tr \bar K)^2=|\phi_1^\ep|^2+|\bar\nabla \phi_0^\ep|^2+2\mathcal{V}
 _{\ep, \delta}(\phi_0^\ep),\\
\bar{\nabla}^j \bar K_{ij}-\bar{\nabla}_i tr\bar K=<\phi_1^\ep,
\bar\nabla_i\phi_0^\ep>,
\end{cases}
\end{equation}
where $\bar \nabla$ is the covariant derivative with respect to the
Riemannian metric $\bar g$ on the initial hypersurface. We assume the initial data satisfy
\begin{equation}
\label{initialcond}
\begin{split}
 &\|\delta^{-1}\phi_0^\ep(x)-\phi^\ep_S(x;\la_0)\|_{H^3_\ep}+\ep\|\delta^{-1}\phi_1^\ep(x)-
 n\phi_S^\ep(x;\la_0)\|_{H^2_\ep}\leq C_0\ep,\\
&\|\nabla(\bar g-\bar h)\|_{H_\ep^2(\Si_0)}+\|\bar K-\bar
k\|_{H_\ep^2(\Si_0)}\leq C_0\delta^2 \ep^{-1}
\end{split}
\end{equation}
for some constant $C_0$, where $(\bar h, \bar k)$ is the initial data for the given vacuum spacetime $([0, T]\times \mathbb{R}^3, h)$.
The weighted Sobolev norm is defined as follows
\begin{equation}
\label{defofscasob}
 \|f\|_{H_\ep^s}=\|f(\ep\cdot)\|_{H^s}=\sum\limits_{|\a|=0}^{s}\ep^{|\a|-\frac{3}{2}}\|\nabla^\a
 f\|_{L^2}.
\end{equation}
The first inequality of \eqref{initialcond} means that the particle is close to a soliton initially. For a more precise definition of the solitons $\phi_S^\ep(x;\la_0)$,
$n\phi_S^\ep(x;\la_0)$, see details in Section 2.

The initial center of mass of the particle $\xi_0$ together with the initial velocity of the soliton $u_h(0)$ (as the particle is initially
close to a soliton), which depends on the initial data, determines a timelike
geodesic $(t, \ga_0(t))$ on the given spacetime $([0, T]\times \mathbb{R}^3, h)$ such that
$\ga_0(0)=\xi_0$, $\ga_0'(0)=u_h(0)$.

We have the following main result
\begin{thm}
 \label{thmbaby}
Let $([0, T]\times \mathbb{R}^3, h)$ be a vacuum spacetime satisfying
\eqref{condfix}. Assume $2\leq p<\frac{7}{3}$. Assume $\delta\leq
\ep^q$, $q>1$ or $\delta=\ep_0\ep$. Suppose the initial data $(\mathbb{R}^3,
\bar g, \bar K, \phi_0^\ep, \phi_1^\ep)$ satisfy conditions
\eqref{constrainteq}, \eqref{initialcond}. Then there exists
$\ep^{*}>0$ such that for all $\ep,\ep_0\in[0, \ep^{*})$ the Cauchy
problem for system \eqref{EQUATION} admits a unique (up to
diffeomorphism) solution $([0, T]\times \mathbb{R}^3, g, \phi^\ep)$ with the
following properties: there exists a foliation $[0,
T^*]\times\Si_\tau$ of $[0, T]\times \mathbb{R}^3$ with coordinates $(s, y)$,
where $\Si_\tau\subseteq \mathbb{R}^3$ and $T^*$ depends only on
$T$, $h$ and $\la_0$, as well as a $C^1$ curve $\la(s)=(\om(s),
\th(s), \xi(s), u_h(s))\in\La_{\textnormal{stab}}(s)$ such that
\begin{itemize}
 \item[(1)] The spacetime $([0, T]\times \mathbb{R}^3, g, \phi^\ep)$ is close to the given vacuum spacetime $([0, T]\times \mathbb{R}^3, h)$
\begin{equation}
\label{metricclose} \|\pa(g-h)\|_{H^2_\ep(\Si_s)}\leq C \ep,\quad
\forall s\in[0, T^*];
\end{equation}
\item[(2)] $\phi^\ep(t, x)$ is approximated by some translated soliton centered at $\xi(s)$
\begin{equation}
\label{solitonclose}
 \|\delta^{-1}\phi^\ep(s, y)-\phi^\ep_S(y;\la(s))\|_{H^3_\ep(\Si_s)}+\ep
 \|\delta^{-1}\phi_t^\ep(s, y)-\psi_S^\ep(y;\la(s))
\|_{H^2_\ep(\Si_s)}\leq C\ep,\quad \forall s\in[0, T^*];
\end{equation}
\item[(3)] The center of the particle $\xi(s)$ is close to the given timelike
geodesic $(s, \ga_0(s))$
\begin{equation}
\label{geoclose}
|\xi(s)-\ga_0(s)|+|\om(s)-\om_0|+|u_h(s)-\ga_0'(s)|\leq C\ep,\quad
\forall s\in[0, T^*],
\end{equation}
\end{itemize}
where the constant $C$ is independent of $\ep$ and the definition for $\La_{\textnormal{stab}}(s)$ can be found in Section 2.
\end{thm}
\begin{remark}
Existence of initial data satisfying conditions in the theorem will be shown in the last section, see
Theorem \ref{exdata}.
\end{remark}
\begin{remark}
 The lower bound of $p$ ($\geq 2$) is required for regularity purpose: the spacetime metric has to be in $C^1$.
 The upper bound $p<\frac{7}{3}$ is needed to guarantee the existence of stable solitons.
\end{remark}
\begin{remark}
We have the same conclusion for much more general potentials $\mathcal{V}$, for example, potentials $\mathcal{V}$ satisfying
conditions given in \cite{einsteinStuart}.
\end{remark}
The theorem shows that the motion of the small bodies $\phi^\ep$ can be described by a timelike geodesic in the spacetime in the
limit when the size and the energy of the small bodies go to zero. Moreover, the gravitational field produced by the particle is negligibly small. In particular, this gives a rigorous
derivation of Einstein's geodesic hypothesis for the model when the test particle is approximated by scalar fields governed by the nonlinear Klein-Gordon equations. Nevertheless one
can also expect the same result for general macroscopic matters such as fluids or elastic matters, see \cite{thomas}, \cite{taub}. For a more general theorem, we refer to the work
\cite{gerochmotion} of Ehlers and Geroch. However, our theorem is not a direct consequence of the result of Ehlers and Geroch as the main difficulty of our theorem is to sovle the Einstein-scalar
system dynamically.

On the other hand, our theorem generalizes the result in \cite{einsteinStuart} obtained
by D. Stuart in two ways. First, our result describes the long time
behavior of the test particles. In
\cite{einsteinStuart}, it was shown that the solution $(g,
\phi^\ep)$ of the Einstein-scalar field system exists only in a
small portion $[0, t^*)\times\mathbb{R}^3$ of the given spacetime
$[0, T]\times\mathbb{R}^3$ for some small positive constant $t^*$.
Here we can extend the solution to the whole spacetime $[0, T]\times\mathbb{R}^3$
and at the same time the energy of the particle is concentrated along a
timelike geodesic for arbitrarily large given time $T$. Second, our result holds
under the assumption $\delta\leq \ep^q$, $q>1$ which is weaker than
$\delta\leq \ep^q$, $q\geq \frac{7}{4}$ imposed in
\cite{einsteinStuart}. Recall that $\delta$ denotes the amplitude or
the energy of the particle. If one wants to show that the
gravitational field produced by the particle is negligibly small in
$C^1$, one has to show that the energy momentum tensor $T_{\mu\nu}$ is small in
$H^{s-1}$, $s>\frac{5}{2}$. This requires $\delta$ to be
sufficiently small in terms of $\ep$. On the other hand, notice that
$\delta^{-1}\phi^\ep$ is close to some soliton which has size $1$. In view of the
equation \eqref{potenscal}, the potential $\mathcal{V}_{\ep,
\delta}$, in particular the energy momentum tensor $T_{\mu\nu}$, has
size $\delta^2\ep^{-2}$. Based on the heuristics that in the
limiting case when $\ep$ goes to zero, the potential
 $\mathcal{V}_{\ep, \delta}$ should be bounded, we see that the condition $\delta\leq \ep^q$, $q\geq 1$ is needed
 to allow a proof of geodesic hypothesis in this setting. Our theorem thus answers
the question of D. Stuart on the optimal value of $q$ when $q>1$.
Furthermore, the
condition can be even weaker by merely assuming that
$\delta=\ep_0\ep$ for some small constant $\ep_0$ independent of
$\ep$. This is robust and interesting as then the potential
$\mathcal{V}_{\ep, \delta}$ always has size $\ep_0^2$ even when
$\ep$ goes to zero.

The key ingredient of our proof for Theorem \ref{thmbaby} is the introduction of a Fermi coordinate system adapted to the center of mass of the
particles, see more detailed description in Section 3. The improvement on the amplitude $\delta\leq \ep^q$, $q>1$ is based on a more careful
estimate for the energy momentum tensor.

The plan of this paper is as follows: in Section 2, we will address the basic setup, define solitons in
Minkowski space and summarize the known results related to stability of solitons. In Section 3, we describe the main ideas for the proof of the main
theorem. In Section 4, we construct the Fermi coordinate system such
that the Christoffel symbols are vanishing along a given timelike
geodesic. In Section 5, we reduce the Einstein-scalar field system to a
hyperbolic system by choosing the relatively harmonic gauge
condition. In Section 6, we prove the orbital stability of stable
solitons along a timelike geodesic up to time $T/\ep$ on a slowly
varying background and show the higher Sobolev estimates for the
solution which are used to control the energy momentum tensor. In Section 7, we conclude the main theorem. In the last
section, we discuss the existence of initial data.

\section{Preliminaries and Stability Results in Minkowski Space}

In Minkowski space, for the nonlinear Klein-Gordon equation
\begin{equation}
 \label{EquationMink}
\Box \phi-m^2 \phi+|\phi|^{p-1}\phi=0, \quad \Box=-\pa_{t}^2+\Delta=-\pa_t^2+\pa_{x_1}^2+\pa_{x_2}^2+\pa_{x_3}^2, \quad m>0, \quad p>1
\end{equation}
of complex functions $\phi(t, x)$, looking for the solitons or stationary waves, that is,
$\phi$ is of the form $\e^{i\om t}f_\om(x)$, $\om\in \mathbb{R}$
, we are led to consider the elliptic equation
\begin{equation}
\label{ellipeq}
 \Delta f_\om-(m^2-\om^2)f_\om+|f_\om|^{p-1}f_\om=0
\end{equation}
on $\mathbb{R}^3$. Such elliptic equation has been studied extensively and has infinite many
solutions \cite{manysolLions}, of which we are particularly interested in the
ground state, that is, solution of \eqref{ellipeq} with lowest
energy
\[
E_\om(v)=\f12\int_{\mathbb{R}^3}|\nabla
v|^2+(m^2-\om^2)|v|^2-\frac{2}{p+1}|v|^{p+1}dx.
\]
It can be shown that the ground state has to be positive, radial
symmetric. The existence of ground state as well as its properties
are summarized in the following theorem, see \cite{exsolitonLions}, \cite{odelions}, \cite{uniqueSolMcLeod}, \cite{decaySerrin}
, \cite{exsoliStrauss}, \cite{moduStuart} and reference therein.
\begin{thm}
 \label{propoff}
For $1<p<5$, $\om\in(-m, m)$, there exists a unique, positive, radial symmetric solution
 $f_\om(x)\in H^4(\mathbb{R}^3)\cap C^4(\mathbb{R}^3)$ of the equation \eqref{ellipeq}.
It is decreasing in $|x|$ with
the following properties:
\begin{itemize}
\item[1]. Exponential decay up to fourth order derivatives
\begin{equation}
\label{decayoffom}
|\nabla^\a f_{\om}(x)|\leq C(\om)e^{-c(\om)|x|},\quad \forall x\in \mathbb{R}^3, \quad |\a|\leq 4
\end{equation}
for some positive constant $c(\om)$;
\item[2]. Asymptotical behavior
\begin{equation}
\label{decayrateoff}
\lim\limits_{|x|\rightarrow\infty}\frac{f_{\om}'(|x|)}{f_{\om}(|x|)}=-\sqrt{m^2-\om^2};
\end{equation}
\item[3]. Scaling of the solutions
\begin{equation}
\label{scalling}
f_\om(x)=(m^2-\om^2)^{\frac{1}{p-1}}f(\sqrt{m^2-\om^2}x),
\end{equation}
where $f(x)$ is the solution for $m^2-\om^2=1$;
\item[4]. Identities of the energy
\begin{equation}
\label{lempropf}
\frac{3(p-1)(m^2-\om^2)}{2\|\nabla f_\om\|_{L^2}^2}=\frac{5-p}{2\|f_\om\|_{L^2}^2}
=\frac{(p+1)(m^2-\om^2)}{\|f_\om\|_{L^{p+1}}^{p+1}}.
\end{equation}
\end{itemize}
\end{thm}
Existence of ground state has been obtained in \cite{exsolitonLions}, \cite{odelions}, \cite{exsoliStrauss}.
K. McLeod \cite{uniqueSolMcLeod} proved the uniqueness of the ground state.
In \cite{exsolitonLions}, it was shown that $f_\om(x)$ is $C^{2}$ and decays
 exponentially up to second order derivatives, which was generalized to be $C^4$
in \cite{moduStuart}. The asymptotic behavior of the
solution \eqref{decayrateoff} was first proven in
\cite{decaySerrin}. Using integration by parts, identities
\eqref{lempropf} follow by multiplying the equation \eqref{ellipeq}
with $f_\om$, $x\cdot \nabla f_\om$ respectively.

\bigskip

Having the basic solitons $\e^{i\om t}f_\om(x)$ of
\eqref{EquationMink} corresponding to the ground state $f_\om(x)$, one can study their stability.
To start with, we must understand the symmetries of the equation \eqref{EquationMink}, connecting to the symmetries of
Minkowski space together
 with the structure of the equation \eqref{ellipeq}, namely the scaling property of the ground state
 \eqref{scalling}. These symmetries give an 8-parameter family of solitons. More precisely, denote
$\la=(\om, \theta,\xi, u)\in \La$ with
\begin{equation*}
 \La\equiv\{(\om, \theta, \xi, u)\in \mathbb{R}^8: |u|<1,\quad |\om|<m\}.
\end{equation*}
A subset of $\La$ is of particular importance
\begin{equation}
\label{defofLastab}
\La_{\textnormal{stab}}\equiv\{(\om, \theta, \xi, u)\in\La,\quad \frac{p-1}{6-2p}<\frac{\om^2}{m^2}<1\},\quad 1<p<\frac{7}{3},
\end{equation}
corresponding to the stable solitons. Define
\begin{align}
\label{z}
 & z(x;\la)=A_{u}(x-\xi)=\rho P_u(x-\xi)+(I-P_u)(x-\xi),\quad \rho=(1-|u|^2)^{-\f12},\\
\notag
& \Th(x;\la)=\th-\om u\cdot z(x;\la),
\end{align}
where $P_u:\mathbb{R}^3\rightarrow \mathbb{R}^3$ is the projection operator in the direction $u\in \mathbb{R}^3$, $I$ is
 the identity map. Let
\begin{equation*}
  \begin{split}
&\phi_S(x;\la)=\e^{i\Th(x;\la)}f_\om(z(x;\la)),\\
& \psi_S(x;\la)=\e^{i\Th(x;\la)}\left(i\rho\om
f_{\om}(z(x;\la))-\rho u\cdot \nabla_z f_\om(z(x;\la))\right).
\end{split}
\end{equation*}
Direct calculations show that $\phi_S(x;\la)$ solves \eqref{EquationMink} if the curve $\la(t)=(\om(t), \th(t), \xi(t), u(t))$
obeys the evolution equations
\begin{equation*}
\dot{\om}=0,\quad \dot{\th}=\frac{\om}{\rho}, \quad \dot{\xi}=u, \quad\dot{u}=0,.
\end{equation*}
Here we use the dot to denote the derivative with respect to $t$.
We remark here that the centers $(t, ut)$ of the solitons forms a straight line(geodesic)
in Minkowski space.

For $\la\in \La$, let
\begin{equation}
\label{defofVla}
V(\la)=(0, \frac{\om}{\rho}, u, 0).
\end{equation}
We find that $\psi_S(\la;x)=D_\la \phi_S(\la;x)\cdot V(\la)$(inner
product of vectors in $\mathbb{R}^8$) and the important identity
\begin{equation}
\label{idenofphiS}
 \Delta_x \phi_S-m^2\phi_S+|\phi_S|^{p-1}\phi_S-D_\la\psi_S \cdot V(\la)=0.
\end{equation}
Here the Laplacian operator $\Delta_x$ is for the variable $x$, used to distinguish the Laplacian $\Delta_z$ for the variable $z$
defined in \eqref{z}. For a $C^1$ curve
$\la(t)=(\om(t), \theta(t), \xi(t), u(t))$, let $\ga(t)=(\om(t), \pi(t), \eta(t), u(t))$ such that
\begin{equation}
\label{lagaV}
 \dot{\la}=\dot{\ga}+V(\la).
\end{equation}
We note that
\begin{align*}
 &\eta(t)=\xi(t)-\int_{0}^{t}u(s)ds,\quad \pi(t)=\th(t)-\int_{0}^{t}\frac{\om(s)}{\rho(s)}ds.
\end{align*}

Consider the Cauchy problem for the equation \eqref{EquationMink} in
Minkowski space with initial data $\phi_0(x)\in H^1(\mathbb{R}^3)$,
$\phi_1(x)\in L^2(\mathbb{R}^3)$. The stability result for solitons
is known. The following theorem is proven in \cite{moduStuart}.
\begin{thm}
\label{thmstbMin}
 Let $1<p<\frac{7}{3}$. For $\la_{0}\in\La_{\textnormal{stab}}$, there exists a small constant $\ep(\la_0)$ such that
if
\[
 \ep=\|\phi(0, x)-\phi_S(x;\la_0)\|_{H^1}+\|\pa_t\phi(0,
 x)-\psi_S(x;\la_0)\|_{L^2}<\ep(\la_0),
\]
then there exist a $C^1$ curve $\la(t)\in \La_{\textnormal{stab}}$
and a solution $\phi(t, x)\in H^1$ of equation
\eqref{EquationMink} satisfying
\[
 \|\phi(t, x)-\phi_S(x;\la(t))\|_{H^1}+\|\pa_t\phi(t, x)-\psi_S(x;\la(t))\|_{L^2}<C\ep
\]
and
\[
 |\pa_t \la(t)-V(\la(t))|<C\ep
\]
for some constant $C$ independent of $\ep$.
\end{thm}
Stability of solitons with $\la_{0}\in\La_{\textnormal{stab}}$ has
first been shown by J. Shatah in \cite{stableShatah} for radial
symmetric initial data and was later put into a very general
framework in \cite{stableShatah1}, \cite{stableShatah2}. Their
approach relies on the fact that the energy $E_\om(f_{\om})$ is
strictly convex in $\om$ if initially
 $\la_0\in\La_{\textnormal{stab}}$, see \cite{stableShatah}. This condition on $\la_0$ is sharp in the sense that the solitons
are unstable if the energy $E_\om(f_{\om})$ is concave in $\om$, see \cite{unstableshatah}, \cite{unstable2shatah}.
 Alternatively, the modulation approach, pioneered by M. Weinstein
\cite{ModulWeinstein} for showing the stability of solitons to
nonlinear Schr$\ddot{o}$dinger equations, leads to Theorem
\ref{thmstbMin} which additionally gives the behavior of the curve
$\la(t)$, see the work of D. Stuart \cite{moduStuart}.

\bigskip

We now briefly describe the modulation approach. Notice that
equation \eqref{EquationMink} locally has a unique solution
$(\phi(t, x), \pa_t\phi(t, x))\in C^1([0, t^*); H^1\times L^2)$.
Decompose the solution $\phi$ as follows
\begin{equation*}
 \begin{split}
&\phi(t, x)=\phi_S(x;\la(t))+\e^{i\Th(x;\la(t))}v(t, x),\\
&\pa_t\phi(t, x)=\psi_S(x;\la(t))+\e^{i\Th(x;\la(t))}w(t, x)
\end{split}
\end{equation*}
for a $C^1$ curve $\la(t)\in \La_{\textnormal{stab}}$ such that the following orthogonality condition hold
\begin{equation}
\label{orthcond}
 <\e^{-i\Th} D_\la \phi_S, w>_{dx}=<\e^{-i\Th}D_\la \psi_S, v>_{dx}, \quad \forall
 t\in\mathbb{R}.
\end{equation}
Here in this paper for complex valued functions $a(x)$, $b(x)$, $<a(x), b(x)>_{dx}$ is short for
\[
\f12\int_{\mathbb{R}^3}a\bar b+\bar a b \quad dx
\]
on $\mathbb{R}^3$ with measure $dx$.
Differentiating \eqref{orthcond} in $t$ and using the equation
\eqref{EquationMink}, we can obtain a coupled system of ODE's for $\la(t)$. To estimate
the curve $\la(t)$, we must control the radiation term
$(v, w)$. We first define two operators, $L_+$ and
$L_-$, appeared in the linearization in the real and
imaginary part of the solution to ~\eqref{EquationMink}. These
two operators act on functions of $z$ in $H^1(\mathbb{R}^3)$,
defined as follows
\begin{equation}
 \label{L+}
\begin{split}
& L_+=-\De_z+(m^2-\om^2)-pf_{\om}^{p-1}(z),\\
& L_{-}=-\De_z+(m^2-\om^2)-f_{\om}^{p-1}(z),
\end{split}
\end{equation}
which satisfy the following properties proven in \cite{ModulWeinstein}.
\begin{prop}
\label{propL}
We have
 \begin{itemize}
  \item [(a)]$L_-$ is a nonnegative self-adjoint operator in $L^2$ with null space $\ker L_-=span\{f_\om\}$;
  \item [(b)] $L_+$ is a self-adjoint operator in $L^2$ with null space $\ker L_+=span\{\nabla_{z^i}f_\om\}_{i=1}^{3}$. The strictly
negative eigenspace of $L_+$ is one dimensional.
 \end{itemize}
\end{prop}

\bigskip

 It can be shown from the nonlinear wave equation of $v(t,
x)$ that the corresponding energy is
\begin{equation}
 \label{defofE0}
E_0(t)=\|w+\rho u\cdot \nabla_z v-i\rho \om v\|_{L^2(dz)}^2+<v_1,
L_{+}v_1>_{dz}+<v_2, L_{-}v_2>_{dz},
\end{equation}
where $v, w$ are viewed as functions of $(t, z)$. Although the
operators $L_{+}$, $L_{-}$ are not positive definite, one still can show that $E_0(t)$
is equivalent to $\|v\|_{H^1}+\|w\|_{L^2}$ under the orthogonality condition
\eqref{orthcond}.
\begin{prop}
\label{positen} Assume $\la\in \La_{\textnormal{stab}}$. Assume $v,
w$ satisfy the orthogonality condition \eqref{orthcond}. Then there
is a positive constant $C(\om, u)$, depending continuously on $\om$,
$u$ such that
\[
C^{-1}(\om, u)(\|w\|_{L^2}^2+\|v\|_{H^1}^2)\leq E_0(t)\leq
C(\om,u)(\|w\|_{L^2}^2+\|v\|_{H^1}^2).
\]
\end{prop}
We will use this proposition for our later argument. The proof could
be found in \cite{moduStuart}, which is based on Proposition
\ref{propL}. The orthogonality condition $<\e^{-i\Th} D_\th \phi_S,
w>_{dx}=<\e^{-i\Th}D_\th \psi_S, v>_{dx}$ which is equivalent to
$<i\phi_S, w>_{dx}=<i\psi_S, v>_{dx}$ shows that the energy $E_0(t)$
is nonnegative. The proposition then follows by using a
contradiction argument, see the detailed proof in \cite{moduStuart}.
Once we have control of $\|v\|_{H^1}+\|w\|_{L^2}$, by analyzing the
ODEs for $\la(t)$, we can control the curve $\la(t)$ and obtain
estimates for the solution $\phi(t, x)$.

\bigskip

Having the definition for solitons on flat spacetimes, we generalize the definition for solitons on a Lorentzian manifold. In particular, we give the precise definitions for
$\phi_S^\ep(x;\la)$, $n\phi_S^\ep(x;\la)$, $ \La_{\textnormal{stab}}(t)$ which were used in the main theorem \ref{thmbaby}.
 We work under the coordinate system $x$ on the initial hypersurface $\Si_0=\mathbb{R}^3$. We
assume the particle enters the spacetime at point $\xi_0\in \Si_0$ and is approximated by a scaled
and translated stable soliton centered at $\xi_0$. We hence has to
define solitons on $(\mathcal{M}, h)$, $\mathcal{M}=[0, T]\times\mathbb{R}^3$. The idea is that at any point $P\in\mathcal{M}$, we simply define the solitons at $P$ to
be those on the flat space with metric $h(P)$. More precisely, for any foliations $[0,
T^{*}]\times\Si_\tau$ of $\mathcal{M}$ with coordinate system $(t,
x)$, where $\Si_\tau\subseteq\mathbb{R}^3$, let
\begin{align*}
 \La_{\textnormal{stab}}(t)=\left.\{(\om, \th, \xi,
u_h)\in
\mathbb{R}^8\right|\frac{p-1}{6-2p}<\frac{\om^2}{m^2}<1,\quad \pa_t
+u_h \pa_x \textnormal{ is timelike at }(t, \xi)\in[0, T^*]\times
\Si_\tau\}.
\end{align*}
We define
\begin{Def}
 \label{defofsolfix}
For a curve $\la(t)=(\om(t), \th(t), \xi(t),
u_h(t))\in\La_{\textnormal{stab}}(t)$, a soliton centered at
$\xi(t)$ in the direction $u_h(t)$ on the space $([0, T^*]\times
\Si_\tau, h)$  is defined as follows
\begin{equation*}
\begin{split}
\phi^\ep_S(x;\la(t))&=\e^{i\ep^{-1}\left(\th(t)-\rho\om(t) u\cdot
(x-\xi(t))Q\right)}f_{\om(t)}(\ep^{-1} A_uQ^T(x-\xi(t))^T),
\end{split}
\end{equation*}
where
\[
 u=a^{-1}(\a+u_h Q),  \quad \rho=(1-|u|^2)^{-\f12}, \quad a>0, \quad \a\in\mathbb{R}^3, \quad Q \textnormal{ is}
 \quad 3\times 3 \textnormal{ matrix}
\]
such that
\[
 h(t, \xi(t))=\left(
\begin{array}{cc}
 -a^2+\a\a^T& \a Q^T\\
Q \a^T & Q Q^{T}
\end{array}
\right)=\left(
\begin{array}{cc}
 a& \a\\
0 & Q
\end{array}
\right)m_0 \left(
\begin{array}{cc}
 a& 0\\
\a^T & Q^{T}
\end{array}
\right).
\]
We recall here that $A_u$ is $3\times 3$ matrix defined in \eqref{z}
and $m_0$ is the Minkowski metric. In particular, we can denote
\begin{align*}
&\psi^\ep_S(x;\la(t))=i\ep^{-1}\rho^{-1}\om a
\phi_S^\ep(x;\la(t))-u_h(t)\nabla_x
 \phi_S^\ep(x;\la(t))
\end{align*}
corresponding to $\pa_t \phi_S^\ep$ in the direction $u_h$ and
\begin{align*}
n\phi_S^\ep(x;\la(t))=i\ep^{-1}\om\rho^{-1}\phi_S^\ep-uQ^{-1}\cdot\nabla_x
\phi_S^\ep(x;\la(t))
\end{align*}
associated to $n\phi_S^\ep$, where $n$ is the unit normal to the
hypersurface $\Si_0$ embedded to $(\mathcal{M}, h)$ at time $t$.
\end{Def}

The above definition is well defined. We first show that $h(t,
\xi(t))$ has a decomposition as in the definition, that is, $a$,
$\a$, $Q$ exist. Notice that $h$ is Lorentzian metric. At point $(t,
\xi(t))$, the $3\times 3$ matrix $(h(t, \xi(t))_{kl})$ is symmetric
and positive definite. We hence can find a $3\times 3$ matrix $Q$
such that
\[
 QQ^T=\left((h(t, \xi(t)))_{kl}\right).
\]
Then $a, \a$ are uniquely determined as follows
\[
 \a=(h_{01}, h_{02}, h_{03})(t, \xi(t))(Q^T)^{-1},\quad a=\sqrt{-h(t,
 \xi(t))_{00}+\a\a^T}.
\]
Secondly, we prove that $|u|<1$. In fact, notice that the vector
$\pa_t+u_h\pa_x$ is timelike at $(t, \xi(t))$. We have
\[
 h(t, \xi(t))_{00}+2h(t, \xi(t))_{0k}u_h^k+u_h^k h(t,
 \xi(t))_{kl}u_h^l<0,
\]
which implies that
\[
 |u|^2=a^{-2}(\a+u_h Q)(\a+u_h Q)^T<1.
\]
Finally, we demonstrate that the definition is independent of the
choice of the $3\times 3$ matrix $Q$. Let $\tilde{u}$, $\tilde{a}$,
$\tilde{\a}$, $\tilde{Q}$ be another decomposition. Let
$P=Q^{-1}\tilde{Q}$. We conclude that $PP^T=I$ and
\[
 \tilde{\a}=\a P, \quad \tilde{a}=a,\quad \tilde{u}=uP.
\]
Hence we can show that
\begin{align*}
 &\tilde{u}\cdot (x-\xi(t))\tilde{Q}=uP \cdot (x-\xi(t)) QP=u\cdot
 (x-\xi(t))Q,\\
 &A_{\tilde u}\tilde{Q}^T=A_{uP}P^T Q^T=P^T A_u Q^T.
\end{align*}
Since $f_\om$ is spherical symmetric by Theorem \ref{propoff}, $PP^T=I$, we thus have shown that $\phi_S^\ep(x;\la(t))$,
$\psi_S^\ep(x;\la(t))$ are well defined.

\section{The main idea of the proof}
In this section, we briefly describe the main ideas for the proof of the main theorem \ref{thmbaby}. As having discussed in the introduction, our theorem generalizes those
results in \cite{einsteinStuart} in two ways. The first aspect that we can extend the solution to arbitrarily
large time $T$ is based on a result on the orbital stability of
solitons on small perturbations of Minkowski space. Note that the
dynamics of the particle $\phi^\ep$ are governed by the nonlinear
Klein-Gorden equation, see the Einstein-scalar field system
\eqref{EQUATION}. In local coordinate $(t, x)$, the corresponding
normalized scalar field $\phi$ solves the nonlinear wave equation
\eqref{fixbackKGeq} on the slowly varying background with metric
$g^\ep(t, x)=g(\ep t, \ep x)$. The assumption on the initial data
for the particle implies that the initial data for the normalized
scalar field $\phi$ are close to some stable solitons. The property
that the particle moves along a timelike geodesic can then be
reduced to the orbital stability of stable solitons
 along a timelike geodesic on a slowly varying background. In other words,
we need to show that the solution $\phi$ of the nonlinear wave \eqref{fixbackKGeq} is close to some soliton for all
$t$ and the center of the soliton propagates along a timelike geodesic.

The related problem of stability of solitons in Minkowski space has been discussed in the previous section.
Stable solitons have been proven to be orbitally stable for all time in \cite{stableShatah1}, \cite{stableShatah2}, \cite{stableShatah}. More precisely,
it was shown that if the initial data are close to
some stable soliton in $H^1$, then the solution to the nonlinear
Klein-Gordon equation exists for all time and is close to some
translated solitons (by the Lorentz transformation of Minkowski
space). However, these works do not characterize the dynamics of the
solitons. In particular, the centers of the solitons were not
explicitly constructed.

The first step along this direction was taken by M. Weinstein. In
\cite{ModulWeinstein}, he proved the orbital stability of solitons
to nonlinear Schr\"odinger equations and gave the additional
information on the position and speed of
 the solitons by using modulation theory. For the modulation theory,
 one decomposes the solution $\phi$ into the soliton
part and the remainder part $\phi=\phi_S+v$. The soliton part $\phi_S$
is unknown, depending on, e.g., its position and speed. Using the
decomposition, the equations for $\phi$ then lead to a linearized
equation for the remainder $v$. We remark here that the
linearized equation also depends on the unknown soliton $\phi_S$. We
choose the decomposition such that the remainder part $v$ is
orthogonal to the generalized null space of the linearized equation
at $\phi_S$. This orthogonality condition together with the
equations for $\phi$ leads to a coupled system of nonlinear
ODE's (modulation equations), which governs the position and speed of
the solitons.

Using this modulation approach, D. Stuart studied the stability of
solitons to nonlinear wave equations. In \cite{moduStuart}, he
proved the orbital stability of stable solitons and showed that
 the center of the soliton moves along a $C^1$ curve which is close to a straight line in Minkowski space.
Later in \cite{psedoStuart}, D. Stuart studied the stability of
stable solitons on small perturbations
of Minkowski space. More precisely, he considered the Cauchy problem for the nonlinear wave
\eqref{fixbackKGeq} with initial data which are close to stable solitons on a slowly varying background with metric $g^{\ep}(t, x)=g(\ep t,\ep x)$. He showed that stable
solitons are orbitally stable and move along timelike geodesics
 up to time $t^*/\ep$ for some small positive constant $t^*$.  Due to the scaling, when applying this result to the problem of geodesic hypothesis, one can only show that the particle travels along a timelike geodesic
in the short time interval $[0, t^*)$, see \cite{einsteinStuart}. Although $t^*$ is independent of $\ep$, it was required to be
sufficiently small.

A key improvement which allows us to obtain Theorem \ref{thmbaby} is
that we are able to extend Stuart's stability result up to time
$T/\ep$ for arbitrary large $T>0$. We show that if the initial data
are close to some stable solitons, then we can solve the nonlinear wave equation
\eqref{fixbackKGeq} up to time $T/\ep$ and demonstrate that the
solution is close to stable solitons centered along a timelike
geodesic. The time $T$ has to be fixed as we need
to require $\ep$ to be sufficiently small depending on $T$. However,
we no longer need the smallness of $T$ as it was in
\cite{psedoStuart}.

We will adapt the modulation approach to treat the
orbital stability of stable solitons on a fixed slowly varying
background. The new ingredient is that we can construct a coordinate system (Fermi coordinate system) such that the
Christoffel symbols vanish along the trajectory of the center of the soliton, which is a timelike geodesic uniquely determined by the
initial data.  The
motivation for choosing such a local coordinate system is to study
the equation for the remainder part $v$ by using the modulation
approach mentioned above. Recall that we decompose the solution
$\phi$ of the equation \eqref{fixbackKGeq} into soliton plus a
remainder $\phi=\phi_S+v$. The remainder $v$ is supposed to be
small. To control $v$, we need to estimate
$\Box_{g^\ep}\phi_S$, in particular the term
\[
\frac{1}{\sqrt{-\det g^\ep}}\pa_\mu\left((g^\ep)^{\mu\nu}\sqrt{-\det
g^\ep}\right)\pa_\nu\phi_S=-(g^\ep)^{\mu\ga}\Ga_{\mu\ga}^{\nu}\pa_\nu\phi_S,
\]
where $\Ga_{\mu\ga}^{\nu}$ is the Christoffel symbols for the metric
$g^\ep$. As the soliton $\phi_S$ is expected to be centered along a timelike
geodesic, to control the above term, a natural way is to choose a good local
coordinate system such that along that geodesic the Christoffel symbols vanish.
Furthermore, under such a coordinate system, the geodesic equations are linear. In
particular, the geodesic can be parameterized by $(t,u_0t)$ for some
constant vector $u_0\in\mathbb{R}^3$. This parametrization is
exactly the one for straight lines in Minkowski space. Now as the full derivative of
the metric components vanishes along the geodesic and the
metric is slowly varying, that is, $g^\ep(t, x)=g(\ep t, \ep x)$, we conclude that
near the geodesic, the metric $g^\ep$ is higher order (at least $\ep^2$)
perturbation of the flat metric $g^\ep(0, 0)$. Hence the errors contributed by
the soliton $\phi_S$ in the equation for the remainder $v$ will have size $\ep^2$.

Nevertheless, for the orbital stability in the energy space $H^1$ on a slowly varying background, we are not going to study the equation for the
remainder $v$ directly. Instead, as in \cite{moduStuart}, we decompose the almost conserved energies of the full solution $\phi$ around the
soliton $\phi_S$. By using Gronwall's inequality, we can prove the orbital stability up to time $T/\ep$. The key is that we have avoided
arguments based on bootstrap. And by doing so we can remove the smallness assumption on $t^*$ which was used to close the bootstrap assumption
as it was in \cite{psedoStuart}. For the higher order Sobolev estimates of the remainder $v$ which are needed to control the spacetime metric
$g$, we turn to rely on the equation of $v$.

We now briefly discuss the proof for the main theorem as well as the second aspect of our work that we can improve the amplitude $\delta$ of the
particle to be $\delta\leq \ep^q$, $q>1$ . We work under the Fermi coordinate system mentioned above. After scaling, it is equivalent to
consider the scaled coupled Einstein equations for $(g^\ep, \phi)$. Starting with the bootstrap assumption on the unknown spacetime metric
$g^\ep$
\begin{equation}
\label{boostrapass}
 \|\pa^s(g^\ep-h^\ep)\|_{L^2(\mathbb{R}^3)}(t)\leq 2\ep^2,\quad 1\leq s\leq 3,\quad t\leq T/\ep,
\end{equation}
we can show that the normalized scalar field $\phi$ is close to some solitons centered along the timelike geodesic $(t, u_0 t)$ in $H^3$. That is
\[
 \|\pa^s(\phi-\phi_S)\|_{L^2(\mathbb{R}^3)}(t)\leq C\ep,\quad \forall s\leq 3, \quad t\leq T/\ep
\]
for some constant $C$ independent of $\ep$, where $\phi_S$ are solitons centered along $(t, u_0 t)$. Here we must clarify that we have used bootstrap argument
in order to estimate the spacetime metric $g$. But as discussed above, we avoid using bootstrap argument when estimating the scalar field $\phi$. The $H^1$ estimates are
implied by the orbital stability of solitons. Then the higher Sobolev estimates follow by
analyzing the nonlinear wave equation of the remainder $\phi-\phi_S$.
These Sobolev estimates for the scalar field $\phi$ are used to control the
 the energy momentum tensor $\delta^2 T_{\mu\nu}(g^\ep, \phi;\mathcal{V}(\phi))$ (after scaling) so that we can estimate the spacetime
metric $g^\ep$ in order to close the above bootstrap assumption. Choosing the relatively
harmonic gauge condition \cite{GREE_bruhat}, \cite{larSca_hawking}, one can turn the Einstein equations into a hyperbolic system
for the components of the metric $g^\ep$. Detailed reduction is carried out in Section 5. Combining with the vacuum Einstein
 equations for the given metric $h^\ep$, we can roughly obtain estimates for $g^\ep-h^\ep$ as follows
\[
 \sup\limits_{t\leq T/\ep}\|\pa^s(g^\ep-h^\ep)\|_{L^2(\mathbb{R}^3)}(t)\leq C\delta^2\int_{0}^{T/\ep} \|T_{\mu\nu}\|_{H^s}(t)dt\leq
C\delta^2\ep^{-1}
\]
 by using energy estimates for hyperbolic equations. This requires $\delta\leq \ep^q$, $q>\frac{3}{2}$ in order to close the
bootstrap assumption. Such a lower bound on $q$ was suggested in \cite{einsteinStuart} and the
result there was proven under the even stronger condition $q\geq\frac{7}{4}$.

However, through another new observation, we are able to improve the
condition on $\delta$ to be $\delta\leq \ep^q$, $q>1$ or
$\delta=\ep_0 \ep$ for sufficiently small $\ep_0$ which is
independent of $\ep$. Notice that the tangent vector
$X=\pa_t+u_0\pa_x$ of the geodesic $(t, u_0 t)$ is timelike and can
be extended to a uniformly timelike vector field on the whole
spacetime. We have already shown that the scalar field $\phi$
decomposes into a soliton part $\phi_S$ and an error term. Hence the energy
momentum tensor $\delta^2 T_{\mu\nu}(g^\ep, \phi;\mathcal{V}(\phi))$
splits into the soliton part $\delta^2 T_{\mu\nu}^S$, which moves
along the timelike geodesic $(t, u_0 t)$ or quantitatively
 \[
  \delta^2|X\pa^s T_{\mu\nu}^S|\leq C\delta^2\ep,\quad \forall s\leq 3,
 \]
 and the
error part $\delta^2 T_{\mu\nu}^R$ which satisfies the estimates
\[
\delta^2\|\pa^s T_{\mu\nu}^R\|_{L^2(\mathbb{R}^3)}(t)\leq C\delta^2\ep,\quad s\leq 3, t\leq T/\ep.
 \]
To make use of the fact that the soliton part moves along the geodesic, when doing energy estimates, we multiply
the hyperbolic equations by $X(g^\ep-h^\ep)$. Integrating by parts, we can pass
 the derivative $X$ to the soliton part
 $T_{\mu\nu}^S$ of the energy momentum tensor $T_{\mu\nu}$. As the soliton decays exponentially, using Hardy's inequality,
we can show that
\[
\left|\int_{\mathbb{R}^3}T_{\mu\nu}^S \cdot X(g^\ep-h^\ep)dx\right|\leq
\|(1+|x|)XT_{\mu\nu}^S\|_{L^2}\|(1+|x|)^{-1}(g^\ep-h^\ep)\|_{L^2}\leq
C\ep\|\pa(g^\ep-h^\ep)\|_{L^2}.
\]
Since the error part is small, the above observation allows us to
improve the condition on $\delta$. In fact, since the vector fields
$X$, $\pa_t$ are uniformly timelike, the energy estimates for
hyperbolic equations show that
\[
 \|\pa(g^\ep-h^\ep)\|_{L^2}^2(t)\leq C\delta^4+C\ep \delta^2\int_{0}^{t}\|\pa(g^\ep-h^\ep)\|_{L^2}(s)ds.
\]
Here we have to require that initially $\|\pa(g^\ep-h^\ep)\|_{L^2}(0)\leq C\delta^2$. Applying Gronwall's inequality
or using bootstrap argument, we obtain
\[
 \|\pa(g^\ep-h^\ep)\|_{L^2}(t)\leq C\delta^2,\quad \forall t\leq T/\ep.
\]
Commuting the equations with $\pa^s$, we have the same estimates for
$\pa^s(g^\ep-h^\ep)$, $1\leq s\leq 3$. Thus to close the bootstrap
assumption \eqref{boostrapass} and hence to conclude our main
theorem, it suffices to require that $\delta\leq \ep^q$, $q>1$ or
$\delta=\ep_0 \ep$ for sufficiently small $\ep_0$ which is
independent of $\ep$. Our main theorem then follows from the
local existence result for Einstein equations in $H^3$, see
\cite{LocalEx_bruhat}, \cite{Uniqu_bruhat}.

Finally, we discuss the existence of the initial data $(\Si_0, \bar
g, \bar K, \phi_0^\ep, \phi_1^\ep)$. As we have mentioned previously, we require
the data to satisfy the constraint equations and the estimates
\[
\|\pa^s(g^\ep-h^\ep)\|_{L^2(\Si_0)}\leq C\delta^2, \quad \forall 1\leq s\leq 3.
\]
For given $(\phi_0^\ep, \phi_1^\ep)$, existence of initial data satisfying the constraint equations has been shown in
\cite{Imfunc_bruhat}, \cite{GREE_bruhat}. However, the above estimates do not follow
directly from previous works. We will revisit the existence of initial data by using the implicit function
 theorem combined with the approach
developed in \cite{Imfunc_bruhat}, \cite{Bruhat_christo},
\cite{GREE_bruhat}. To show that the data also satisfy the above estimates,
we rely on a Hardy type inequality for a first order linear operator with trivial
kernel in some weighted Sobolev space. We refer the reader to Lemma \ref{lemhardyineqV} in the last section for details.

\section{Construction of Fermi Coordinate System}
The geodesic hypothesis states that a test particle moves along a
timelike geodesic. This timelike geodesic $(t, \ga_0(t))$, using the
notations in the previous section, is determined by the initial
position $\xi_0$ and speed $u_h(0)$. As having discussed in the
introduction, our approach relies on the Fermi coordinate system
such that the Christoffel symbols are vanishing along the timelike
geodesic $(t, \ga_0(t))$. In this section, we give an explicit
construction of such local coordinate system.

\begin{lem}
 \label{propfermiCord}
Let $([0, T]\times\mathbb{R}^3, h)$ be a smooth Lorentzian spacetime
and $(t, \ga_0(t))$ be a timelike geodesic. Assume that the metric $h$
satisfies \eqref{condfix}. Then for all $\delta_1\in(0, \f12 T)$,
there exists a subspace $M$
\[
[0, T-2\delta_1]\times\mathbb{R}^3\subset M\subset [0,
T]\times\mathbb{R}^3
\]
with a new coordinate system  $(s, y)\in [0, C(\la'(0),
h)(T-\delta_1)]\times\mathbb{R}^3$ such that the given timelike
geodesic is represented as $(s, u_0s)$ for some constant vector
$u_0\in\mathbb{R}^3$, $|u_0|<1$. Moreover, along this geodesic, we
have
\[
 \Ga_{\mu\nu}^\a(s, u_0 s)=0,\quad h_{\mu\nu}(s, u_0s)=(m_0)_{\mu\nu},\quad \forall s\in [0, C(\la'(0),
 h)(T-\delta_0)],
\]
where $\Ga_{\mu\nu}^\a$ are the associated Christoffel symbols and
$m_0$ is the Minkowski metric, that is $(m_0)_{00}=-1$,
$(m_0)_{kk}=1$, $(m_0)_{\mu\nu}=0, \forall \mu\neq\nu$. Furthermore,
under the new coordinate system $(s, y )$, there exists a positive
constant $K(\delta_1, \ga_0'(0))$ depending on $\delta_1$, $\ga_0'(0)$,
$K_0$ such that
\[
h(X, X)\leq -K(\delta_1, \ga_0'(0)),\quad X=\pa_s+u_0^k\pa_{y_k}
\]
on the space $[0, C(\ga_0'(0), h)(T-\delta_1)]\times\mathbb{R}^3$. In
particular, the vector field $X=\pa_s+u_0^k\pa_{y_k}$ is timelike.
\end{lem}
\begin{proof} Let $(t, x)$ be
the given coordinate system on $([0, T]\times\mathbb{R}^3, h)$. Our
first step is to choose a coordinate system such that the given
geodesic is represented by $(t, 0)$. Let $(t, \bar x)$ be a new
coordinate system defined as follows
\[
(t, \bar x)=(t, x-\ga_0(t)).
\]
With this translation, we can show that the geodesic is
parameterized by $(t, 0)$. However, along this geodesic we only have
$\Ga_{00}^{\mu}(t, 0)=0$. For simplicity, let's still use $(t, x)$ to denote $(t, \bar x)$.

Next, we change the coordinates in a small neighborhood of the geodesic such that all
the Christoffel symbols are vanishing along the geodesic. To achieve this, we reproduce the
construction here, which is essentially the same as that in
\cite{NormalCord}, \cite{Fermi}. Let $\chi(x)$ be a cutoff function
\begin{equation*}
\begin{cases}
\chi(x)=1, \quad |x|\leq r_0,\\
\chi(x)=0,\quad |x|\geq 2 r_0
\end{cases}
\end{equation*}
for some positive number $r_0$. Introduce new coordinates defined as
follows
\[
\tilde{x}^\mu=x^\mu+a_{k}^\mu(t) x^k\chi(x)+\f12 b_{kl}^\mu(t)
x^kx^l\chi(x),
\]
where $(x^0, x^1, x^2, x^3)=(t, x^1, x^2, x^3)$. Here we recall that
Greek letters $k, l$ run from 1 to 3 while the Latin letters $\mu,
\nu$ run from 0 to 3. The geodesic in this new coordinate system is
still parameterized by $(t, 0)$. The transformation law for
Christoffel symbols
\[
\tilde{\Ga}_{\si\nu}^\mu=\frac{\pa x^\b}{\pa \tilde{x}^\si}\frac{\pa
x^\ga}{\pa\tilde{x}^\nu}\left(\Ga_{\b\ga}^{\a}\frac{\pa
\tilde{x}^\mu}{\pa x^\a}-\frac{\pa^2 \tilde{x}^\mu}{\pa x^\b \pa
x^\ga}\right)
\]
together with fact that $\Ga_{00}^\mu(t, 0)=0$, implies that
$\tilde{\Ga}_{\si\nu}^{\mu}$ vanishes along the geodesic $(t, 0)$ is
equivalent to the following ODEs for $a_{k}^\mu(t)$, $b_{kl}^\mu(t)$
\begin{equation*}
\begin{cases}
\frac{d}{dt}a_{k}^\mu=\Ga_{0 k}^\mu(t,
0)+a_{j}^\mu\Ga_{0k}^j(t, 0),\\
b_{kl}^\mu=\Ga_{kl}^\mu(t, 0)+a_{j}^\mu\Ga_{kl}^j(t,0).
\end{cases}
\end{equation*}
Now take $a_{k}^\mu(0)=b_{kl}^\mu(0)=0$ initially. We conclude that the
coefficients $a_{k}^\mu$, $b_{kl}^\mu$ are uniquely determined up to
any time $T$. Moreover, we have the estimates
\[
|a_{k}^\mu|+|b_{kl}^\mu|\leq C(h, T,\ga_0),
\]
where the constant $C(h, T,\ga_0)$ depends only on the metric $h$, the time $T$ and the geodesic $(t, \ga_0(t))$.
Therefore for any $\delta_1>0$, there exist $r_0$ sufficiently small
and the associated cutoff function $\chi$ such that
\[
\tilde x: [0, T-\delta_1]\times \mathbb{R}^3\rightarrow [0,
T]\times\mathbb{R}^3,\quad \tilde{x}^{-1}:[0, T-2\delta_1]\times
\mathbb{R}^3\rightarrow [0, T-\delta_1]\times\mathbb{R}^3
\]
are inclusions. Hence let $M=\tilde{x}([0,
T-\delta_1]\times\mathbb{R}^3)$ with the induced metric $h$. We have
\[
[0, T-2\delta_1]\times \mathbb{R}^3\subset M\subset [0, T]\times
\mathbb{R}^3
\]
with a coordinate system $\tilde x$ such that the Christoffel
symbols $\tilde{\Ga}_{\mu\nu}^{\b}$ vanish along the timelike
geodesic $(t, 0)$.

\bigskip

We still need to change the coordinate such that $h=m_0$ along
the geodesic. We have shown that under the coordinate system
$\tilde{x}=(t, x')$, $\tilde{\Ga}_{\mu\nu}^{\b}(t, 0)=0$, which
imply that $\pa h(t, 0)=0$, $h(t, 0)=h(0, 0)$. Since the geodesic is
timelike, we have $h_{00}(0, 0)<0$ and $(h(0, 0))_{kl}$ is positive
definite. Assume
\[
 h(0, 0)=\left(
              \begin{array}{cc}
                -a & \a \\
                \a^T & P \\
              \end{array}
            \right),\quad a>0.
\]
Let $Q_{3\times3}$ be such that $QPQ^T=I_{3\times3}$. Consider the
following new coordinate system
\[
 (s, y)=(t, x')
 \left(
\begin{array}{cc}
1&\a P^{-1}\\
0&I\\
\end{array}
\right)\left(
\begin{array}{cc}
\sqrt{a+\a P^{-1}\a^T}&0\\
0&Q^{-1}\\
\end{array}
\right)=(t\sqrt{a+\a P^{-1}\a^T}, (t\a P^{-1}+x')Q^{-1}).
\]
We can show that under this new coordinate system, the given
timelike geodesic is parameterized by $(s, u_0 s)$. Moreover, $h(s,
u_0 s)=m_0$ for some constant vector $u_0\in \mathbb{R}^3$ given as
follows
\[
 u_0=\frac{\a P^{-1}Q^{-1}}{\sqrt{a+\a
 P^{-1}\a^T}}\in\mathbb{R}^3,\quad |u_0|<1.
\]
By the previous construction, we have $(s, y)\in[0, (T-\delta_1)C(\ga_0'(0), h)]\times
\mathbb{R}^3$, where we denote $C(\ga_0'(0), h)=\sqrt{a+\a
P^{-1}\a^T}$ depending only on $\ga_0'(0)$ and the metric $h$.

Finally, notice that under the coordinate system $\tilde{x}=(t,
x')$, the vector field $\pa_t$ is timelike, that is, by
\eqref{condfix} and the construction of $\tilde{x}$
\[
h(\pa_t,\pa_t)\leq -K(\delta_1, \ga_0'(0))
\]
for some constant depending on $h$, $\delta_1$, $\ga_0'(0)$. We remark
here that the positive constant $K(\delta_1, \ga_0'(0))$ also depends
on the cutoff function $\chi$. However, when $\delta_1$ is fixed,
the cutoff function is also fixed. Therefore the corresponding
vector field $X=\pa_s+u_0^k\pa_{y_k}$ on the space $(M, h)$ with
coordinate system $(s, y)\in[0, (T-\delta_1)\sqrt{a+\a
P^{-1}\a^T}]\times \mathbb{R}^3$ is also timelike and satisfies
\[
h(X, X)\leq -K(\delta_1, \ga_0'(0)).
\]
\end{proof}

\section{Reduced Einstein Equations}
In this section, we reduce the scaled Einstein equations to a
hyperbolic system for the components of the metric $g^\ep_{\mu\nu}$
by fixing a relatively harmonic gauge condition, see
\cite{GREE_bruhat}, \cite{larSca_hawking}.

 The Einstein equations are independent of the choice of
local coordinate system. To solve the Einstein equations, we work
under the Fermi coordinate system constructed in the previous
section. More precisely, starting with the given vacuum spacetime
$([0, T]\times \mathbb{R}^3, h)$ with local coordinate system $(t,
x)$, Lemma \ref{propfermiCord} shows that we can find a new
coordinate system $(s, y)\in[0, T^*]\times\mathbb{R}^3$ on $M$,
where $[0, T-2\delta_1]\times\mathbb{R}^3\subset M\subset [0,
T]\times\mathbb{R}^3$, such that the timelike geodesic is
represented as $(s, u_0 s)$ for some constant vector field
$u_0\in\mathbb{R}^3$, $|u_0|<1$ and
\[
 h(s, u_0 s)=m_0,\quad \Ga_{\mu\nu}^{\a}(s, u_0 s)=0.
\]
Moreover, the vector field
\[
 X=\pa_s+u_0^k\pa_{y_k}
\]
is uniformly timelike. To simplify the notations, we work on the space
$([0, T]\times\mathbb{R}^3, h)$ for arbitrary $T>0$ with Fermi coordinate
system $(s, y)$.

With this local coordinate system $(s, y)$, to understand the
particle $\phi^\ep$, which has small size $\ep$ and small energy
$\delta$, we instead consider the scaled Einstein equations. More
precisely, recall that the potential $\mathcal{V}_{\ep, \delta}$
satisfies the scaling $\mathcal{V}_{\ep,
\delta}(\phi)=\delta^2\ep^{-2}\mathcal{V}(\delta^{-1}\phi)$. We
infer that
 if $(g(s, y), \phi^\ep(s, y))$ solves system
\eqref{EQUATION} on the space $[0, T]\times\mathbb{R}^3$, then
\begin{equation*}
g^\ep(s, y)=g(\ep s, \ep y),\quad \phi(s, y)=\delta^{-1}\phi^\ep(\ep s,\ep y)
\end{equation*}
satisfy the rescaled Einstein equations
\begin{equation}
\label{EQUATIONSC}
\begin{cases}
R_{\mu\nu}(g^\ep)-\f12 R(g^\ep) (g^\ep)_{\mu\nu}=\delta^2 T_{\mu\nu}(g^\ep, \phi;\mathcal{V}(\phi)),\\
\Box_{g^\ep}\phi-\mathcal{V}'(\phi)=0
\end{cases}
\end{equation}
on the rescaled space $[0, T/\ep]\times\mathbb{R}^3$. Conversely, a
solution $(g^\ep, \phi)$ of \eqref{EQUATIONSC} on the space $[0,
T/\ep]\times\mathbb{R}^3$ gives a solution $(g, \phi^\ep)$ of
\eqref{EQUATION} on the space $[0, T]\times\mathbb{R}^3$. It hence
suffices to consider the scaled Einstein equations
\eqref{EQUATIONSC} on the space $[0, T/\ep]\times\mathbb{R}^3$ with
initial data determined by $(\bar g, \bar K, \phi^\ep_0,
\phi_1^\ep)$ as well as the coordinate change from $(t, x)$ to the
Fermi coordinate system constructed in the previous section.
\begin{remark}
 The scaling heavily relies on the Fermi local coordinate system. However, the system \eqref{EQUATIONSC}
  itself is independent of choice of local
coordinate system on the space $[0, T/\ep]\times\mathbb{R}^3$.
\end{remark}

We must determine the initial data for the corresponding
Cauchy problem for the rescaled Einstein equations
\eqref{EQUATIONSC}. By the construction of the Fermi coordinate system in
Lemma \ref{propfermiCord} and by Definition \ref{defofsolfix}, we
define
\[
\phi_0(y)=\delta^{-1}\phi_0^\ep(\ep y),\quad
\phi_1(y)=\delta^{-1}\ep \phi_1^\ep(\ep y), \quad \bar
{g}^\ep(y)=\bar g(\ep y),\quad \bar{K}^\ep(y)=\ep \bar {K}(\ep y)
\]
on the initial hypersurface $\mathbb{R}^3$ with coordinate system $\{y\}$. Denote
\[
 \la_0^n=(\om_0, \ep^{-1}\th_0, 0, u_0),\quad
 n\phi_S(y;\la_0^n)=i\rho \om \phi_S(y;\la_0^n)-u_0
 Q^{-1}\nabla_y\phi_S(y;\la_0^n),
\]
where $\la_0=(\om_0, \th_0, \xi_0, u_h(0))$ is given in Theorem
\ref{thmbaby} and $Q$ is the $3\times 3$ matrix in Definition
\ref{defofsolfix}. By \eqref{constrainteq}, we have
\begin{equation}
\label{constraineqsc}
\begin{cases}
 \bar R(\bar {g}^\ep)-|\bar K^\ep|^2+(tr \bar K^\ep)^2=\delta^2(|\phi_1|^2+|\bar\nabla \phi_0|^2+2\mathcal{V}(\phi_0)),\\
 \bar{\nabla}^j \bar K^\ep_{ij}-\bar{\nabla}_i tr\bar
K^\ep=\delta^2<\phi_1, \bar\nabla_i\phi_0>,
\end{cases}
\end{equation}
where the covariant derivative is with respect to the metric $\bar
g^\ep$. Using this scaling, the condition \eqref{initialcond} becomes
\begin{equation}
\label{inicondgphi}
\begin{split}
 &\|\phi_0(y)-\phi_S(y;\la_0^n)\|_{H^3}+\|\phi_1(y)-
 n\phi_S(y;\la_0^n)\|_{H^2}\leq C_0\ep,\\
&\|\nabla(\bar g^\ep-\bar h^\ep)\|_{H^2}+\|\bar K^\ep-\bar
k^\ep\|_{H^2}\leq C_0\delta^2,
\end{split}
\end{equation}
where $\bar {h}^\ep(y)=\bar {h}(\ep y)$, $\bar k^\ep(y)=\bar k(\ep
y)$, $\nabla $ is the covariant derivative for the metric $\bar h^\ep$. We hence obtain the initial data $(\Si_0,
\bar g^\ep, \bar K^\ep, \phi_0, \phi_1)$ for rescaled Einstein equations
\eqref{EQUATIONSC}. We also remark here that the
existence of the initial data $(\Si_0, \bar{g}, \bar K, \phi^\ep_0,
\phi^\ep_1)$ in our main theorem \ref{thmbaby} is equivalent to the existence
of $(\Si_0, \bar{g}^\ep, \bar K^\ep, \phi_0, \phi_1)$ satisfying the
conditions \eqref{constraineqsc}, \eqref{inicondgphi} under the
Fermi coordinate system. Existence of such initial data set under
certain conditions is discussed in details in the last section.

\bigskip

The Cauchy problem for system \eqref{EQUATIONSC} is
underdetermined as any diffeomorphism of the
spacetime $([0, T/\ep]\times\mathbb{R}^3, g^\ep, \phi)$ satisfying
\eqref{EQUATIONSC} leads to another solution. Such freedom can be
removed by choosing a gauge condition such that system
\eqref{EQUATIONSC} is equivalent to a hyperbolic system for the
components of the metric $g^\ep$. Under the Fermi coordinate system
on the space $([0, T/\ep]\times\mathbb{R}^3, h^\ep)$, we define
\begin{equation*}
 \mathcal{G}^\la(g^\ep,
h^\ep)=(g^\ep)^{\mu\nu}(\Ga_{\mu\nu}^\la-\hat{\Ga}_{\mu\nu}^\la),\quad
\la=0, 1, 2, 3,
\end{equation*}
where $\Ga_{\mu\nu}^{\la}$, $\hat{\Ga}_{\mu\nu}^\la$ are Christoffel
symbols for the unknown metric $g^\ep$ and the given vacuum metric
$h^\ep$ respectively. The gauge condition that we choose is
\begin{equation}
 \label{gaugecond}
 \mathcal{G}^\la(g^\ep,
h^\ep)=0,\quad \forall \la =0, 1, 2, 3,
\end{equation}
which will be called the relatively harmonic gauge condition, see \cite{GREE_bruhat},
\cite{larSca_hawking}.  Instead of considering the full Einstein equations \eqref{EQUATIONSC}, we consider the following
reduced Einstein equations
\begin{equation}
\label{EEQred}
R_{\mu\nu}(g^\ep)-\frac{1}{2}R(g^\ep)g^\ep_{\mu\nu}-\f12\left(g^\ep_{\mu\la}\nabla_\nu
\mathcal{G}^\la+g^\ep_{\nu\la}\nabla_\mu\mathcal{G}^\la-g^\ep_{\mu\nu}\nabla_\la
\mathcal{G}^\la \right)=\delta^2 T_{\mu\nu}(g^\ep, \phi;
\mathcal{V}(\phi)),
\end{equation}
where the covariant derivative $\nabla$ is for the metric $g^\ep$ on
the space $[0, T/\ep]\times\mathbb{R}^3$. Then a solution of the
full Einstein equations \eqref{EQUATIONSC} can be constructed as
follows: we write down \eqref{EEQred} under the Fermi coordinate
system and check that it is hyperbolic for the components of the
metric $g^\ep$. We then construct initial data $(g^\ep(0)$, $\pa_t
g^\ep(0))$ from the given data $(\bar g^\ep, \bar K^\ep)$ such that
the gauge condition \eqref{gaugecond} holds initially. Hence we can
get a unique short time solution $([0, t^*)\times\mathbb{R}^3,
g^\ep, \phi)$ of the reduced Einstein equations \eqref{EEQred}
coupled with the matter field equation for $\phi$(second equation in
\eqref{EQUATIONSC}). We then argue that the gauge condition is
propagated, i.e., the relation \eqref{gaugecond} holds on $[0,
t^*)\times\mathbb{R}^3$. This implies that the solution $([0,
t^*)\times\mathbb{R}^3, g^\ep, \phi)$ to the reduced Einstein
equations is, in fact, a solution of the full Einstein equations
\eqref{EQUATIONSC}.

\bigskip

Our first step is to check the hyperbolicity of the reduced Einstein
equations \eqref{EEQred}. We notice that under the fixed Fermi
coordinate system(can be viewed as a local coordinate system), the
full Einstein equations can be written as
\begin{equation*}
-(g^\ep)^{\a\b}\pa_{\a\b}g^\ep_{\mu\nu}+
g^\ep_{\nu\la}\pa_{\mu}\left((g^\ep)^{\a\b}\Ga_{\a\b}^{\la}\right)+g^\ep_{\mu\la}\pa_{\nu}\left((g^\ep)^{\a\b}
\Ga_{\a\b}^{\la}\right)+Q_{\mu\nu}(g^\ep)=\delta^2(2T_{\mu\nu}-tr
T\cdot g^\ep_{\mu\nu}),
\end{equation*}
where
\begin{equation*}
Q_{\mu\nu}(g^\ep)=C_{\rho\si\la\mu\nu}^{\a\b\ga}(g^\ep)\pa_{\a}g^\ep_{\b\ga}\pa_{\rho}g^\ep_{\si\la}
\end{equation*}
and $C(g^\ep)$ denotes polynomials of $g^\ep_{\a\b}$,
$(g^\ep)^{\a\b}$. Taking trace of the reduced Einstein equations
\eqref{EEQred} , we get
\[
-R(g^\ep)+\nabla_\mu \mathcal{G}^\mu=tr T.
\]
Plug this into \eqref{EEQred}. We have the reduced Einstein
equations
\begin{equation*}
-(g^\ep)^{\a\b}\pa_{\a\b}g^\ep_{\mu\nu}+
g^\ep_{\nu\la}\pa_{\mu}\left((g^\ep)^{\a\b}\hat{\Ga}_{\a\b}^{\la}\right)+g^\ep_{\mu\la}\pa_{\nu}\left((g^\ep)^{\a\b}
\hat{\Ga}_{\a\b}^{\la}\right)+Q_{\mu\nu}(g^\ep)-\mathcal{G}^\la
\pa_\la g^\ep_{\mu\nu}=\delta^2(2T_{\mu\nu}-tr T\cdot
g^\ep_{\mu\nu}).
\end{equation*}
Similarly, we recall the vacuum Einstein equations for $h^\ep$ under
the Fermi local coordinate system
\begin{equation*}
-(h^\ep)^{\a\b}\pa_{\a\b}h^\ep_{\mu\nu}+
h^\ep_{\nu\la}\pa_{\mu}\left((h^\ep)^{\mu\nu}\hat{\Ga}_{\mu\nu}^\la\right)+h^\ep_{\mu\la}\pa_{\nu}\left((h^\ep)^{\mu\nu}
\hat{\Ga}_{\mu\nu}^\la\right)+Q_{\mu\nu}(h^\ep)=0.
\end{equation*}
Subtract the above two equations. We can show that the reduced Einstein equations coupled
with the matter field equation for $\phi$ are equivalent to the
following hyperbolic system for the difference
$\psi^\ep=g^\ep-h^\ep$
\begin{equation}
\label{EQUATIONRED}
\begin{cases}
-(g^\ep)^{\a\b}\pa_{\a\b}\psi^\ep_{\mu\nu}+\delta P_{\mu\nu}+\delta Z_{\mu\nu}+\delta Q_{\mu\nu}
=\delta^2(2T_{\mu\nu}-
tr T\cdot g^\ep_{\mu\nu}),\\
\Box_{g^\ep}\phi-\mathcal{V}'(\phi)=0,
\end{cases}
\end{equation}
where
\begin{align}
\notag
 &\delta Q_{\mu\nu}=Q(g^\ep)-Q(h^\ep)-(g^\ep)^{\a\b}\Ga_{\a\b}^{\la}\pa_{\la}g^\ep_{\mu\nu}+
(h^\ep)^{\a\b}\hat{\Ga}_{\a\b}^{\la}\pa_{\la}h^\ep_{\mu\nu},\\
\label{defofPQ}
 &\delta Z_{\mu\nu} =\left(g^\ep_{\mu\la}(g^\ep)^{\a\b}-h^\ep_{\mu\la}(h^\ep)^{\a\b}\right)\pa_\nu \hat{\Ga}_{\a\b}^\la+
\left(g^\ep_{\nu\la}(g^\ep)^{\a\b}-h^\ep_{\nu\la}(h^\ep)^{\a\b}\right)\pa_\mu \hat{\Ga}_{\a\b}^\la
-(\psi^\ep)^{\a\b}\pa_{\a\b}h^\ep_{\mu\nu},\\
\notag & \delta
P_{\mu\nu}=\hat{\Ga}_{\a\b}^\la\left((g^\ep)^{\a\b}\pa_\la
g^\ep_{\mu\nu}+g^\ep_{\nu\la}\pa_\mu (g^\ep)^{\a\b}+
g^\ep_{\mu\la}\pa_\nu (g^\ep)^{\a\b}-(h^\ep)^{\a\b}\pa_\la
h^\ep_{\mu\nu}-h^\ep_{\nu\la}\pa_\mu (h^\ep)^{\a\b}-
h^\ep_{\mu\la}\pa_\nu (h^\ep)^{\a\b}\right).
\end{align}
Hence we have shown that the reduced Einstein equations \eqref{EEQred} are hyperbolic.

Secondly, we demonstrate that the gauge condition \eqref{gaugecond}
propagates. Suppose $(g^\ep, \phi)$ satisfies system
\eqref{EQUATIONRED} on $[0, t^*)\times\mathbb{R}^3$ for some small
positive time $t^*$. Take divergence of both sides of the reduced
Einstein equations \eqref{EEQred}. Using Bianchi's identity and the
fact that the energy momentum tensor $T_{\mu\nu}$ is divergence
free(due to the matter field equation of $\phi$), we obtain the
evolution equations for $\mathcal{G}^\la$
\[
 \nabla^\mu\nabla_\mu \mathcal{G}_\nu+[\nabla_\mu,
 \nabla_\nu]\mathcal{G}^\mu=0,\quad \nu=0, 1, 2, 3,
\]
where the commutator $[\nabla_\mu,
\nabla_\nu]=\nabla_\mu\nabla_\nu-\nabla_\nu\nabla_\mu$. Hence
$\mathcal{G}^{\la}$ satisfies the above wave equations on $[0,
t^*)\times\mathbb{R}^3$. Thus $\mathcal{G}^\la$ vanishes on $[0,
t^*)\times\mathbb{R}^3$ if $\mathcal{G}^\mu$, $\pa_t
\mathcal{G}^\mu$ vanish initially. We can make $\mathcal{G}^\mu$
vanish on the initial hypersurface $\Si_0$ by choosing initial data
for $g^\ep_{0\mu}, \pa_t g^\ep_{0\mu}$(recall that only
$\bar{g}^\ep$, $\bar K^\ep$ or $g^\ep_{ij}$, $\pa_t g^\ep_{ij}$ are
given). Once we have $\mathcal{G}^\mu=0$ on $\Si_0$, we can show
that $\pa_t \mathcal{G}^\mu$ also vanishes on $\Si_0$ due to the
constraint equations. In fact, since $(g^\ep, \phi)$ solves
\eqref{EEQred}, the constraint equations \eqref{constraineqsc} for
$\bar {g}^\ep$, $\bar K^\ep$ together with the vacuum constraint
equations for $\bar h^\ep$, $\bar k^\ep$ imply that
\[
 \nabla_0
 \mathcal{G}_\nu+\nabla_\nu\mathcal{G}_0-g^\ep_{0\nu}\nabla_\mu
 \mathcal{G}^\mu=0,\quad \nu=0, 1, 2, 3,
\]
where the covariant derivative $\nabla$ is for the metric
 $g^\ep$. Hence $\pa_t\mathcal{G}^\mu=0$ on $\Si_0$ if $\mathcal{G}^\mu=0$
 initially. Therefore, we have shown that as long as \eqref{gaugecond} holds on $\Si_0$,
 a solution $(\psi^\ep, \phi)$ of the reduced Einstein equations
 \eqref{EQUATIONRED} on $[0, t^*)\times\mathbb{R}^3$ leads to a solution
 $([0, t^*)\times\mathbb{R}^3, \psi^\ep+h^\ep, \phi)$ to the full Einstein equations \eqref{EQUATIONSC}.
Since the solution of the Einstein equations \eqref{EQUATIONSC}
exists locally and is unique up to diffeomorphism
\cite{LocalEx_bruhat}, \cite{Uniqu_bruhat}, it
suffices to consider the reduced Einstein equations
\eqref{EQUATIONRED} with initial data $(g^\ep(0, y), \pa_t g^\ep(0,
y))$ such that \eqref{gaugecond} holds on the initial hypersurface
$\Si_0$.

It remains to construct initial data $(g^\ep(0)$, $\pa_t g^\ep(0))$ from $(\bar g^\ep, \bar K^\ep)$ on $\Si_0$ such that the gauge condition
\eqref{gaugecond} holds initially. Let
\[
(g^\ep)^{00}(0)=-\bar {N}^{-2},\quad g^\ep_{0i}(0)=\b_i,\quad i=1,
2, 3,
\]
where $\bar N$ is the lapse function and $\b$ is the shift vector on
$\Si_0$. Since the Riemannian metric $\bar {g}^\ep$ is the
Lorentzian metric $g^\ep$ restricted to $\Si_0$ and $\bar K^\ep$ is
the second fundamental form, we can show that
\[
g^\ep_{ij}(0)=\bar g^\ep_{ij},\quad \pa_t g^\ep_{ij}(0)=2\bar N \bar
K^\ep_{ij}+\bar{\nabla}_i \b_j+\bar{\nabla}_{j}\b_i,
\]
where $\bar{\nabla}$ is the covariant derivative for $\bar g^\ep$ on
$\Si_0$. That is $(g^\ep(0), \pa_t g^\ep_{ij}(0))$ is uniquely determined by $(\bar g^\ep, \bar K^\ep)$, $\bar N$, $\b_i$.
We claim that $\pa_t g^\ep_{0\mu}(0)$ is given by the gauge condition \eqref{gaugecond}. In fact, the gauge condition implies that
\begin{align*}
 (2-(m_0)_0^\mu)\pa_t\psi^\ep_{0\mu}=&2\bar {N} ^2(2(m_0)_k^\mu (g^\ep)^{0l}-(g^\ep)^{kl}(m_0)_0^\mu)\pa_t\psi^\ep_{kl}+
\bar{N}^2(2(m_0)_0^\mu
(g^\ep)^{k\b}-(g^{\ep})^{\a\b}(m_0)_k^\mu)\pa_k
\psi^\ep_{\a\b}\\&+\bar {N}^2 g^\ep_{\la\mu}(g^\ep)^{\a\b}
(\psi^\ep)^{\la \nu}(2\pa_\a h^\ep_{\nu \b}-\pa_\nu
h^\ep_{\a\b}),\quad \psi^\ep=g^\ep-h^\ep.
\end{align*}
We choose $\pa_t g^\ep_{0\mu}$ on $\Si_0$ as above. Notice that
different lapse functions $\bar N$ and shift vectors $\b_i$ will
lead to the same solution of the full Einstein equation
\eqref{EQUATIONSC} up to a change of local coordinate system. To
better estimate the difference $\psi^\ep$, we simply set
\[
 (g^\ep)^{00}(0, y)=(h^\ep)^{00}(0, y),\quad g^\ep_{0i}(0, y)=h^\ep_{0i}(0,
 y).
\]
The above discussion shows that we have constructed the initial data $(g^\ep(0), \pa_t g^\ep(0))$ such that the gauge condition
\eqref{gaugecond} holds initially.

Finally, we show that $g^\ep$ is sufficiently close to $h^\ep(0)$ and hence is Lorentzian initially.
Notice that
\[
 \pa_t h^\ep_{ij}=2\bar N
(\bar k^\ep_0)_{ij}+\nabla_i \b_j+\nabla_{j}\b_i,
\]
where the covariant derivative is for the metric $\bar {h}^\ep$. By
\eqref{condfix}, \eqref{inicondgphi}, we can show that
\begin{equation}
 \label{inicondgf}
\|\nabla({g}^\ep-{h}^\ep)\|_{H^2(\Si_0)}+\|\pa_t(g^\ep-h^\ep)\|_{H^2(\Si_0)}\leq
C(h, C_0)\delta^2,
\end{equation}
where the constant $C(h, C_0)$ depends on $h$, $C_0$. Thus if $\ep$
is sufficiently small, $g^\ep$ is a Lorentzian metric initially. We
summarize what we have obtained in this section.
\begin{lem}
 \label{lemredeq}
Let $([0, T]\times\mathbb{R}^3, h)$ be a vacuum spacetime. The
Cauchy problem for \eqref{EQUATION} with initial data
$(\Si_0, \bar{g}, \bar K, \phi_0^\ep, \phi_1^\ep)$ is equivalent(up
to a diffeomorphism and scaling) to the hyperbolic system
\eqref{EQUATIONRED}, \eqref{defofPQ} with initial data
 $(\phi_0(0, y), \phi_1(0, y), g^\ep(0, y)$,
$\pa_t g^\ep(0, y))$ satisfying \eqref{inicondgphi} and
\eqref{inicondgf}.
\end{lem}

\section{Stability of Stable Solitons on a Fixed Background}
We have shown in the previous section that to solve the full
Einstein equations \eqref{EQUATION}, it suffices to consider the
hyperbolic system \eqref{EQUATIONRED} on the scaled space $[0, T/\ep]\times\mathbb{R}^3$
with initial data $(\phi_0(0, x), \phi_1(0, x), g^\ep(0, x)$, $\pa_t g^\ep(0, x))$ satisfying \eqref{inicondgphi} and
\eqref{inicondgf}. The standard local
existence results imply that there is a unique short time solution $(g^\ep, \phi)$ on $[0, t^*]\times\mathbb{R}^3$.
Using continuity argument, the solution can be extended to $T/\ep$ as long as $(g^\ep, \phi)$ satisfies condition
\eqref{inicondgphi}, \eqref{inicondgf} for a constant $C$ independent of $\ep$. The estimate for $g^\ep-h^\ep$ follows from
 energy estimates for hyperbolic equations if the source term $T_{\mu\nu}-\f12 tr T\cdot g^\ep_{\mu\nu}$ lies
 in $C([0, T/\ep]; H^3)$. Since it is expected that the unknown metric $g^\ep$ is close to the given vacuum metric
$h^\ep$, the main difficulty for proving the main theorem is to show that $\phi$ exists and
is close to some translated soliton up to time $T/\ep$, that is, orbital stability of stable solitons along a timelike
geodesic on the slowly varying background $([0, T/\ep]\times\mathbb{R}^3, g^\ep)$.

\bigskip

In this section, we will let $([0, T]\times\mathbb{R}^3, h)$ be a Lorentzian spacetime with the Fermi coordinate system $(t, x)$ constructed
 in Lemma \ref{propfermiCord}. More precisely, assume that $(t, u_0 t)$, $t\in[0, T]$ is a timelike geodesic, where the constant
vector $u_0\in\mathbb{R}^3$, $|u_0|<1$. Along this geodesic, we have
\[
 h(t, u_0 t)=m_0,\quad \Ga_{\mu\nu}^{\a}(t, u_0 t)=0,
\]
where $\Ga_{\mu\nu}^{\a}$ is the Christoffel symbols for $h$.
Moreover, the vector field
\[
 X=\pa_t+u_0^k\pa_{k}
\]
is uniformly timelike on $([0, T]\times\mathbb{R}^3, h)$. We assume the metric $h\in C^4([0, T]\times\mathbb{R}^3)$ and
satisfies \eqref{condfix} for some constant $K_0$. Let $g$ be another Lorentzian metric on $[0, T]\times \mathbb{R}^3$.
Consider the Cauchy problem for the rescaled nonlinear wave equation
\begin{equation}
 \label{FEQUATION}
\begin{cases}
\Box_{g^\ep} \phi-m^2 \phi+|\phi|^{p-1}\phi=0,\\
\phi(0, x)=\phi_0(x),\quad \pa_t\phi(0, x)=\phi_1(x)
\end{cases}
\end{equation}
on the rescaled space $[0, T/\ep]\times \mathbb{R}^3$, where the slowly varying metric $g^\ep(t, x)=g(\ep t, \ep x)$.
We show the orbital stability of stable solitons to \eqref{FEQUATION} if $g-h$ vanishes as $|x|\rightarrow \infty$ and satisfies
\begin{equation}
\label{bapsi}
 \|\pa^{s+1}\psi^\ep(t, \cdot)\|_{L^2(\mathbb{R}^3)}\leq 2\ep^2, \quad |s|\leq 2, \quad \forall t\in [0, T/\ep],
\end{equation}
where $\psi=g-h,\quad \psi^\ep=g^\ep-h^\ep$. This condition can be
viewed as a bootstrap assumption.

Our main result in this section is
\begin{thm}
\label{thmfix}
 Let $g, h$ be Lorentzian metrics satisfying \eqref{condfix}, \eqref{bapsi} on the space $[0,
T]\times \mathbb{R}^3$ with Fermi coordinate system $(t, x)$. Assume $(t, u_0 t)$ is a timelike geodesic for the metric
 $h$ for a vector $u_0\in\mathbb{R}^3$, $|u_0|<1$. Assume $2\leq p<\frac{7}{3}$. Then for all $\la_0=(\om_0, \th_0, 0,
u_0)\in\La_{\textnormal{stab}}$, there exists a positive number
$\ep^*$ depending on $h$, $\la_0$ such that for all positive
$\ep<\ep^*$, if the initial data $\phi_0$, $\phi_1$ satisfy
\begin{equation}
\label{initdata1}
 \|\phi_0(x)-\phi_S(x;\la_0)\|_{H^1(\mathbb{R}^3)}+\|\phi_1(x)-\psi_S(x;\la_0)\|_{L^2(\mathbb{R}^3)}\leq \ep,
\end{equation}
then there exists a unique solution $\phi(t, x)\in C([0, T/\ep];
H^1(\mathbb{R}^3))$ of the equation \eqref{FEQUATION} with the
following property: there is a $C^1$ curve $\la(t)=(\om(t), \th(t),
\xi(t)+u_0t, u(t)+u_0)\in\La_{\textnormal{stab}}$ such that
\begin{equation}
\label{obsphi}
\begin{split}
 \|\phi(t, x)-\phi_S(x;\la(t))\|_{H^1(\mathbb{R}^3)}+\|\pa_t\phi(t, x)-\psi_S(x;\la(t))\|_{L^2(\mathbb{R}^3)}
\leq C\ep,\quad \forall t\in [0, T/\ep].
\end{split}
\end{equation}
Moreover
\begin{equation}
\label{obsla} |\la(0)-\la_0|\leq C\ep,\quad
 |\dot \ga|=|\dot \la-V(\la)|\leq C\ep^2,\quad \forall
 t\in[0,T/\ep],
\end{equation}
where $\dot \ga$, $V(\la)$ are defined in \eqref{defofVla},
\eqref{lagaV} and the constant $C$ is independent of $\ep$.
\end{thm}

\begin{remark}
When $g=h$, D. Stuart in \cite{psedoStuart} proved the orbital
stability result up to time $t^*/\ep$ for sufficiently small $t^*$.
\end{remark}

To show that the metric $g^\ep$ is $C^1$, we need to control the
energy momentum tensor $T_{\mu\nu}$ in $H^2$. We hence have to
estimate the higher Sobolev norm of $\phi$. If initially $\phi_0\in
H^3$, $\phi_1\in H^2$, then we have
\begin{prop}
\label{prophigherSobnorm} Assume
\[
\ep_1=\|\phi_0(x)-\phi_S(x;\la_0)\|_{H^3}+\|\phi_1(x)-\psi_S(x;\la_0)\|_{H^2}<\infty.
\]
Let $\la(t)$ be the curve obtained in Theorem
\ref{thmfix}. Then
\[
\sum\limits_{|s|\leq
3}\|\pa^s(\phi(t, x)-\phi_S(x;\la(t))\|_{L^2(\mathbb{R}^3)}\leq
C\max\{\ep, \ep_1\},\quad \forall t\in[0, T/\ep]
\]
for a constant $C$ independent of $\ep, \ep_1$.
\end{prop}

A direct corollary of the above proposition is the boundedness of
the $H^3$ norm of the solution $\phi$.
\begin{cor}
\label{corhighSobnorm}
 If the initially data $\phi_0\in H^3$, $\phi_1\in H^2$, then
\[
\sum\limits_{|s|\leq 3}\|\pa^s\phi(t,
\cdot)\|_{L^2(\mathbb{R}^3)}\leq C, \quad \forall t\in[0,
T/\ep],
\]
where the constant $C$ is independent of $\ep$.
\end{cor}
The above boundedness of $\phi$ in $H^3$ was used in \cite{einsteinStuart}. In this paper, we have to use the fact
that if initially the data $(\phi_0, \phi_1)$ are close to some soliton in $H^3$, then the solution $\phi$ stays close to
some translated solitons for all $t\leq T/\ep$.
\begin{cor}
\label{corsmall} If $\ep_1\leq C_0\ep$ for some constant
$C_0$, then
\[
\sum\limits_{|s|\leq
3}\|\pa^s(\phi-\phi_S(x;\la(t))\|_{L^2(\mathbb{R}^3)}\leq CC_0\ep,\quad
\forall t\in[0, T/\ep].
\]
\end{cor}

To avoid too many constants, we make a convention that $A\les B$ means $A\leq C B$ for some universal constant $C$ depending on
 $h$, $T$, $\la_0$, $m$, $p$.

\subsection{Decomposition of the Solution}
Since solution $\phi$ of \eqref{FEQUATION} exists locally, to begin with, we decompose the solution as follows
\begin{equation}
 \label{decomp}
\begin{cases}
\phi(t, x)=\phi_S(x;\la(t))+\e^{i\Th(\la(t))}\frac{1}{q_\ep(t, x)d_\ep}v(t, x),\quad d_\ep(t, x)=(-\det g)^
{\frac{1}{4}}(\ep t, \ep x),\\
\pa_t\phi(t, x)=\psi_S(x;\la(t))+\e^{i\Th(\la(t))}\frac{1}{p_\ep(t,
x)d_\ep}w(t, x),\quad p_\ep(t, x)=\sqrt{-(g^\ep)^{00}(t, x)},
\end{cases}
\end{equation}
where $\la(t)\in \La$. We also denote
\begin{align*}
&\tilde{p}_\ep(t,x)=\sqrt{-(h^\ep)^{00}(t, x)}, \quad\tilde{p}(t,x)=\sqrt{-h^{00}(t, x)}.
\end{align*}
We define three functions depending on $d_\ep$, $p_\ep$, $q_\ep$ as follows
\[
 a_0=q_\ep d_\ep, \quad a_1=p_\ep d_\ep, \quad a=p_\ep^{-1}q_\ep=a_0 a_1^{-1}, \quad b=p_\ep^{-1}d_\ep.
\]
We now define the function $q_\ep(t, x)$ explicitly. Notice that $h(t, u_0 t)=m_0$ and the
Christoffel symbols $\Ga_{\mu\nu}^{\b}$ for the metric $h$ are vanishing along $(t, u_0 t)$. In
particular we have $\pa h(t, u_0 t)=0$. Hence we can show that
\begin{equation}
\label{pofhm}
\begin{split}
& |h^{\mu \nu}(t, u_0 t+ x)|\leq |(m_0)^{\mu \nu}|+C_0|x|^{2},\quad \forall |x|\leq \delta_0, \quad t\in[0, T],\\
&\left|\sum h^{kl}(t, u_0 t+x)\xi_k\xi_l\right|\geq
(1-C_0|x|^{2})\sum\limits_{k=1}^{3}|\xi_k|^2,\quad \forall |x|\leq
\delta_0, \quad t\in[0, T]
\end{split}
\end{equation}
for some positive constants $C_0$, $\delta_0$, where we recall
$\mu,\nu\in\{0,1, 2, 3\}$ and $k,l\in\{1, 2, 3\}$. Denote
$\rho_0=(1-|u_0|^2)^{-\f12}$. In particular, we have $\rho_0\geq 1$.
Without loss of generality, assume $\frac{1}{4}
K_0^{-3}\rho_0^{-1}\leq 1-3\rho_0C_0 \delta_0^{2}$, $K_0\geq10$ and
$C_0\delta_0^{2}\rho_0\leq \frac{1}{10}$. Then choose a function
$q(x)\in C^{3}(\mathbb{R}^3)$ as follows
\begin{equation*}
q(x)=
\begin{cases}
1-3C_0\rho_0|x|^{2}, \quad |x|\leq \f12 \delta_0,\\
\frac{1}{3} K_0^{-3}\rho_0^{-1}\delta_0^{-2}(|x|^2-\frac{\delta_0^2}{4})
+(\frac{4}{3}\delta_0^{-2}-C_0\rho_0)(\delta_0^2-|x|^2),\quad |x|\in(\f12 \delta_0, \delta_0),\\
\frac{1}{4} K_0^{-3}\rho_0^{-1},\qquad \quad \quad|x|\geq \delta_0.
 \end{cases}
\end{equation*}
We define $q_\ep(t, x)$ as
\begin{equation}
 q_\ep(t, x)=q(\ep (x-u_0t)).
\end{equation}
We prove an inequality.
\begin{lem}
 \label{lemqq}
Let $y=(y_1, y_2, y_3)\in \mathbb{R}^3$. Then
\begin{equation*}
 m^2q_\ep^{-2}b^2+q_\ep^{-2}(h^\ep)^{kl}y_k y_l-2(q_\ep \tilde{p}_\ep)^{-1}|(h^\ep)^{0k}y_k| |y\cdot u|-3\rho_0m(q_\ep \tilde{p}_\ep)^{-1}|(h^\ep)^{0k}y_k| |b|
\geq m^2 b^2 +|y|^2
\end{equation*}
for all $b\in \mathbb{R}$, $(t, x)\in[0,
T/\ep]\times\mathbb{R}^3$, $y=(y_1, y_2, y_3)\in\mathbb{R}^3$ and
$u\in \mathbb{R}^3$, $|u|\leq 1$.
\end{lem}
\begin{proof}
 Fix $t$. After scaling, at point $(t, u_0t+x)$, it suffices to
 show that
\begin{equation}
\label{qq00}
 m^2q^{-2}b^2+q^{-2}h^{kl}y_k y_l-2(q\tilde{ p})^{-1}|h^{0k}y_k| |y\cdot u|-3\rho_0 m(q \tilde{p})^{-1}|h^{0k}y_k| |b|
\geq m^2 b^2 +|y|^2,\forall t\in[0, T].
\end{equation}
When $|x|\geq \delta_0$, we have $q(x)=\frac{1}{4} K_0^{-3}\rho_0^{-1}$. Hence the left hand side of \eqref{qq00}
\begin{align*}
& \geq m^2 q^{-2}b^2+q^{-2}K_0^{-1}|y|^2-2\sqrt{3}q^{-1}K_0^2|y|^2-3\sqrt{3}\rho_0mq^{-1}K_0^2|y||b|\\
&\geq 15K_0^6\rho_0^2m^2b^2+(15K_0^5\rho_0^2-8\sqrt{3}K_0^5\rho_0)|y|^2-12\sqrt{3}K_0^5\rho_0^2m|y||b|+m^2b^2+|y|^2\\
& \geq m^2 b^2+|y|^2.
\end{align*}
When $|x|\leq \delta_0$, we have $|h^{0k}(t, u_0t+x)|\leq C_0 |x|^{2}$ by
\eqref{pofhm}. Notice that $q(x)\leq 1-3C_0\rho_0|x|^2$ for $|x|\leq \delta_0$.
By \eqref{condfix}, we can show that the left hand side of \eqref{qq00}
\begin{align*}
&\geq m^2q^{-2}b^2+q^{-2}(1-C_0|x|^{2})|y|^2-2q^{-1}(1+C_0|x|^{2})C_0|x|^{2}|y|^2\\
&\qquad -3\sqrt{3}\rho_0mq^{-1}(1+C_0|x|^{2})C_0|x|^{2}|y| |b|\\
&\geq (q^{-2}-1)m^2b^2+q^{-2}(1-q^2-C_0|x|^2-3qC_0|x|^2)|y|^2-6\rho_0 m q^{-1}C_0|x|^2|yb|+m^2b^2+|y|^2\\
&\geq 5C_0\rho_0|x|^2m^2b^2+q^{-2}2C_0\rho_0 |x|^2|y|^2-6\rho_0 m q^{-1}C_0|x|^2|yb|+m^2b^2+|y|^2\\
&\geq m^2b^2+|y|^2,
\end{align*}
where we recall that $C_0\rho_0|x|^2\leq C_0\delta_0^2\rho_0\leq \frac{1}{10}$ and $\rho_0=(1-|u_0|^2)^{-\f12}\geq 1$.
Hence the lemma holds.
\end{proof}
In application, we need a similar inequality for the metric $g^\ep$.
\begin{cor}
\label{corgposi} Let $g^\ep(t, x)=g(\ep t, \ep x)$, $h^\ep(t, x)=h(\ep t, \ep x)$. Assume $\|g-h\|_{C^0}\leq \ep$. Then
\begin{equation*}
 m^2q_\ep^{-2}b^2+q_\ep^{-2}(g^\ep)^{kl}y_k y_l-2(q_\ep p_\ep)^{-1}|(g^\ep)^{0k}y_k| |y\cdot u|-2m(q_\ep p_\ep)^{-1}|
(g^\ep)^{0k}y_k| |b| \geq (1-C\ep)(m^2 b^2 +|y|^2)
\end{equation*}
for all $b\in \mathbb{R}$, $(t, x)\in[0,
T/\ep]\times\mathbb{R}^3$, $y=(y_1, y_2, y_3)\in\mathbb{R}^3$ and
$u\in \mathbb{R}^3$, $|u|\leq 1$, where the constant $C$ depends only on $h$ and $m$.
\end{cor}

\subsection{Orthogonality Condition and Modulation Equations}
The decomposition of the solution \eqref{decomp} relies on $\la(t)$, which we write as
\[
 \la(t)=(\om(t), \th(t), \xi(t)+u_0t, u(t)+u_0).
\]
Denote $\ga(t)=(\om(t), \pi(t), \eta(t), u(t))$ such that
\[
 \dot{\la}=\dot\ga+V(\la),
\]
where $V(\la)=(0, \frac{\om}{\rho}, u+u_0, 0)$, $\rho=(1-|u+u_0|^2)^{-\f12}$.
 Recall the notation defined in line \eqref{orthcond}. We choose $\la(t)$ such that the following orthogonality condition
\begin{equation}
 \label{orthcondfix}
<D_\la\phi_S, \e^{i\Th}w>_{dx}=<D_\la\psi_S, \e^{i\Th}v>_{dx}
\end{equation}
holds. Differentiate it with respect to $t$. We can conclude that
the orthogonality condition \eqref{orthcondfix} holds if it holds initially and the curve $\la(t)$ satisfies
\begin{equation*}
\begin{split}
 &<D_\la^2\phi_S\cdot\dot{\la}, \e^{i\Th}w>_{dx}+<D_\la\phi_S, \pa_t\left(\e^{i\Th}w\right)>_{dx}=
<D_\la^2\psi_S\cdot\dot{\la},  \e^{i\Th}v>_{dx}+<D_\la\psi_S, \pa_t\left(\e^{i\Th}v\right)>_{dx}.
\end{split}
\end{equation*}
Using the decomposition \eqref{decomp} and the relation $\dot{\la}=\dot{\ga}+V(\la)$,
we can show that the above equation is equivalent to
\begin{align}
\notag
 &\quad\left(<D_\la \psi_S, a_0 D_\la \phi_S>_{dx}-<D_\la\phi_S, a_1 D_\la\psi_S>_{dx}+
<D_\la^2\phi_S, \e^{i\Th}w>_{dx}-<D_\la^2\psi_S, \e^{i\Th}v>_{dx}\right)\dot{\ga}\\
\label{modeq0}
&=<D_\la(V(\la)D_\la\psi_S), \e^{i\Th}v>_{dx}-<D_\la\phi_S, a_1(\phi_{tt}-D_\la\psi_S V(\la))>_{dx}\\
\notag
&\qquad+<D_\la\psi_S, \dot a_0(\phi-\phi_S)>_{dx}
-<\dot a_1 D_\la\phi_S+(a_1-a_0)D_\la\psi_S, \phi_t-\psi_S>_{dx},
\end{align}
where we have replaced $<D_\la V(\la)\cdot D_\la \phi_S, \e^{i\Th}w>$ with $<D_\la V(\la)\cdot D_\la \psi_S, \e^{i\Th}v>$
by the orthogonality
condition \eqref{orthcondfix}. Denote
\begin{equation}
\label{defofH}
\begin{split}
 H(t, x)&=b\Box_{g^\ep}\phi+a_1\pa_{tt}\phi-\Delta \phi=
p_\ep^{-1} d_\ep\Box_{g^\ep}\phi+p_\ep d_\ep\pa_{tt}\phi-\Delta \phi\\
&=(p_\ep^{-1}d_\ep(g^\ep)^{ij}-(m_0)^{ij})\pa_{ij}\phi+2p_\ep^{-1}d_\ep(g^\ep)^{0k}\pa_{tk}\phi+
\frac{1}{p_\ep d_\ep}\pa_\mu(d_\ep^2(g^\ep)^{\mu\nu})\pa_\nu \phi\\
&=a^{\mu k}\pa_{\mu k}\phi+b^{\nu}\pa_\nu\phi,
\end{split}
\end{equation}
where $a^{\mu k}$, $b^\mu$ are the corresponding coefficients.
Recall the identity \eqref{idenofphiS} for $\phi_S$. Using
integration by parts and observing that
$Re(\e^{-i\Th}D_\la\phi_S)=D_\la|\phi_S|$, we can show that
\begin{align*}
<D_\la\phi_S, \Delta_x \phi>_{dx}&=<D_\la\phi_S, \Delta_x\phi_S>_{dx}+<\Delta_x D_\la\phi_S, \phi-\phi_S>_{dx}\\
&=<D_\la\phi_S, m^2\phi_S-|\phi_S|^{p-1}\phi_S+D_\la\psi_S\cdot V(\la)>_{dx}+<D_\la\Delta_x \phi_S, \phi-\phi_S>_{dx}\\
&=<D_\la(D_\la\psi_S\cdot V(\la)), \phi-\phi_S>_{dx}+<D_\la\phi_S, D_\la\psi_S\cdot V(\la)>_{dx}\\
&\quad -<D_\la\phi_S, -m^2\phi+|\phi_S|^{p-1}\phi+(p-1)|\phi_S|^{p-2}\phi_S Re(\e^{-i\Th}(\phi-\phi_S))>_{dx}.
\end{align*}
Now use the identity \eqref{defofH} to
 replace $p_\ep d_\ep \pa_{tt}\phi$ in \eqref{modeq0}. The equation
\eqref{FEQUATION} of $\phi$ then implies that the right hand side of \eqref{modeq0} can be written as
\begin{align}
\notag
 F(t;\la(t))=&<\dot a_0 D_\la\psi_S, \phi-\phi_S>_{dx}+<(a_0-a_1)D_\la\psi_S-\dot a_1 D_\la\phi_S, \phi_t-\psi_S>_{dx}\\
\notag
&+<D_\la(V(\la)D_\la \psi_S), (a_0-1)(\phi-\phi_S)>_{dx}+<D_\la \phi_S, (a_1-1)D_\la \psi_S \cdot V(\la)>_{dx}\\
\label{defofF}
&+<\pa_k(a^{k\mu}D_\la\phi_S)-b^\mu D_\la \phi_S, \pa_\mu\phi>_{dx}-<b D_\la \phi_S, \e^{i\Th}\mathcal{N}(\la)>_{dx}\\
\notag
&+<(b-1)D_\la\phi_S, m^2\phi-|\phi_S|^{p-1}\phi>_{dx}
-<(b-1)\phi_S D_\la|\phi_S|^{p-1}, \phi-\phi_S>_{dx},
\end{align}
where we denote
\begin{equation}
\label{defnonlin}
 \mathcal{N}(\la)=\e^{-i\Th}\left(|\phi|^{p-1}\phi-|\phi_S|^{p-1}\phi_S-|\phi_S|^{p-1}(\phi-\phi_S)-
(p-1)|\phi_S|^{p-1}\e^{i\Th}Re(\e^{-i\Th}(\phi-\phi_S))\right)
\end{equation}
as the nonlinearity depending also on $\la$. Here recall that $b=p_\ep^{-1}d_\ep$. In particular if $v=v_1+iv_2$ for real functions $v_1$, $v_2$,
then $Re(\e^{-i\Th}(\phi-\phi_S))=(q_\ep d_\ep)^{-1}v_1$ by the decomposition \eqref{decomp}.

\bigskip

To further simplify the above modulation equations, denote
\begin{equation}
\label{defofD}
\begin{split}
& D=<D_\la \psi_S,  D_\la \phi_S>_{dx}-<D_\la\phi_S, D_\la\psi_S>_{dx},\\
 & D_1=<D_\la \psi_S, (a_0-1) D_\la \phi_S>_{dx}-<D_\la\phi_S, (a_1-1)D_\la\psi_S>_{dx},\\
&  D_2=<D_\la^2\phi_S, \e^{i\Th}w>_{dx}-<D_\la^2\psi_S, \e^{i\Th}v>_{dx}.
\end{split}
\end{equation}
We point out here that $D$, $D_1$, $D_2$ are $8\times 8$ matrices. Hence we choose the curve $\la(t)$ such that
\begin{equation}
 \label{modeq}
(D+D_1+D_2)\dot{\ga}=F(t;\la(t)),\quad \dot{\la}=\dot\ga+V(\la), \quad \ga(0)=\la(0)
\end{equation}
and we require $\la(0)$ satisfies the orthogonality condition ~\eqref{orthcondfix}.

\bigskip

 Equations \eqref{modeq} are called the
modulation equations for the curve $\la(t)$. If the orthogonality condition \eqref{orthcondfix} holds initially and $\la(t)$
solves the ODE \eqref{modeq}, then \eqref{orthcondfix} holds as long as $\phi$, $\la(t)$ exist.

\subsection{Initial Data}
We have reduced the orthogonality condition \eqref{orthcondfix} to a coupled system of ODE's for $\la(t)$ if initially $\la(0)$ satisfies
\eqref{orthcondfix}. We use the implicit functional theorem to show the existence of the initial data $\la(0)$ satisfying
the orthogonality condition.
\begin{lem}
\label{datapre}
Denote
\[
 \ep=\|\phi_0(x)-\phi_S(x;\la_0)\|_{H^1(\mathbb{R}^3)}+\|\phi_1(x)-\psi_S(x;\la_0)\|_{L^2(\mathbb{R}^3)}
\]
for some $\la_0\in \La_{\textnormal{stab}}$. Then there exists a positive constant $\ep_1(\la_0)$,
 depending only on $\la_0$, such that if $\ep<\ep_1(\la_0)$, then there exists
$\la(0)\in  \La_{\textnormal{stab}}$ with the property that if
\begin{equation*}
\begin{cases}
\phi_0(x)=\phi_S(x;\la(0))+\e^{i\Th(\la(0))}\frac{1}{q_\ep(0, x)d_\ep}v(x;\la(0)),\\
\phi_1( x)=\psi_S(x;\la(0))+\e^{i\Th(\la(0))}\frac{1}{p_\ep(0,
x)d_\ep}w(x;\la(0)),
\end{cases}
\end{equation*}
then the orthogonality condition holds
\begin{align*}
&<D_\la\psi_S(x;\la(0)),
\e^{i\Th(\la(0))}v(x;\la(0))>_{dx}=<D_\la\phi_S(x;\la(0)),
\e^{i\Th(\la(0))}w(x;\la(0))>_{dx}.
\end{align*}
Moreover, we have
\begin{align*}
&|\la(0)-\la_0|\leq C(\la_0)\ep,\\
& \|v(x;\la(0))\|_{H^1}+\|w(x;\la(0))\|_{L^2}\leq C(\la_0)\ep
\end{align*}
for some constant $C(\la_0)$ depending only on $\la_0$.
\end{lem}
\begin{proof}
Define a functional $\mathcal{F}: H^1\times L^2\times
\mathbb{R}^8\rightarrow \mathbb{R}^8$ such that
\begin{align*}
\mathcal{F}(v, w, \la)=&<a_1 D_\la\phi_S(\la;x),\phi_1(x)
-\psi_S(\la;x)>_{dx}-<a_0
D_\la\psi_S(\la;x),\phi_0(x)-\phi_S(\la;x)
>_{dx}.
\end{align*}
In particular, we have $\mathcal{F}(0, 0, \la_0)=0$. Notice that
\begin{align*}
F_\la(0, 0, \la_0)=<a_0 D_\la\psi_S, D_\la \phi_S>_{dx}-<a_1 D_\la\phi_S, D_\la \psi_S>_{dx}
=D+D_1.
\end{align*}
By Lemma \ref{nondegD} and Lemma \ref{lemDF} proven later, we can conclude that if $\ep$ is sufficiently small, then $D+D_1$
 is nondegenerate initially. Since $\mathcal{F}$ is
Lipschitz continuous in $(v, w)$, the implicit function theorem then implies that there exists
$\la(0)\in \La_{\textnormal{stab}}$ satisfying
the orthogonality condition \eqref{orthcondfix} as well as the estimates in the lemma.
\end{proof}

\subsection{Bootstrap Argument}
By Lemma \ref{datapre}, we can choose $\la(0)$ close to $\la_0$ such that the orthogonality condition \eqref{orthcondfix} holds
initially. The local existence result shows that there is a short time solution $\phi(t, x)$ of \eqref{FEQUATION}. To prove
the existence of solution $\la(t)$ of the modulation equations \eqref{modeq}, we have to demonstrate that the $8\times 8$
matrix $D+D_1+D_2$ is nondegenerate, which depends on $\la(t)$ itself. We thus need to consider the modulation equations
coupled to the nonlinear wave equation \eqref{FEQUATION}. Since initially the radiation term $(v, w)$ is small in $H^1\times L^2$,
we show that for the coupled equations, $(v, w)$ stays small in $H^1\times L^2$ for all $t\leq T/\ep$, which implies that
the modulation equations are solvable and we can control the modulation curve $\la(t)$.
\begin{prop}
\label{propH1est}
Let $(\phi(t, x), \la(t))$ be solutions of \eqref{FEQUATION}, \eqref{modeq} on $[0, T/\ep]\times\mathbb{R}^3$. Assume the complex
functions $v(t, x)$, $w(t, x)$ satisfy the decomposition \eqref{decomp}. Then for sufficiently small $\ep$, we have
\begin{align*}
&\|v(t, x)\|_{H^1(\mathbb{R}^3)}+\|w(t, x)\|_{L^2(\mathbb{R}^3)}\leq C\ep,\quad \forall t\in[0,T/\ep],\\
&|\dot{\ga}(t)|\leq C\ep^2,\quad \forall t\in[0, T/\ep]
\end{align*}
for some constant $C$ independent of $\ep$.
\end{prop}

We use bootstrap argument to prove this proposition. To estimate $|\dot \ga|$, we rely on the modulation equations \eqref{modeq}
together with the fact that $D+D_1+D_2$, as an $8\times 8$ matrix, is nondegenerate. It turns out that $D$ depends only on
$\om$ and is nondegenerate for $\om=\om_0$. Thus if $\om(t)$ is close to $\om_0$, we have the nondegeneracy of $D$.
From this point of view, we take a subset of $\La_{\textnormal{stab}}$ defined as follows
\begin{equation}
\label{defofLadel}
 \La_{\delta_0}=\{(\om, \th, \xi, u)|\om\in [\om_0-\delta_0, \om_0+\delta_0]\subset\{\om|\frac{p-1}{6-2p}<\frac{\om^2}{m^2}<1\}
,\quad |u-u_0|\leq \delta_0\},
\end{equation}
where $\delta_0$ is a positive constant, which will be
determined in Lemma \ref{lempositdH}. Here we recall that $\la_0=(\om_0, \th_0, 0, u_0)\in \La_{\textnormal{stab}}$. Thus
we can choose $\delta_0$ sufficiently small such that $\La_{\delta_0}$ is nonempty. Our first bootstrap assumption is that $\la(t)\in
\La_{\delta_0}$ for all $t\in[0, T/\ep]$.

By the definition of $D_1$, $D_2$ in line \eqref{defofD}, we see that $D_1$ has size $\ep$ if the center $(t, \xi(t)+u_0t)$ of the soliton
$\phi_S(x;\la(t))$ does not diverge far away from the timelike geodesic $(t, u_0t)$.
Hence we expect that $|\xi(t)|$ is uniformly bounded, which is also suggested by Proposition \ref{propH1est}.
In fact since $\dot \xi=u(t)+\dot\eta$, if $|\dot \ga|=|(\dot \om, \dot \pi, \dot\eta, \dot u)|\leq C\ep^2$, then
\[
|\xi(t)|\leq |\xi(0)|+|u(0)|t+Ct^2\ep^2\leq |\xi(0)|+(C+1)T^2,\quad \forall t\leq T/\ep.
\]
We remark here that this is compatible with Theorem \ref{thmbaby} as if we scale it back to the
space $[0, T]\times \mathbb{R}^3$, the center of the soliton becomes $(t, \ep\xi+u_0 t)$ which is close to the geodesic
$(t, u_0 t)$. Similarly, $D_2$ is an error term if $\|v\|_{H_1(\mathbb{R}^3)}+\|v\|_{L^2(\mathbb{R}^3)}$ is small.

To prove Proposition \ref{propH1est}, in addition to the assumption that $\la(t)\in \La_{\delta_0}$, we assume
\begin{align}
&\label{baxi}
 |\xi(t)|\leq 2C_2, \quad \forall t\in[0, T/\ep],\\
 &\label{bawv}
 \|w(t, x)\|_{L^2(\mathbb{R}^3)}+\|v(t, x)\|_{H^1(\mathbb{R}^3)}\leq \delta_1,\quad \forall t\in[0, T/\ep]
 \end{align}
 for some constants $C_2$, $\delta_1$ which will be fixed later on. Without loss of generality, we assume
$C_2>1$, $\delta_1<1$, $C_2^4\ep<1$. These are the bootstrap assumptions
 in this subsection.

As having mentioned previously, we do not have estimates for $\|g^\ep-h^\ep\|_{L^2(\Si_0)}$ initially. In fact, our construction
of initial data implies that $g^\ep-h^\ep$ is not bounded in $L^2$ for general data. To bound $g^\ep-h^\ep$, we rely on Hardy's
inequality.
\begin{lem}
 \label{lemhardyineq}
Let $f(x)\in C^1(\mathbb{R}^3)$. Assume $f(x)\rightarrow 0$ as $|x|\rightarrow \infty$. Then
\[
 \|f(x)|x|^{-1}\|_{L^2(\mathbb{R}^3)}\leq 6\|\nabla f(x)\|_{L^2(\mathbb{R}^3)}.
\]
 \end{lem}
This inequality can be proven by using integration by parts under polar coordinate system.
Detailed proof could be found in \cite{dr3}, \cite{yang1}. In particular, using Sobolev embedding, the assumption \eqref{bapsi}
 implies that
\begin{equation}
 \label{psiep}
 \|\psi^\ep\|_{C^{1,\f12}(\mathbb{R}^3)}+\|\psi^\ep(1+|x|)^{-1}\|_{L^2(\mathbb{R}^3)}+\|\pa^{s+1}\psi^\ep\|_{L^2(\mathbb{R}^3)}
\les \ep^2,\quad |s|\leq 2, \quad \forall t\in[0, T/\ep].
\end{equation}

To prove Proposition \ref{propH1est}, we use bootstrap argument. We show that the matrix $D+D_1+D_2$ are nondegenerate initially. Thus
the modulation equations can be solved locally. If under the bootstrap assumptions \eqref{baxi}, \eqref{bawv}, we can show that
Proposition \ref{propH1est} holds for some constant $C$ which is independent of $C_2$, then we can close the
bootstrap assumptions if we take
$C_2=|\xi(0)|+(C+1)T^2$, $\delta_1=2\ep$, $\delta_0=2CT\ep$. We thus can conclude Proposition \ref{propH1est} if $\ep$ is sufficiently
small.

The strategy is as follows: we first show the nondegeneracy the
modulation equations and obtain the estimates for $|\dot\ga(t)|$.
Under the orthogonality condition \eqref{orthcondfix}, we use energy
estimate to demonstrate that the radiation $(v, w)$ is small in
$H^1\times L^2$.

\subsubsection{Nondegeneracy of the Modulation Equations}
To obtain estimates of $\la(t)$ and to show the local existence of
the modulation equations \eqref{modeq}, we demonstrate that, under
the bootstrap assumption $\la(t)\in \La_{\delta_0}$, the leading
coefficient $D$ in \eqref{modeq} is nondegenerate.
\begin{lem}
 \label{nondegD}
Let $D$ be the $8\times 8$ matrix defined in \eqref{defofD}. If
$\la=(\om, \th, \xi, u)\in \La_{\delta_0}$, then
\begin{equation}
 |\det D|\geq C(\delta_0,\la_0)>0
\end{equation}
for some constant $C(\delta_0,\la_0)$ depending on $\delta_0$, $\la_0$. In particular, $D$ is nondegenerate.
\end{lem}
\begin{proof} Direct calculations show that
\begin{align*}
 D_{\om\th}&=D_\om(\om \|f_\om\|_{L^2}^2),\quad D_{\om\xi}= \rho u D_\om B,\quad  D_{\xi u}=-\rho B(I+\rho^2 u\cdot u),\\
D_{\om u}&= D_{\th \xi}=D_{\xi\xi}=D_{uu}=D_{\th u}= D_{\th\th}=0,
\end{align*}
where we denote
\begin{equation}
 \label{defofB}
B(t)=\om^2 \|f_\om(x)\|_{L^2(\mathbb{R}^3)}^2+\frac{1}{3}\|\nabla_x f_\om(x)\|_{L^2(\mathbb{R}^3)}^2.
\end{equation}
Notice that $D$ is antisymmetric. We can show that
\[
\det D=|D_\om(\om \|f_\om\|_{L^2}^2)|^2\cdot |\det D_{\xi
u}|^2=|D_\om(\om \|f_\om\|_{L^2}^2)|^2 B^{6}\rho^{10}.
\]
Using the scaling property \eqref{scalling} of $f_{\om}(x)$, when
$
\frac{p-1}{6-2p}<\frac{\om^2}{m^2}<1
$, we have
\[
D_\om(\om \|f_\om\|_{L^2}^2)=(m^2-\frac{6-2p}{p-1}\om^2)(m^2-\om^2)^{\frac{9-5p}{2(p-1)}}\|f\|_{L^2(\mathbb{R}^3)}^2<0.
\]
Since $\rho\geq 1$, the lemma then follows.
\end{proof}

\subsubsection{Estimates for the Modulation Curve}
We have shown that $D$ is nondegenerate. To estimate $\ga(t)$ by
using the modulation equations, we have to show that $D_1$, $D_2$,
$F(t;\la(t))$ are error terms. Estimates for $D_1$, $D_2$ can be
obtained by using Cauchy-Schwartz inequality. The main difficulty
for estimating $F(t;\la(t))$ lies in the nonlinear term
$<D_\la\phi_S, \e^{i\Th}\mathcal{N}(\la)>_{dx}$. We first prove two lemmas. The first lemma gives control of the
nonlinearity $\mathcal{N}(\la)$. Let
\begin{equation*}
 N(x)= \frac{1}{p+1}|x|^{p+1}
\end{equation*}
for complex number $x$. In particular, we have $D_{\bar x}N=\f12
|x|^{p-1}x$, where $\bar x$ is the complex conjugate of $x$. By the
definition \eqref{defnonlin} of the nonlinearity $\mathcal{N}(\la)$,
we can write
\[
\mathcal{N}(\la)=2D_{\bar x}N(f_\om+(q_\ep d_\ep)^{-1}v)-2D_{\bar x}N(f_\om)-2D_{\bar x}DN(f_\om)\cdot
(q_\ep d_\ep)^{-1}v,
\]
where $DN\cdot v=D_xNv+D_{\bar x}N\bar v$.
\begin{lem}
 \label{lemnonlinear}
Let $\mathcal{N}(\la)$ be defined in line \eqref{defnonlin}. Assume $\phi$ decomposes as \eqref{decomp}. For all $p\geq 2$, we have
\begin{align}
\label{Nvpt}
 &|\mathcal{N}(\la)|\les |v|^2+|v|^p,\\
\label{nonlinq} &\left|N(f_\om+v)-N(f_\om)-DN(f_\om)v-\f12
vD^2N(f_\om)v\right|\les |v|^3+|v|^{p+1}.
\end{align}
\end{lem}
\begin{proof}
Let $\tilde{v}=(q_\ep d_\ep)^{-1}v$. We can write $\mathcal{N}(\la)$ as an integral
\begin{align*}
\mathcal{N}(\la)=4\int_{0}^{1}\int_{0}^{1}s\tilde{v}^2 D^2D_{\bar
x}N(f_\om+ts\tilde{v})dtds.
\end{align*}
Since $p\geq 2$, we can show that
\[
|D^2 D_{\bar x}N(f_\om+ts\tilde{v})|\les 1+|f_\om + ts \tilde{v}|^{p-2}\les 1+|\tilde{v}|^{p-2}.
\]
Hence \eqref{Nvpt} holds. The second inequality \eqref{nonlinq} follows similarly.
\end{proof}
This lemma implies that the nonlinearity in $F(t;\la(t))$ is in fact
higher order error term. For the other terms in $F(t;\la(t))$,
observe that if $g=h$, then $p_\ep-1$, $d_\ep-1$, $q_\ep-1$ have
size $\ep^2$ near the geodesic $(t, u_0t )$ by \eqref{pofhm} and the
fact that $\phi_S$ decays exponentially. To pass $h^\ep$ to $g^\ep$
satisfying \eqref{psiep}, we use the following lemma.
\begin{lem}
\label{lemb}
Let $F_1(x)$, $F_2(x)$ be two $C^1$ functions such that $F_1(0)=F_2(0)=0$. Then we have
\begin{equation*}
 \left\|F_1(\phi_S(x;\la))(F_2(g^\ep)-F_2(m_0))\right\|_{L^r(\mathbb{R}^3)}\les \|F_1\|_{C^1}\|F_2\|_{C^1}(1+|\xi|^2)\ep^{2}
,\quad \forall r\in[1, 2],
\end{equation*}
where $\la=(\om, \th, \xi+u_0 t, u+u_0)\in\La_{\textnormal{stab}}$. Similarly
\begin{equation*}
 \left\|F_1(\phi_S(x;\la))F_2(\pa g^\ep)\right\|_{L^r(\mathbb{R}^3)}\les \|F_1\|_{C^1}\|F_2\|_{C^{1}}(1+|\xi|)\ep^{2},
\quad \forall r\in[1, \infty].
\end{equation*}
\end{lem}
\begin{proof}
Recall that $\phi_S(x;\la)=\e^{i\Th}f_{\om}(z)$,
$z=A_{u+u_0}(x-\xi-u_0t)$. We conclude that $|z|^2\geq
|x-\xi-u_0t|^2$. Hence by Theorem \ref{propoff}, we have
\[
|F_1(\phi_S(x;\la))|\les \|F_1\|_{C^1}\e^{-c(\om) |z|}
\les\|F_1\|_{C^1}\e^{-c(\om) |x-\xi-u_0t|}.
\]
Since $\pa h(t, u_0 t)=0$, $h(t, u_0t)=m_0$, we can show that
\[
 |F_2(g^\ep)(t, u_0t +x)-F_2(m_0)(t, u_0 t+x)|\leq \|F_2\|_{C^1}(\ep^2|x|^2\|h\|_{C^2}+|g^\ep-h^\ep|).
\]
Then \eqref{psiep} and H$\ddot{o}$lder's inequality imply that
\begin{align*}
 &\left\|F_1(\phi_S(x;\la))(F_2(g^\ep)-F_2(m_0))\right\|_{L^r(\mathbb{R}^3)}\\
&\les
 \|F_1(\phi_S)(1+|z|)\|_{L^{\frac{2r}{2-r}}}
\|F_2\|_{C^1}\|\psi^\ep(1+|z|)^{-1}\|_{L^2}+\|F_1\|_{C^1}\|F_2\|_{C^2}
(1+|\xi|^2)\ep^2\\
&\les \|F_1\|_{C^1}\|F_2\|_{C^1}(1+|\xi|^2)\ep^2,\quad r\in[1, 2].
\end{align*}
Similarly, by \eqref{psiep}, we have
\begin{align*}
 \left\|F_1(\phi_S(x;\la))F_2(\pa g^\ep)\right\|_{L^r(\mathbb{R}^3)}&\les \left\|F_1(\phi_S(x;\la))F_2(\pa h^\ep)\right\|_{L^r(\mathbb{R}^3)}
+\|F_2\|_{C^1}\|F_1(\phi_S)\pa\psi^\ep\|_{L^r}\\
&\les \|F_1\|_{C^1}\|F_2\|_{C^1}(1+|\xi|)\ep^2,\quad r\in[1, \infty].
\end{align*}
\end{proof}
Having proven the above two lemmas, we are able to estimate $\la(t)$. It suffices to show that $D_1$, $D_2$, $F$ are error terms.
\begin{prop}
 \label{lemDF}
Under the bootstrap assumptions \eqref{baxi}, \eqref{bawv}, we have
\begin{align}
\label{D1}
  \|D_1\|&\les (1+|\xi|^{2})\ep^{2}\les (C_2\ep)^2,\\
\label{D2}
\|D_2\|&\les \|w\|_{L^2(\mathbb{R}^3)}+\|v\|_{H^1(\mathbb{R}^3)}\les \delta_1,\\
\label{F} \|F\|&\les (C_2\ep)^2+\|v\|_{H^1}^2.
\end{align}
\end{prop}
If $\ep$, $\delta_1$ is sufficiently small, by Lemma \ref{nondegD}, estimates \eqref{D1}, \eqref{D2} show that $D$ dominates
 $D_1+D_2$. Hence $D+D_1+D_2$ is nondegenerate. Then using \eqref{F}, we can
estimate $\dot\ga$.

\begin{proof}
Inequality \eqref{D1} follows from Lemma \ref{lemb} and the bootstrap assumption \eqref{baxi}. Estimate \eqref{D2} for
$D_2$ can be obtained by using Cauchy-Schwartz's inequality and the assumption \eqref{bawv}.

For \eqref{F}, we apply Lemma \ref{lemnonlinear} to control the
nonlinearity $<bD_\la\phi_S, \e^{i\Th}\mathcal{N}(\la)>_{dx}$ and
use Lemma \ref{lemb} to estimate other terms by observing that
\[
b-1=p_\ep^{-1}d_\ep-1, \quad a_0-1=q_\ep d_\ep-1, \quad a_1-1=p_\ep d_\ep -1,\quad
 a^{k\mu}, \quad b^\mu
\]
 can be written as the form $F_1(g^\ep)-F_1(m_0)$.
Hence we can show that
\[
 \|F\|\les\ep^2(1+|\xi|^2)+\||v|^2+|v|^p\|_{L^1}\les (C_2\ep)^2+\|v\|_{H^1}^2.
\]
\end{proof}

To make $D+D_1+D_2$ to be nondegenerate, we choose $\delta_1$ in the following way. Fix $\delta_0$. Then let $\ep$, $\delta_1$
 be sufficiently small such that
\begin{equation*}
 \|D_1\|+\|D_2\|\leq C(C_2\ep)^2+C\delta_1\leq
\frac{1}{10} C(\delta_0, \la_0)^{\frac{1}{8}},
\end{equation*}
where $C$ is the implicit constant in Proposition \ref{lemDF}, which by our
notations is independent of $\ep$, $C_2$. Here $C(\delta_0,
\la_0)$ is the constant in Lemma \ref{nondegD}. So far, only $\delta_0$, $C_2$ are unknown constants. However,
all the implicit constants are independent of $C_2$, $\ep$. With this choice of $\delta_1$, we can estimate $\dot \ga$.
\begin{cor}
\label{corcontrga}
Let $\delta_1$ be chosen as above.
Suppose $\la(t)\in \La_{\delta_0}$ and $\xi(t)$, $w, v$ satisfy the bootstrap assumptions \eqref{baxi},
\eqref{bawv}. Then we have
\begin{equation*}
|\dot{\ga}|\les \|F\|\les(C_2\ep)^2+\|v\|_{H^1}^2.
\end{equation*}
\end{cor}
This corollary shows that as long as the radiation term $v$ in the decomposition \eqref{decomp} of the solution \eqref{decomp}
is small, we can solve the modulation equations \eqref{modeq} and obtain estimates for the modulation curve $\dot \ga(t)$.
The modulation equations are used to guarantee the orthogonality condition \eqref{orthcondfix}. Next, we show that
under the orthogonality condition, the energy $\|v\|_{H^1}+\|w\|_{L^2}$ of the radiation term $(v, w)$ is small.

\subsubsection{Energy Decomposition}
We use energy estimates to show that the radiation term $(v, w)$ is
small in $H^1\times L^2$. We consider the almost conserved energies
for the full solution $\phi$ of the nonlinear wave equation
\eqref{FEQUATION}. We mention here that similar almost conservations laws have also been studied in \cite{davidKGM}. Using the decomposition \eqref{decomp}, we
decompose the energies for $\phi$ around the solitons $\phi_S$. The
associated energy for the solitons $\phi_S$ can be computed
explicitly up to an error. Then combining with the orthogonality
condition, we can obtain estimates for $\|v\|_{H^1}+\|v\|_{L^2}$. We
first define almost conserved energies for $\phi$ and decompose them
around the solitons $\phi_S$.

\bigskip

 We recall the energy momentum
tensor with respect to the metric $g^\ep$
\[
 T_{\mu\nu}[\phi]=<\pa_\mu\phi, \pa_\nu\phi>-\f12 g^\ep_{\mu\nu}(<\pa^{\ga}\phi, \pa_\ga \phi>+2\mathcal{V}(\phi)),
\]
where
\[
\mathcal{V}(\phi)=\frac{m^2}{2}|\phi|^2-\frac{1}{p+1}|\phi|^{p+1}=\frac{m^2}{2}|\phi|^2-N(\phi).
\]
For a vector field $Y$ and a real function $\b$, we have the identity
\begin{equation}
\label{enerestfor}
 D^{\mu}(\b T_{\mu\nu}[\phi]Y^\nu)=\b T^{\mu\nu}[\phi]\pi^{Y}_{\mu\nu}+\b<\Box_{g^\ep}\phi-m^2\phi+|\phi|^{p-1}\phi, Y(\phi)>
+T(Y, D\b),
\end{equation}
where $\pi^{Y}_{\mu\nu}=\f12 \mathcal{L}_Y g^\ep_{\mu\nu}$ is the
deformation tensor of $Y$ and $D\b$ is the gradient of the function $\b$. Consider the region $\mathbb{R}^3\times
[0, t]$. First we take $\b \equiv 1$, $Y=\pa_t$. Using Stoke's
formula and equation \eqref{FEQUATION}, we have
\begin{equation}
 \label{EnergyC}
\mathcal{H}(t)=\mathcal{H}(0)-\int_{0}^{t}\int_{\mathbb{R}^3}T^{\mu\nu}[\phi]\pi^{\pa_t}_{\mu\nu}d\vol,
\end{equation}
where
\begin{equation}
\label{defenergy}
\begin{split}
 \mathcal{H}(t)&=-\int_{\mathbb{R}^3}-T_{0\nu}[\phi](g^{\ep})^{0\nu}d\si=\f12
 \int_{\mathbb{R}^3}<\pa^k \phi,\pa_k \phi>-<\pa^t\phi, \pa_t\phi>+ 2\mathcal{V}(\phi) d\si.
\end{split}
\end{equation}
Here we recall that $d\si=\sqrt{-\det g^\ep}=d_\ep^2 dx$,  $d\vol =d_\ep^2 dt dx$.

Then let $\b=a=p_\ep^{-1}q_\ep$, $Y=\pa_k$. We have
\begin{equation}
 \label{MomentumC}
\begin{split}
\Pi_k(t)=\Pi_k(0)+\int_0^{t}\int_{\mathbb{R}^3}a
T^{\mu\nu}[\phi]\pi_{\mu\nu}^{\pa_k}+ \pa^\mu a T_{\mu
k}[\phi]d\vol,
\end{split}
\end{equation}
where
\begin{equation}
 \label{defmomentum}
\Pi_k(t)=\int_{\mathbb{R}^3}a T_{\mu
k}[\phi](g^\ep)^{0\mu}d\si=\int_{\mathbb{R}^3}p_\ep^{-1}q_\ep<\pa^t\phi,
\pa_k\phi> d\si.
\end{equation}
We remark here that $\mathcal{H}$, $\Pi_k$ are conserved quantities
of the equation \eqref{FEQUATION} if the metric $g^\ep$ is flat. For general slowly varying metric $g^\ep$, we will show in the
next section that these quantities are almost conserved, that is, the error is of higher order.

Since $\phi$ is complex valued, we define the charge
of the solution $\phi$
\begin{equation}
 \label{defcharge}
Q(t)=\int_{\mathbb{R}^3}a<i\pa^t\phi, \phi>d\si,\quad d\si=d_\ep^2dx.
\end{equation}
Using the equation ~\eqref{FEQUATION}, we can get an integral form of $Q(t)$
\begin{equation}
 \label{ChargeC}
\begin{split}
Q(t)&=Q(0)+\int_{0}^{t}\int_{\mathbb{R}^3}\frac{\pa}{\pa t}(a d_\ep^2)<i\pa^t\phi, \phi>
 +a d_\ep^2
\left(<i\pa_t\pa^t\phi, \phi>+<i\pa^t\phi, \pa_t\phi>\right) dxds\\
&=Q(0)+\int_{0}^{t}\int_{\mathbb{R}^3}d_\ep^2<i\pa_\mu a\pa^\mu\phi, \phi>+a d_\ep^2<i\pa^\mu\phi, \pa_\mu\phi> dxds\\
&=Q(0)+\int_{0}^{t}\int_{\mathbb{R}^3}<i\pa_\mu a\pa^\mu\phi, \phi> d_\ep^2dxds,
\end{split}
\end{equation}
where we have used the equation
\[
 \pa_\mu\pa^\mu\phi+d_\ep^{-2}\pa_\mu(d_\ep^2)\pa^\mu\phi-m^2\phi+|\phi|^{p-1}\phi=0
\]
together with the fact that the quadratic form $<i(g^\ep)\cdot, \cdot>$ is
antisymmetric.

\begin{remark}
The reason that we put some weight in the decomposition
\eqref{decomp}, also in the definition of the almost conserved
quantities $\Pi_k$, $Q$, is to reduce the positivity of the energy
of the radiation term $(v,w)$ to the case in Minkowski space, which has been
proven in \cite{moduStuart}, see Proposition \ref{positen}.
\end{remark}

Next, we expand $\mathcal{H}$, $\Pi_k$, $Q$, as functionals of the
full solution $\phi$, around the soliton $\phi_S$. The soliton part
can be calculated explicitly. The crossing terms are close to the
orthogonality condition \eqref{orthcondfix} and hence are small. A
combination of the quadratic terms in $v, w$ gives the energy
$E_0(t)$ defined in \eqref{defofE0}, which is positive definite by
Proposition \ref{positen}. We thus end up with an estimate for
$\|w\|_{L^2}+\|v\|_{H^1}$ if we can further show that
 $\mathcal{H}$, $\Pi_k$, $Q$ are almost conserved.

\bigskip

We first consider the angular momentum $\Pi_k(t)$. Using the decomposition
~\eqref{decomp}, we can show that
\begin{equation*}
\begin{split}
\Pi_k(t)&=\int_{\mathbb{R}^3}p_\ep^{-1}q_\ep<\pa^t\phi, \pa_k\phi>d\si \\
&=\int_{\mathbb{R}^3}p_\ep^{-1}q_\ep <(g^\ep)^{00}\pa_t\phi, \pa_k\phi>+p_\ep^{-1}q_\ep (g^\ep)^{0l}<\pa_l\phi, \pa_k\phi>d\si\\
&=\int_{\mathbb{R}^3}p_\ep^{-1}q_\ep (g^\ep)^{00}<\psi_S+\e^{i\Th}(p_\ep d_\ep)^{-1}w, \pa_k(\phi_S+\e^{i\Th}(q_\ep d_\ep)^{-1}v)>\\
&\qquad+p_\ep^{-1}q_\ep (g^\ep)^{0l}<\pa_l(\phi_S+\e^{i\Th}(q_\ep
d_\ep)^{-1}v), \pa_k(\phi_S+\e^{i\Th}(q_\ep d_\ep)^{-1}v)> d\si.
\end{split}
\end{equation*}
Recall that $h\in C^4$ and $h=m_0$ along the geodesic $(t, u_0t)$.
Since $g^\ep$ is close to $h^\ep$ in terms of the condition
\eqref{psiep}, we compare the above integral with that associated to
the metric $h$. We hence can write $\Pi_k(t)$ as a sum of main terms
plus an error
\begin{equation*}
\begin{split}
\Pi_k(t)=&-<\psi_S, \pa_k\phi_S>_{dx}-<\psi_S, \pa_k(\e^{i\Th}v)>_{dx}-<\e^{i\Th}w, \pa_k\phi_S>_{dx}-
<\e^{i\Th}w, \pa_k(\e^{i\Th}v)>_{dx}\\
& +<(p_\ep
q_\ep)^{-1}(g^\ep)^{0l}\pa_l(\e^{i\Th}v),\pa_k(\e^{i\Th}v)>_{dx}+Err(\Pi_k),
\end{split}
\end{equation*}
in which, by using Lemma \ref{lemb}, we can show that the error term $Err(\Pi_k)$ can be bounded as follows
\begin{equation}
 \label{errpi}
\begin{split}
|Err(\Pi_k)|&\les\ep^2
(1+|\xi|^{2})(1+\|v\|_{H^1})+\|\pa g^\ep\|_{L^\infty}(\|w\|_{L^2}\|v\|_{H^1}+\|v\|_{H^1}^2)\\
&\les(1+|\xi|^2)\ep^2+\ep\|w\|_{L^2}^2+\ep\|v\|_{H^1}^2.
\end{split}
\end{equation}
We show the crossing terms are vanishing due to the orthogonality condition. In fact,
recall the definition of $z$ in line \eqref{z}. We find that
$$D_\xi\phi_S=D_z\phi_S\cdot\frac{\pa z}{\pa \xi}=-D_z\phi_S\cdot\frac{\pa z}{\pa x}=-\nabla_x\phi_S.$$
Since $\la(t)$ solves the modulation equations \eqref{modeq} and hence the orthogonality condition \eqref{orthcondfix}
holds, integration by parts implies that
\[
 <\psi_S, \nabla_x(\e^{i\Th}v)>_{dx}+<\e^{i\Th}w, \nabla_x\phi_S>_{dx}=-<\e^{i\Th}w, D_\xi\phi_S>_{dx}+<D_\xi\psi_S,
 \e^{i\Th}v>_{dx}=0.
\]
We now compute the soliton part. Since $dz=\rho dx$, we can compute
\begin{align*}
 <\psi_S, \nabla_x\phi_S>_{dx}&=<\e^{i\Th}(i\rho \om f_\om-\rho (u+u_0)\cdot \nabla_z f_\om), \e^{i\Th}(-i\rho\om f_\om (u+u_0)+\nabla_zf_\om \cdot\frac{\pa z}{\pa x})>_{dx}\\
&=-\int_{\mathbb{R}^3}\rho^2\om^2 f_\om^2 (u+u_0)+ \rho (u+u_0)\cdot \nabla_z f_\om \nabla_z f_\om \cdot \frac{\pa z}{\pa x}dx\\
&=-\rho B (u+u_0),
\end{align*}
where by Theorem \ref{propoff}, the ground state $f_\om$ is spherical symmetric and
$B$ is given in line \eqref{defofB}. We hence can write
\begin{equation}
 \label{Pidecomp}
\begin{split}
\Pi_k(t)=&\rho B (u^k+u_0^k)-<\e^{i\Th}w, \pa_k(\e^{i\Th}v)>_{dx}
+<(p_\ep
q_\ep)^{-1}(g^\ep)^{0l}\pa_l(\e^{i\Th}v),\pa_k(\e^{i\Th}v)>_{dx}+Err(\Pi_k),
\end{split}
\end{equation}
where $u=(u^1, u^2, u^3)$, $u_0=(u_0^1, u_0^2, u_0^3)$ and the error term $Err(\Pi_k)$ satisfies \eqref{errpi}.

\bigskip

We decompose the charge $Q(t)$ in a similar way. Recall that
$\b=p_\ep^{-1}q_\ep $. We can show that
\begin{equation}
 \label{Qdecomp0}
\begin{split}
Q(t)&=\int_{\mathbb{R}^3}p_\ep^{-1}q_\ep<i(g^\ep)^{00}(\psi_S+\e^{i\Th}(p_\ep d_\ep)^{-1}w),
\phi_S+\e^{i\Th}(q_\ep d_\ep)^{-1}v>\\
&\qquad+p_\ep^{-1}q_\ep<i(g^\ep)^{0k}\pa_k(\phi_S+\e^{i\Th}(q_\ep d_\ep)^{-1}v),\phi_S+\e^{i\Th}(q_\ep d_\ep)^{-1}v> d\si\\
&=-<i\psi_S, \phi_S>_{dx}-<iw, v>_{dx}-<i\psi_S, \e^{i\Th}v>_{dx}-<i\e^{i\Th}w, \phi_S>_{dx}\\
&\qquad+<(p_\ep
q_\ep)^{-1}(g^\ep)^{0k}i\pa_k(\e^{i\Th}v),\e^{i\Th}v>_{dx}+Err(Q),
\end{split}
\end{equation}
where the error term $Err(Q)$ satisfies the estimate
\begin{equation}
\label{errq}
 |Err(Q)|\les(1+|\xi|^2)\ep^2+\ep\|w\|_{L^2}^2+\ep\|v\|_{H^1}^2.
\end{equation}
Now observe that $D_\th\phi_S=i\phi_S, D_\th\psi_S=i\psi_S$. The orthogonality condition \eqref{orthcondfix} together with
integration by parts implies that
\[
 <i\psi_S, \e^{i\Th}v>_{dx}+<i\e^{i\Th}w, \phi_S>_{dx}=<D_\th\psi_S, \e^{i\Th}v>_{dx}-<D_\th\phi_S, \e^{i\Th}w>_{dx}=0.
\]
For the soliton part, we can compute
\[
 <i\psi_S, \phi_S>_{dx}=\int_{\mathbb{R}^3}-\rho\om f_\om^2 dx=-\om\|f_\om\|_{L^2}.
\]
Therefore, we can write
\begin{equation}
 \label{Qdecomp}
Q(t)=\om\|f_\om\|_{L^2}^2-<iw, v>_{dx}+<(p_\ep q_\ep
)^{-1}(g^\ep)^{0k}i\pa_k(\e^{i\Th}v),\e^{i\Th}v>_{dx}+Err(Q).
\end{equation}
The error term $Err(Q)$ satisfies \eqref{errq}.

\bigskip

Finally, we consider the main energy $\mathcal{H}(t)$, containing of quadratic terms and a higher order term corresponding to
the nonlinearity. The quadratic part can be
decomposed as follows
\begin{equation*}
\begin{split}
&\quad <\pa^k\phi, \pa_k\phi>_{d\si}-<\pa^t\phi,
\pa_t\phi>_{d\si}+m^2<\phi, \phi>_{d\si}\\
&=<(g^\ep)^{kl}\pa_k\phi,
\pa_l\phi>_{d\si}-<(g^\ep)^{00}\pa_t\phi, \pa_t\phi>_{d\si}+m^2<\phi_S+\e^{i\Th}(q_\ep d_\ep)^{-1}v,\phi_S+\e^{i\Th}(q_\ep d_\ep)^{-1}v>_{d\si}\\
&=\int_{\mathbb{R}^3}|\nabla_x\phi_S|^2+|\psi_S|^2+m^2|\phi_S|^2dx+m^2<q_\ep^{-2}v,v>_{dx}+
<q_\ep^{-2}(g^\ep)^{kl}\pa_k(\e^{i\Th}v),
\pa_l(\e^{i\Th}v)>_{dx}\\
&+2<\nabla_x\phi_S, \nabla_x(\e^{i\Th}v)>_{dx}+<w,
w>_{dx}+2<\e^{i\Th}w, \psi_S>_{dx}+2m^2<\phi_S,
\e^{i\Th}v>_{dx}+Err(Hq),
\end{split}
\end{equation*}
where
\begin{equation}
\label{errhq}
|Err(Hq)|\les(1+|\xi|^2)\ep^2+\ep\|w\|_{L^1}^2+\ep\|v\|_{H^1}^2.
\end{equation}
We expand the nonlinear term up to second order
\begin{align*}
-\int_{\mathbb{R}^3}N(\phi)d\si&=-\int_{\mathbb{R}^3}N(f_\om+(q_\ep
d_\ep)^{-1}v)-N(f_\om)-DN(f_\om)(q_\ep d_\ep
)^{-1}v-\f12(q_\ep d_\ep)^{-2}vD^2N(f_\om)vd\si\\
&\quad - \int_{\mathbb{R}^3}\frac{1}{p+1}f_\om^{p+1}+f_\om^p (q_\ep
d_\ep)^{-1}v_1+\f12 f_\om^{p-1}(q_\ep
d_\ep)^{-2}(|v|^2+(p-1)|v_1|^2)\quad d\si\\
&=- \int_{\mathbb{R}^3}\frac{1}{p+1}f_\om^{p+1}+f_\om^p
v_1+\f12f_\om^{p-1}(|v|^2+(p-1)|v_1|^2)\quad dx+Err(Hq),
\end{align*}
where by using Lemma \ref{lemnonlinear}, we can control the error term
\begin{equation}
 \label{errhn}
|Err(Hn)|\les
\|v\|_{H^1}^3+\|v\|_{H^1}^{p+1}+(1+|\xi|^{2})\ep^{2}\les
\|v\|_{H^1}^3+(1+|\xi|^2)\ep^2.
\end{equation}
Group these together. We end up with
\begin{align*}
\mathcal{H}(t)&=\f12\int_{\mathbb{R}^3}|\nabla_x\phi_S|^2+|\psi_S|^2+m^2f_\om^2-\frac{2}{p+1}f_\om^{p+1}
+m^2q_\ep^{-2}|v|^2
+<q_\ep^{-2}(g^\ep)^{kl}\pa_k(\e^{i\Th}v), \pa_l(\e^{i\Th}v)>dx\\
&+<\nabla_x\phi_S, \nabla_x(\e^{i\Th}v)>_{dx}+\f12<w,
w>_{dx}+<\e^{i\Th}w, \psi_S>_{dx}+m^2<\phi_S,
\e^{i\Th}v>_{dx}+Err(Hq)\\
& -<f_\om^p, v_1>_{dx}-\f12\int_{\mathbb{R}^3}f_\om^{p-1}|v|^2+(p-1)f_\om^{p-1}v_1^2 dx+Err(Hn).
\end{align*}
We can compute the soliton part
\begin{align*}
& \f12\int_{\mathbb{R}^3}|\nabla_x\phi_S|^2+|\psi_S|^2+m^2f_\om^2-\frac{2}{p+1}f_\om^{p+1}dx\\
=& \f12 \int_{\mathbb{R}^3}\rho^2\om^2|u|^2f_\om^2+|\nabla_z f\frac{\pa z}{\pa x}|^2+\rho^2\om^2 f_\om^2+\rho^2\om^2 f_\om^2+m^2f_\om^2-\frac{2}{p+1}f_\om^{p+1}dx\\
=& \f12\int_{\mathbb{R}^3}2\rho^2\om^2 f_\om^2+(m^2-\om^2)f_\om^2+(\frac{1}{3}+\frac{2}{3}\rho^2)|\nabla_z f|^2-\frac{2}{p+1}f_\om^{p+1}dx\\
=& \rho\left(\om^2\|f_\om\|_{L^2}^2+\frac{1}{3}\|\nabla_z f_\om\|_{L^2}^2\right)=\rho B.
\end{align*}
Using the identity \eqref{idenofphiS} and the orthogonality
condition \eqref{orthcondfix}, we can eliminate the crossing terms
\begin{align*}
& <\nabla_x \phi_S, \nabla_x(\e^{i\Th}v)>_{dx}+<\psi_S, \e^{i\Th}w>_{dx}+<m^2\phi_S, \e^{i\Th}v>_{dx}-<f_\om^p, v_1>_{dx}\\
=&<\psi_S, \e^{i\Th}w>_{dx}+<-\Delta_x \phi_S+m^2\phi_S-\e^{i\Th}f_\om^p,\e^{i\Th}v>\\
=&<D_\la \phi_S\cdot V(\la), \e^{i\Th}w>-<D_\la \psi_S \cdot V(\la), \e^{i\Th}v>=0.
\end{align*}
Combining all these together, we can write
\begin{align}
 \notag
\mathcal{H}(t)&=\rho B+\f12\int_{\mathbb{R}^3}m^2q_\ep^{-2}|v|^2
+<q_\ep^{-2}(g^\ep)^{kl}\pa_k(\e^{i\Th}v), \pa_l(\e^{i\Th}v)>+|w|^2-f_\om^{p-1}|v|^2\\
\label{Hdecomp}
&\qquad\qquad\qquad-(p-1)f_\om^{p-1}v_1^2 dx+Err(Hn)+Err(Hq),
\end{align}
where the errors terms satisfy \eqref{errhq}, \eqref{errhn}.

\bigskip

We proceed by arguing that the energy of the radiation term $(v, w)$ is positive definite. Using the Fermi coordinate system,
we have turned the stability of solitons on a slowly varying background into that on a small
perturbation of Minkowski space. In Minkowski space, observe that the associated quantities $\mathcal{H}$, $\Pi_k$, $Q$,
 as functionals of the solution $\phi$ of \eqref{EquationMink}, are conserved
and that the soliton $\phi_S$ is critical point of the Lagrange
\[
 \mathcal{H}-u^k\Pi_k-\frac{\om}{\rho}Q
\]
with Lagrange multipliers $u^k$, $\frac{\om}{\rho}$. To study the stability of stable solitons, we expand the above
Lagrange around $\phi_S$. It turns out that the soliton part is convex function of $\om$, $u$. The crossing term is vanishing since
 $\phi_S$ is critical point. We remark here that this is also given by the orthogonality condition \eqref{orthcond}, which
has been shown above. The quadratic part gives the energy of the remainder $\phi-\phi_S$, which is proven to be positive definite under
the orthogonality condition. One thus can obtain the orbital stability result of Theorem \ref{thmstbMin} in Minkowski space, see the
work of D. Stuart \cite{moduStuart}.
We use a similar idea in our case by considering
\[
 \mathcal{H}(t)-(u^k(0)+u_0^k)\cdot \Pi_k(t)-\frac{\om}{\rho(0)}Q(t).
\]
Having the decomposition formulae \eqref{Pidecomp}, \eqref{Qdecomp}, \eqref{Hdecomp}, we can group the soliton part
\[
 \rho B(1-(u(0)+u_0)\cdot (u+u_0))-\frac{\om^2}{\rho(0)}\|f_\om\|_{L^2(\mathbb{R}^3)}^2.
\]
We denote the quadratic part
\begin{align*}
E(t)&=\f12\int_{\mathbb{R}^3}m^2q_\ep^{-2}|v|^2+<q_\ep^{-2}(g^\ep)^{kl}\pa_k(\e^{i\Th}v), \pa_l(\e^{i\Th}v)>
+|w|^2-f_\om^{p-1}|v|^2\\
&-(p-1)f_\om^{p-1}v_1^2+2\frac{\om}{\rho}<iw, v>-2\frac{\om}{\rho(0)}<(p_\ep q_\ep)^{-1}(g^\ep)^{0l}i\pa_l(\e^{i\Th}v),
 \e^{i\Th}v>\\
&+2<\e^{i\Th}w, \pa_k(\e^{i\Th}v)>(u^k+u_0^k)-2(u^k(0)+u_0^k)<(p_\ep q_\ep)^{-1}(g^\ep)^{0l}\pa_l(\e^{i\Th}v),
 \pa_k(\e^{i\Th}v)>dx
\end{align*}
as the energy of the radiation term $(v, w)$. Since $|u(0)+u_0|<1$, $\rho(0)\geq 1$, $|\om|\leq m$, using Corollary \ref{corgposi}
and condition \eqref{psiep}, we can conclude that
\begin{align}
\label{lowE}
E(t)&\geq\f12 \int_{\mathbb{R}^3}m^2|v|^2+|\nabla_x(\e^{i\Th}v)|^2+|w|^2-f_\om^{p-1}|v|^2-
(p-1)f_\om^{p-1}v_1^2+2\frac{\om}{\rho}<iw, v>\\
\notag&\qquad+2<\e^{i\Th}w, \pa_k(\e^{i\Th}v)>(u^k+u_0^k) dx-C_0\ep\|v\|_{H^1}^2\\
\notag&=\frac{1}{2\rho}\left(\|w+\rho (u+u_0)\nabla_z v-i\rho \om
v\|_{L^2(dz)}^2+<v_1, L_{+}v_1>_{dz}+<v_2,
L_{-}v_2>_{dz}\right)-C_0\ep\|v\|_{H^1}^2\\
\notag
&=\frac{1}{2\rho}E_0(t)-C_0\ep\|v\|_{H^1}^2,
\end{align}
where the constant $C_0$ depends only on $h$, $m$. In the last line, we must view $v$, $w$ as functions of $z$ instead of $x$.
The operators $L_{+}$, $L_{-}$ are defined in \eqref{L+}. We hence can write
\begin{equation*}
\begin{split}
 \mathcal{H}(t)-(u^k(0)+u_0^k)\cdot \Pi_k(t)-\frac{\om}{\rho(0)}Q(t)=&\rho B(1-(u(0)+u_0)\cdot
 (u+u_0))-\frac{\om^2}{\rho(0)}\|f_\om\|_{L^2}^2\\
&+E(t)+Err_1+Err_2,
\end{split}
\end{equation*}
where $Err_1$ consists of errors in $\mathcal{H}(t)$, $\Pi_k(t)$, $Q(t)$ satisfying the estimate
\begin{align}
\notag
 |Err_1|&=\left|Err(Hn)+Err(Hq)-(u^k(0)+u_0^k)Err(\Pi_k)-\frac{\om}{\rho(0)}Err(Q)\right|\\
\label{err1}
&\les\|v\|_{H^1}^3+(1+|\xi|^2)\ep^2+\ep\|v\|_{H^1}^2+\ep\|w\|_{L^1}^2,
\end{align}
 $Err_2$ denotes the errors from the coefficients $u+u_0$, $\frac{\om}{\rho}$, obeying the estimates
\begin{align}
\notag
 |Err_2|&=\left|-(u^k-u^k(0))<\e^{i\Th}w, \pa_k(\e^{i\Th}v)>_{dx}-\left(\frac{\om}{\rho}-\frac{\om}{\rho(0)}\right)<iw, v>_{dx}\right|\\
\label{err2} &\les |u(t)-u(0)|\|w\|_{L^2}\|v\|_{H^1}.
\end{align}
Now denote
\begin{equation}
 \label{defdeltaH}
\begin{split}
d\mathcal{H}(t)&=\rho B(1-(u(0)+u_0)\cdot (u+u_0))
-\frac{\om^2}{\rho(0)}\|f_\om\|_{L^2}^2-\frac{B(0)}{\rho(0)}
+\frac{\om(0)\om}{\rho(0)}\|f_{\om}\|_{L^2}^2(0)
\end{split}
\end{equation}
and
\begin{equation*}
\begin{split}
Err_3=
\mathcal{H}(0)-(u^{k}(0)+u_0^k)\Pi_k(0)-\frac{B(0)}{\rho(0)}-\frac{\om}{\rho(0)}\left(Q(0)-\om(0)\|f_{\om}\|_{L^2}^2(0)\right).
\end{split}
\end{equation*}
We obtain
\begin{equation}
 \label{energycomb1}
\begin{split}
E(t)+d \mathcal{H}(t) =&\mathcal{H}(t)-\mathcal{H}(0)-(u^k(0)+u_0^k)\cdot (\Pi_k(t)-\Pi_k(0))\\
 &-\frac{\om}{\rho(0)}(Q(t)-Q(0))-Err_1-Err_2+Err_3.
\end{split}
\end{equation}
Since
\[
 \rho(0)B(0)-|u(0)+u_0|^2\rho(0)B(0)-\frac{B(0)}{\rho(0)}=0,
\]
by Lemma \ref{datapre}, we can show that
\begin{equation}
 \label{err3est}
|Err_3|\les \ep^2.
\end{equation}
Since $(v, w)$ satisfies the orthogonality condition
\eqref{orthcondfix}, Proposition \ref{positen} implies that $E_0(t)$
is equivalent to the energy
$\|v\|_{H^1(\mathbb{R}^3)}^2+\|w\|_{L^2(\mathbb{R}^3)}^2$ when $\ep$
is sufficiently small. Hence from \eqref{lowE}, $E(t)$ is equivalent
to $\|v\|_{H^1(\mathbb{R}^3)}^2+\|w\|_{L^2(\mathbb{R}^3)}^2$ up to
an error. The right hand side of \eqref{energycomb1} will be proven
to be small in the next section. To obtain estimates for
$\|v\|_{H^1(\mathbb{R}^3)}^2+\|w\|_{L^2(\mathbb{R}^3)}^2$, it
suffices to show that $d\mathcal{H}$ is positive definite. In fact,
we can show
\begin{lem}
 \label{lempositdH}
Let $\la(t)\in \La_{\delta_0}$. If $\delta_0$ is sufficiently small, depending only $\la(0)$, then
\[
d \mathcal{H}(t)\geq c(|\om(t)-\om(0)|^2+|u(t)-u(0)|^2)
\]
for some positive constant $c$ depending only on $\la(0)$.
\end{lem}
\begin{proof}
As functions of $u$, we can compute
\begin{align*}
 D_u (\rho(1-(u(0)+u_0)\cdot (u+u_0)))(0)=&D\rho(0)\rho(0)^{-2}-\rho(0) (u(0)+u_0)=0,\\
 D^2_u (\rho(1-(u(0)+u_0)\cdot (u+u_0)))(0)=&D^2\rho(0)
(1-|u(0)+u_0|^2)-2 D\rho(0) \cdot(
u(0)+u_0)\\
=&\rho(0)I+\rho(0)^{3}(u(0)+u_0)\cdot (u(0)+u_0),
\end{align*}
which implies that the Hessian of $\rho(1-u(0)\cdot u)$ with respect to $u$ is positive definite. Hence, if $\delta_0$ is
 sufficiently small, depending only on $\la(0)$, then
\[
 \rho(1-(u(0)+u_0)\cdot (u+u_0))\geq \rho(0)^{-1}+c_1|u(t)-u(0)|^2,\quad |u(t)-u(0)|\leq \delta_0
\]
for some constant $c_1$ depending only on $u(0)+u_0$.

For the other part on the right hand side of \eqref{defdeltaH}, we rely on the properties of the ground state $f_\om$.
Recall the definition \eqref{defofB} of $B(t)$ and the scaling property \eqref{scalling} as well as the energy identities,
 we can show that
\begin{align*}
& D_\om\left(B(t)-\om^2\|f_\om\|_{L^2}^2-B(0)+\om(0)\om\|f_\om\|_{L^2}^2(0)\right)(0)=0,\\
& D_\om^2\left(B(t)-\om^2\|f_\om\|_{L^2}^2-B(0)+\om(0)\om\|f_\om\|_{L^2}^2(0)\right)(0)\\
&\quad=\left(\frac{6-2p}{p-1}\om(0)^2-m^2\right)(m^2-\om(0)^2)^{\frac{9-5p}{2(p-1)}}\|f\|_{L^2}^2>0.
\end{align*}
The positivity is due to the assumption $\la(0)\in \La_{\delta_0}$. Hence if $\delta_0$ is sufficiently small, we can conclude that
\[
 B(t)-\om^2\|f_\om\|_{L^2}^2-B(0)+\om(0)\om\|f_\om\|_{L^2}^2(0)\geq c_2|\om(t) -\om(0)|^2,\forall \om(t)\in(\om_0-\delta_0,
 \om_0+\delta_0)
\]
for some constant $c_2$ depending only on $\la(0)$.

Therefore, we have shown that
\begin{align*}
 d\mathcal{H}(t)&\geq \rho(0)^{-1}B(t)+c_1|u(t)-u(0)|^2B(t)-\rho(0)^{-1}(\om^2\|f_\om\|_{L^2}^2+B(0)-\om(0)\om\|f_{\om}\|_{L^2}^2(0))\\
&\geq c_1|u(t)-u(0)|^2B(t)+c_2\rho(0)^{-1}|\om(t)-\om(0)|^2\geq c(|\om(t)-\om(0)|^2+|u(t)-u(0)|^2)
\end{align*}
for some constant c depending only on $\la(0)$ if $\delta_0$ is sufficiently small.
\end{proof}

We now choose $\delta_0$ such that Lemma \ref{lempositdH} holds and the set $\La_{\delta_0}$ defined in line \eqref{defofLadel} is
nonempty. By Proposition \ref{positen}, we have shown the left hand side of \eqref{energycomb1} is bounded below as follows
\[
 E(t)+d\mathcal{H}\geq c(\|v\|_{H^1}(t)+\|w\|_{L^2}(t)+|\om(t)-\om(0)|^2+|u(t)-u(0)|^2)
\]
for some constant $c$ depending only on $\la(0)$ if $\ep$ is sufficiently small.

\subsubsection{Energy Estimates}
We estimate the right hand side of \eqref{energycomb1} in this section by using the integral formulae
\eqref{EnergyC}, \eqref{MomentumC}, \eqref{ChargeC}. Denote the soliton part
\begin{equation*}
\mathcal{H}_S^{\a}(t)=\int_{0}^{t}\int_{\mathbb{R}^3}T^{\mu\nu}[\phi_S]\pi_{\mu\nu}^{\pa_{\a}}dxdt,\quad
\a\in\{0, 1, 2, 3\},
\end{equation*}
where $\pa_t\phi_S$ should be replaced with $\psi_S$.
\begin{prop}
 \label{propcon}
We have
\begin{align*}
|\mathcal{H}(t)-\mathcal{H}(0)+\mathcal{H}_S^{0}(t)|&\les \ep^{2}+\ep\int_{0}^{t}|u|^2+\|v\|_{H^1}^2+\|w\|_{L^2}^2ds,\\
|\Pi_k(t)-\Pi_k(0)-\mathcal{H}_S^k(t)|&\les\ep^{2}+\ep\int_{0}^{t}|u|^2+\|v\|_{H^1}^2+\|w\|_{L^2}^2ds,\\
|Q(t)-Q(0)|&\les\ep^{2}+\ep\int_{0}^{t}|u|^2+\|v\|_{H^1}^2+\|w\|_{L^2}^2ds
\end{align*}
if $\ep$ is sufficiently small.
\end{prop}

\begin{proof}
We decompose the energy momentum tensor $T_{\mu\nu}[\phi]$ as follows
\begin{equation*}
T_{\mu\nu}[\phi]=T_{\mu\nu}[\phi_S]+T_{\mu\nu}[\phi_S,
\phi-\phi_S]+QT_{\mu\nu}[\phi-\phi_S],
\end{equation*}
where
\begin{align*}
 T_{\mu\nu}[\phi_S, \phi-\phi_S]=&<\pa_\mu\phi_S, \pa_\nu(\phi-\phi_S)>+<\pa_\nu\phi_S, \pa_\mu(\phi-\phi_S)>\\
&-g^\ep_{\mu\nu}(<\pa^\ga\phi_S, \pa_\ga(\phi-\phi_S)>+<\mathcal{V}'(\phi_S), \phi-\phi_S>),
\end{align*}
$QT_{\mu\nu}[\phi-\phi_S]$ is at least quadratic in $v, w$. Again, $\pa_t\phi_S$ should be replaced with
$\psi_S$ in the above expressions as well as in the following argument of the proof. Recall the deformation tensor
$(\pi^{\pa_\a})^{\mu\nu}=-\f12\pa_\a (g^\ep)^{\mu\nu}$. Using Lemma \ref{lemnonlinear} to estimate the nonlinear terms
in $QT_{\mu\nu}[\phi-\phi_S]$, we can show that
\[
\left|\int_{0}^{t}\int_{\mathbb{R}^3}T_{\mu\nu}[\phi-\phi_S](\pi^{\pa_\a})^{\mu\nu}d\vol\right|\les \ep\int_{0}^{t}\|
v\|_{H^1}^2+\|w\|_{L^2}^2ds,\quad \forall \a\in\{0, 1, 2, 3\}.
\]
For the linear term $T_{\mu\nu}[\phi_S, \phi-\phi_S]$, we apply Lemma \ref{lemb} to get
\[
\left|\int_{0}^{t}\int_{\mathbb{R}^3}T_{\mu\nu}[\phi_S,\phi-\phi_S](\pi^{\pa_\a})^{\mu\nu}d\vol\right|\les
\ep^2\int_{0}^{t}(\|v\|_{H^1}+\|w\|_{L^2})(1+|\xi(s)|)ds,\quad \forall \a\in\{0, 1, 2 ,3\}.
\]
For the soliton part, first we can drop the volume factor $d_\ep^2$
as
\[
\left|\int_{0}^{t}\int_{\mathbb{R}^3}T_{\mu\nu}[\phi_S](\pi^{\pa_t})^{\mu\nu}(d_\ep^2-1)dxdt\right|\les
\ep^3\int_{0}^{t}1+|\xi(s)|^2ds.
\]
The above estimates rely on the unknown upper bound of $|\xi(t)|$. Although it is bounded by $C_2$ as a bootstrap assumption,
we do not want $C_2$ to appear in the following estimates. Otherwise, we need smallness on $t$ in order to close our bootstrap
argument. We will instead use Gronwall's inequality to control $|\xi|$. First, we use the relation $\dot{\xi}=u+\dot{\eta}$
together with Corollary \ref{corcontrga} to control $\xi$ in terms of $u$, $v$, $w$.
We can control $\dot \eta$ by Corollary \ref{corcontrga} as follows
\[
|\dot \eta|^2\leq |\dot{\ga}|^2\les (C_2\ep)^4+\|v\|_{H^1}^4\les \ep^3+\|v\|_{H^1}^2.
\]
If $C_2^4\ep<1$, then we can show that
\begin{equation}
\label{xiest}
 |\xi(t)|^2\les
|\xi(0)|^2+t\int_{0}^{t}|u|^2+|\dot\ga|^2ds\les
1+\ep^{-1}\int_{0}^{t}|u|^2+\|v\|_{H^1}^2ds,\quad \forall t\leq
T/\ep.
\end{equation}
Plug this into above estimates. We conclude that
\begin{align*}
|\mathcal{H}(t)-\mathcal{H}(0)+\mathcal{H}_S^{0}(t)|&\les
\ep^3\int_{0}^{t}1+|\xi|^2ds+\ep\int_{0}^{t}\|v\|_{H^1}^2+\|w\|_{L^2}^2ds\\
&\les \ep^2+\ep\int_{0}^{t}|u|^2+\|v\|_{H^1}^2+\|w\|_{L^2}^2ds.
\end{align*}
Similarly, we can show that
\begin{align*}
 \left|\Pi_k(t)-\Pi_k(0)-\mathcal{H}_{S}^k(t)-\int_{0}^{t}\int_{\mathbb{R}^3}(m_0)^{\mu\nu}\pa_\nu\b T^{m_0}_{\mu k}[\phi_S]dxds\right|&\les
 \ep^2+\ep\int_{0}^{t}|u|^2+\|v\|_{H^1}^2+\|w\|_{L^2}^2ds,\\
\left|Q(t)-Q(0)-\int_{0}^{t}\int_{\mathbb{R}^3}<i\pa_\mu
\b (m_0)^{\mu\nu}\pa_\nu\phi_S,
\phi_S>dxds\right|&\les\ep^2+\ep\int_{0}^{t}|u|^2+\|v\|_{H^1}^2+\|w\|_{L^2}^2ds,
\end{align*}
where we recall that $m_0$ is the Minkowski metric, $\b=p_\ep^{-1}q_\ep$ and we denote
$T_{\mu\nu}^{m_0}[\phi_S]$ as the energy momentum tensor associated to the Minkowski metric $m_0$.
We thus can conclude the proposition if we can
control the two integrals
\[
 \int_{0}^{t}\int_{\mathbb{R}^3}(m_0)^{\mu\nu}\pa_\nu\b T^{m_0}_{\mu k}[\phi_S]dxds,\quad
\int_{0}^{t}\int_{\mathbb{R}^3}<i\pa_\mu
\b (m_0)^{\mu\nu}\pa_\nu\phi_S,
\phi_S>dxds.
\]
The only smallness in the above integrals is contributed by $\pa\b$, which has size $\ep$. To prove that they are higher order
error terms, we have to exploit the properties of the solitons $\phi_S$. The key observation is that the solitons travel
in the direction $\pa_t +u_0^k\pa_k$.
More precisely, we compute the second integral
\begin{align*}
&<i\pa_\mu\b(m_0)^{\mu\nu}\pa_\nu\phi_S, \phi_S>=-\pa_t\b<i\psi_S,
\phi_S>+\pa_k\b<i\pa_k\phi_S, \phi_S>=\rho\om
f_{\om}^2(z)(\pa_t+(u+u_0)\nabla)\b.
\end{align*}
Observing that $z=A_u(x-\xi-u_0t)$, $f_{\om}(z)$ travels along the geodesic $(t, u_0t)$, thus integration by parts
may allow us to gain extra smallness. For any function $F(z, \la(t))$ independent of $\Th$, we have the identity
\begin{equation}
\label{patnabla}
\begin{split}
(\pa_t+(u+u_0)\nabla_x)\b_0\cdot F
 &=\pa_t(\b_0 F)-\b_0 D_\la F\cdot
(V(\la)+\dot{\ga})+(u+u_0)\nabla_x\b_0 \cdot F\\
&=\pa_t\left(\b_0F\right)-\b_0 D_\la F \dot\ga-\b_0(u+u_0)D_\xi F+(u+u_0)\nabla_x\b_0 \cdot F\\
&=\pa_t\left(\b_0 F\right)-\b_0 D_\la F\dot\ga+(u+u_0)D_x(\b_0 F),
\end{split}
\end{equation}
where we have used $D_\xi F= -D_x F$ as $z=A_u(x-\xi-u_0t)$. This key observation \eqref{patnabla} is of particular
 importance in this paper. It allows us to prove Theorem \ref{thmbaby} under the sharp condition $q>1$.
Now let $\b_0=\b-1$, $F=\rho \om f_{\om}^2(z)$. The inequality \eqref{patnabla} yields the estimate
\begin{align*}
\left|\int_{0}^{t}\int_{\mathbb{R}^3}<i\pa_\mu \b
(m_0)^{\mu\nu}\pa_\nu\phi_S,
\phi_S>dxds\right|&\les\ep^2(1+|\xi(t)|^2)+\int_{0}^{t}|\dot\ga|\ep^2(1+|\xi(s)|^2)ds\\
&\les \ep^2(1+\ep^{-1}\int_{0}^t|u|^2+\|v\|_{H^1}^2ds)+C_2^2\ep^2\int_{0}^{t}(C_2\ep)^2+\|v\|_{H^1}^2ds\\
&\les\ep^2+\ep\int_{0}^{t}|u|^2+\|v\|_{H^1}^2+ \|w\|_{L^2}^2ds
\end{align*}
for all $t\leq T/\ep$ by Corollary \ref{corcontrga} and the bootstrap assumption \eqref{baxi}, $C_2^4\ep<1$.
We thus conclude the first and the third inequality of this proposition.

\bigskip

To show the second inequality of this proposition, it suffices to estimate the first integral above. Similarly, we can compute
\begin{align*}
&(m_0)^{\mu\nu}\pa_\nu\b T_{\mu k}^{m_0}[\phi_S]=(m_0)^{\mu\nu}\pa_\nu\left((\b-1)T_{\mu k}^{m_0}[\phi_S]\right)
-(\b-1)(m_0)^{\mu\nu}\pa_\nu T_{\mu k}^{m_0}[\phi_S]\\
&\quad=(m_0)^{\mu\nu}\pa_\nu\left((\b-1)T_{\mu k}^{m_0}[\phi_S]\right)-(\b-1)<-\pa_t\psi_S+\Delta_x\phi_S-m^2\phi_S+|\phi_S|^{p-1}
\phi_S,\pa_k\phi_S>\\
&\quad =(m_0)^{\mu\nu}\pa_\nu\left((\b-1)T_{\mu k}^{m_0}[\phi_S]\right)+(\b-1)<D_\la\psi_S \cdot \dot\ga, \pa_k\phi_S>
\end{align*}
by using the identity \eqref{idenofphiS} and the relation $\dot \la=V(\la)+\dot \ga$. Hence according to Lemma \ref{lemb}, we can
show that
\begin{align*}
 \left|\int_{0}^{t}\int_{\mathbb{R}^3}(m_0)^{\mu\nu}\pa_\nu\b T^{m_0}_{\mu k}[\phi_S]dxds\right|
&\les\ep^2(1+|\xi(t)|^2)+\int_{0}^{t}|\dot\ga|\ep^2(1+|\xi(s)|^2)ds\\
&\les \ep^2+\ep\int_{0}^{t}|u|^2+\|v\|_{H^1}^2+ \|w\|_{L^2}^2ds.
\end{align*}
We hence have proven the proposition.
\end{proof}

This proposition allows us to control the right hand side of \eqref{energycomb1}.
\begin{cor}
\label{corrightb}We have
\begin{align*}
&\left|\mathcal{H}(t)-\mathcal{H}(0)-(u^k(0)+u_0^k)\cdot
(\Pi_k(t)-\Pi_k(0))-\frac{\om}{\rho(0)}(Q(t)-Q(0))\right|\\&\les\ep^2+\ep\int_{0}^{t}|u|^2+\|v\|_{H^1}^2+\|w\|_{L^2}^2ds.
\end{align*}
\end{cor}
\begin{proof}
Notice that $T^{\mu\nu}[\phi_S]$ is a function of $(z, \la(t))$ and is independent of $\Th$. By applying
\eqref{patnabla}, we can show that
\begin{align*}
\left|\int_{0}^{t}\int_{\mathbb{R}^3}T^{\mu\nu}[\phi_S](\pa_t+(u+u_0)\nabla_x)(g^\ep)_{\mu\nu}dxds\right|
\les\ep^2+\ep\int_{0}^{t}|u|^2+\|v\|_{H^1}^2+\|w\|_{L^2}^2ds.
\end{align*}
Since $\pi_{\mu\nu}^{\pa_\a}=\f12\pa_\a(g^\ep)_{\mu\nu}$, we have
\begin{align*}
\left|\mathcal{H}_S^{0}(t)+(u(0)+u_0)^k\mathcal{H}_S^{k}(t)\right|&\les
\left|\int_{0}^{t}\int_{\mathbb{R}^3}T^{\mu\nu}[\phi_S](\pa_t+(u+u_0)\nabla_x)(g^\ep)_{\mu\nu}dxds\right|\\
&\quad +\left|\int_{0}^{t}\int_{\mathbb{R}^3}T^{\mu\nu}[\phi_S](u(s)-u(0))^k
\pa_k(g^\ep)_{\mu\nu}dxds\right|\\
&\les
\ep^2+\ep\int_{0}^{t}|u|^2+\|v\|_{H^1}^2+\|w\|_{L^2}^2ds+\ep^2\int_{0}^{t}(1+|\xi|)(|u|+\ep)ds\\
&\les\ep^2+\ep\int_{0}^{t}|u|^2+\|v\|_{H^1}^2+\|w\|_{L^2}^2ds,
\end{align*}
where we have used $|u(s)-u(0)|\les |u(s)|+\ep$ by Lemma
\ref{datapre}. Then the corollary follows from Proposition \ref{propcon}.
\end{proof}

\subsubsection{Proof of Proposition \ref{propH1est} and Theorem \ref{thmfix}}
We are now able to improve the bootstrap assumptions and to conclude Proposition \ref{propH1est} and Theorem \ref{thmfix}.
Denote
\[
 \mathcal{E}(t):=| u|^2+|\om(t)-\om(0)|^2+\|w\|_{L^2}^2+\|v\|_{H^1}^2.
\]
According to  Proposition \ref{positen} and Lemma \ref{lempositdH}, we have
\[
E(t)+d\mathcal{H}(t)\geq c(1-C_0\ep)\mathcal{E}(t)
\]
for some positive constants $c$, $C_0$ independent of $\ep$, $C_2$. Hence by \eqref{energycomb1}, \eqref{err1},
\eqref{err2}, \eqref{err3est} and Corollary \ref{corrightb}, for sufficiently small $\ep$,
we can show that
\begin{equation*}
\begin{split}
 \mathcal{E}(t)&\les \ep^2(1+|\xi|^2)+\ep\int_{0}^{t}\mathcal{E}(s)ds+\mathcal{E}(t)^{\frac{3}{2}}+\ep \mathcal{E}(t)\\
&\les \ep^2+\ep\int_{0}^{t}\mathcal{E}(s)ds+\mathcal{E}(t)^{\frac{3}{2}}+\ep \mathcal{E}(t),
\end{split}
\end{equation*}
where we have used inequality \eqref{xiest} to control $|\xi|^2$.
Since the implicit constant is independent of $\ep$ and initially
$\mathcal{E}(0)\les \ep^2$, using Gronwall's inequality, we have
\[
\mathcal{E}(t)\les \ep^2, \quad \forall t\leq T/\ep.
\]
Let $C_3$ be the universal implicit constant
appeared before. By our notation, $C_3$ depends on $h$, $m$,
$\la_0$, $T$ and is independent of $\ep$, $C_2$. In particular, we have
\[
\mathcal{E}(t)=|u(t)|^2+| \om(t)-\om(0)|^2+\|v\|_{H^1}^2(t)+\|w\|_{L^2}^2(t)\leq
C_3 \ep^2,\quad \forall t\leq T/\ep.
\]
By Corollary \ref{corcontrga}, this implies that
\begin{equation*}
 \begin{split}
 |\xi(t)|&\leq |\xi(0)|+\int_{0}^{t}|u|+|\dot{\ga}|ds\\
&\leq C_3\ep+C_3\int_{0}^{t}C_3^{\f12}\ep+C_2^2\ep^2+\|v\|_{H^1}^2ds\\
&\leq C_3\ep+C_3^{\frac{3}{2}}T+C_3 T+C_3^2\ep T,
 \end{split}
\end{equation*}
where we let $\ep$ to be small such that $C_2^4\ep<1$. Take
\begin{equation*}
 C_2=C_3+C_3^{\frac{3}{2}}T+C_3 T+2C_3^2 T.
\end{equation*}
If
\[
 \ep\leq \min\{C_2^{-4}, \frac{1}{10} C_2^{-1}\delta_1, 1\},
\]
then for all $t\leq T/\ep$, we have
\begin{equation*}
 \begin{split}
  &\|v\|_{H^1}+\|w\|_{L^2}\leq \mathcal{E}(t)^{\f12}\leq C_2\ep\leq \f12 \delta_1,\\
&|\xi|\leq C_0\ep+C_0^{\frac{3}{2}}T+C_0 T+C_0^2\ep T \leq C_2.
 \end{split}
\end{equation*}
This improves the bootstrap assumptions \eqref{bawv},
\eqref{baxi}. Therefore we can conclude that Proposition \ref{propH1est} follows from Corollary \ref{corcontrga}.

For Theorem \ref{corcontrga}, notice that there is a unique short time solution $\phi(t, x)$ on $[0, t^*)\times\mathbb{R}^3$. Then
Lemma \ref{nondegD} implies that the modulation equations \eqref{modeq} admits a local solution $\la(t)$, $t\in[0, t^{**}]
\subset [0, t^*)$. Proposition \ref{propH1est} then shows that $\phi(t, x)$, $\la(t)$ satisfy \eqref{obsphi}, \eqref{obsla} with
a constant $C$ independent of $\ep$. Since Proposition \ref{propH1est} holds for any $t^{**}\leq T/\ep$, we can conclude that
the solution $\phi(t, x)$, $\la(t)$ can be uniquely extended to $[0, T/\ep]\times\mathbb{R}^3$ and satisfy estimates \eqref{obsphi},
\eqref{obsla}.

\subsection{Estimates of Higher Sobolev Norms}
In the previous section, we have proven the orbital stability of stable solitons of
equation \eqref{FEQUATION} on a slowly varying background in the energy space
$H^1\times L^2$. To solve the full Einstein equations \eqref{EQUATION} and to obtain a $C^1$ spacetime
$([0, t^*]\times\mathbb{R}^3, g)$, we need higher Sobolev estimates for the matter field $\phi$.
In this section, we prove Proposition \ref{prophigherSobnorm} and Corollary \ref{corhighSobnorm}.

Since we already have a solution $\phi(t, x)$ and a curve $\la(t)$ satisfying estimates \eqref{obsphi}, \eqref{obsla}, we use
energy estimates to obtain the higher Sobolev estimates by considering the equation of the remainder $v$.
However, to simply the argument in the sequel and to avoid taking fourth order derivative of $\ga(t)$, we choose a modified curve
$\tilde{\la}(t)\in \La_{\textnormal{stab}}$, defined as the integral
curve of $V(\la)$, that is
\begin{equation*}
\pa_t\tilde{\la}=V(\la),\quad
\tilde{\la}(0)=(\om_0, \th_0, 0, u_0).
\end{equation*}
Using this modified curve $\tilde{\la}(t)$, we decompose the solution $\phi$ as follows
\begin{equation}
\label{decompmod} \phi(t,
x)=\phi_S(\tilde{\la};x)+\e^{i\Th(\tilde{\la})}\tilde{v}.
\end{equation}
Then we claim that Proposition \ref{prophigherSobnorm} is reduced to the following estimates.
\begin{prop}
 \label{prophighsobred}
Assume the initial data $\phi_0\in H^3$, $\phi_1\in H^2$. Let $\ep_1$ be defined in Proposition \ref{prophigherSobnorm}.
Then
\[
 \sum\limits_{|s|\leq 3}\|\pa^s \tilde{v}(t, \cdot)\|_{L^2(\mathbb{R}^3)}\les \max\{\ep, \ep_1\},
\quad \forall t\leq T/\ep.
\]
\end{prop}
In fact, if Proposition \ref{prophighsobred} holds, observing that
\begin{align*}
|\la-\tilde{\la}|&\les\int_{0}^{t}|\dot\ga|ds+|\la(0)-\tilde{\la}(0)|\les \ep^2 t+\ep\les \ep,\quad \forall t\leq \frac{T}{\ep},\\
 \pa_t \Th(\tilde{\la})&=\frac{\om}{\rho}+\rho_0\om_0 u_0(u_0+u),\quad \nabla_x\Th(\tilde{\la})=-\rho_0\om_0 u_0,
\end{align*}
then we have
\begin{align*}
 \|\phi-\phi_S(x;\la)\|_{H^3}&=\|\phi_S(x;\tilde{\la})+\e^{i\Th(\tilde{\la})}\tilde{v}-\phi_S(x;\la)\|_{H^3}\\
&\les|\la(t)-\tilde{\la}(t)|+\| \tilde{v}\|_{H^3}\les \max\{\ep, \ep_1\}.
\end{align*}
This partially explains Proposition \ref{prophigherSobnorm} as $\|\pa_t^s(\phi-\phi_S(x;\la))\|_{L^2}$
requires taking higher order derivatives of the modulation curve $\la(t)$. We will estimate the higher derivatives of $\la(t)$
in next section.

Although we have modified the curve $\la(t)$, the remainder $\tilde{v}$ is still small in $H^1$.
\begin{lem}
\label{lemsmallrad}
We have
\[
\|\tilde{v}\|_{H^1}+\|\pa_t\tilde{v}\|_{L^2}\les \ep,\quad \forall t\in[0, T/\ep].
\]
\end{lem}
\begin{proof}
Since $\|v\|_{H^1}+\|w\|_{L^2}\les \ep$, by using the decomposition \eqref{decomp}, we can show that
\begin{align*}
\|\e^{i\Th(\tilde{\la})}\tilde{v}\|_{H^1}&=\|\phi_S(\la;x)+\e^{i\Th(\la)}(q_\ep
d_\ep)^{-1}v-\phi_S(\tilde{\la};x)\|_{H^1}\les \ep,\\
\|\e^{i\Th(\tilde{\la})}\pa_t
\tilde{v}\|_{L^2}&=\|\pa_t(\phi-\phi_S(\tilde{\la};x))-i\pa_t\Th(\tilde{\la})\e^{i\Th(\tilde{\la})}\tilde{v}\|_{L^2}\\
&=\|D_\la\phi_S(\la;x)\cdot V(\la)+\e^{i\Th}(p_\ep
d_\ep)^{-1}w-D_\la\phi_S(\tilde{\la};x)\cdot V(\la)-i\pa_t\Th(\tilde{\la})\e^{i\Th(\tilde{\la})}\tilde{v}\|_{L^2}\les
\ep.
\end{align*}
\end{proof}
\begin{remark}
The reason that we consider the modified curve $\tilde{\la}$ is that
we must avoid taking the fourth derivative of $\la$. Otherwise, we have to take third derivative of the
 nonlinearity $|\phi|^{p-1}\phi$, which is impossible
since $p$ is assumed to be less than 3. Using the modified
curve guarantees that when we differentiate the equation of the
radiation term $\tilde{v}$ twice, we only need third derivative of $\la$. We
remark here that $\tilde{\la}$ still depends on $\la$.
\end{remark}

\subsubsection{Estimates for Higher Derivatives of $\la(t)$}

Since $\la(t)$ satisfies the modulation equations \eqref{modeq}, by differentiating the equations, we are able to derive
estimates for derivatives of $\la(t)$. It turns out that we first have to estimate the derivatives of the nonlinearity
$\mathcal{N}(\la)$.
\begin{lem}
\label{lemnonlest} Let $\mathcal{N}(\la)$ be defined in line \eqref{defnonlin}. Assume $p\geq 2$. Then for
any vector field $Y$ on $[0, T/\ep]\times\mathbb{R}^3$, we have
\begin{align*}
|Y\mathcal{N}(\la)|\les &(|v|+|v|^{p-1})(|Y
f_\om|+|Y v|+|v|),\\
 |Y^2\mathcal{N}(\la)|\les&(1+\|v\|_{L^\infty}+|Y \ln
f_\om|)(|Y f_\om|^2+|Y
v|^2+|v|^2)+(|v|+|v|^{p-1})(|Y^2 f_\om|+|Y^2v|)\\
 &+|Y^2d_\ep|(|v|^2+|v|^{p}).
\end{align*}
\end{lem}
\begin{proof}
It follows by direct calculations and the properties of $f_\om$
summarized in Theorem \ref{propoff}.
\end{proof}
Using this lemma together with the modulation equations \eqref{modeq}, we are
able to estimate the higher order derivatives of the modulation curve $\la(t)$.
\begin{prop}
\label{hiorderga}
 Let $\ga(t)$ be defined in line \eqref{lagaV}. Assume $\ga(t)$
satisfies the modulation equations \eqref{modeq}. Then we have
\begin{align*}
&|\ddot{\ga}|\les\ep^2,\\
 &|\pa_t^3\ga|\les\ep^2(1+\|v\|_{L^\infty})+\ep\|D_\la\phi_S X^2v\|_{L^2}.
\end{align*}
\end{prop}
\begin{proof}
 Differentiate the modulation equations \eqref{modeq}, we get the ODE
for $\dot{\ga}$
\begin{equation*}
(D+D_1+D_2)\ddot{\ga}+\pa_t(D+D_1+D_2)\dot{\ga}=\pa_t F(t;\la(t)).
\end{equation*}
Since $|\dot \la(t)|\les 1$, $|\pa g^\ep|\les 1$, we can show that
\[
 |\pa_t D|+|\pa_t D_1|\les |\dot \la|+|\pa g^\ep|\les 1
\]
according to the definition of $D$, $D_1$ given in line \eqref{defofD}. For $\pa_t D_2$, we have to use the equation of $\phi$.
In fact, by \eqref{FEQUATION}, \eqref{defofH}, we can show that
\begin{align*}
|\pa_t D_2|&\les |\dot \la|\cdot \|w\|_{L^2}+|<D_\la^2\phi_S, \pa_t(\e^{i\Th}w)>_{dx}|+
|\dot \la|\cdot \|v\|_{L^2}+|<D_{\la}^2\phi_S, \pa_t\left(a_0(\phi-\phi_S)\right)>_{dx}|\\
& \les 1+|<D_{\la}^2\phi_S,
(\phi_{tt}-\pa_t\psi_S)a_1>_{dx}|+|<D_{\la}^2\psi_S, a_0(\pa_t\phi-D_{\la}\phi_S (V(\la)+\dot \ga))>_{dx}|\\
&\les 1+|<D_\la^2\phi_S, H(t, x)+\Delta \phi-b(m^2\phi-|\phi|^{p-1}\phi)>_{dx}|
+|<D_{\la}^2\psi_S, a_0(\phi_t-\phi_S-D_{\la}\phi_S\dot\ga)>_{dx}\\
&\les 1+|<\pa_k(a^{k\mu}D_\la^2\phi_S), \pa_\mu\phi>_{dx}|
+|<\nabla_x D_\la^2\phi_S, \nabla_x \phi>_{dx}|\\
&\les 1+\|v\|_{H^1}+\|w\|_{L^2}\les 1.
\end{align*}
We hence have shown
\[
 |\pa_t(D+D_1+D_2)|\les 1.
\]
Then Lemma \ref{nondegD} implies that estimates $|\ddot{\ga}|\les \ep^2$ follow if we can show that
\begin{equation}
\label{paF}
|\pa_t F(t;\la(t))|\les \ep^2.
\end{equation}
By the definition \eqref{defofF} of $F(t;\la(t))$, it suffices to estimate
\[
<bD_\la\phi_S, \e^{i\Th}\pa_t\mathcal{N}(\la)>_{dx},\quad
<\dot a_1 D_\la\phi_S, \phi_{tt}>_{dx}, \quad <\ddot{a_0}D_\la\psi_S, \phi-\phi_S>_{dx}.
\]
All the other terms can be estimated similarly to $F(t;\la(t))$ in Lemma \ref{lemDF}.
We first consider the nonlinear term, which is in fact the main term in $\pa_t F(t;\la(t))$ as other terms are errors from
the slowly varying metric $g^\ep$.
The key observation is that for vector field $X=\pa_t+u_0\nabla_x$ and any $C^1$ function $F_1$, we have
\begin{equation}
\label{Xfom}
\begin{split}
 |X F_1(f_{\om}(z))|\les|(\pa_t+u_0 \nabla_x)f_\om(z)|&=|D_\om f_\om \dot \om +\nabla_z f_\om \dot z+u_0\nabla_z f_\om A_u|\\
&\les |\dot \ga|+|\nabla_z f_\om (A_u u_0+A_u(-\dot\xi-u_0))|\\
&\les \ep^2+|u|\les \ep
\end{split}
\end{equation}
by using Theorem \ref{thmfix} as well as the relation $\dot \xi=u+\dot\eta$. Hence
by Lemma \ref{lemnonlest} and Lemma \ref{lemnonlinear}, we can show
\begin{align}
\notag
|<bD_\la\phi_S, \pa_t\mathcal{N}(\la)>_{dx}|&\les
|<bD_\la\phi_S, X
\mathcal{N}(\la)>_{dx}|+|<bD_\la\phi_S, u_0\nabla_x\mathcal{N}(\la)>_{dx}|\\
\notag
&\les\|D_\la\phi_S(|v|+|v|^{p-1})(|X f_\om|+|X
v|+|v|)\|_{L^1}+|<u_0\nabla_x(bD_\la\phi_S), \mathcal{N}(\la)>_{dx}|\\
\notag
&\les(\|D_\la\phi_S \cdot X v\|_{L^2} +\ep)(\|v\|_{L^2}+\|v\|_{H^1}^{p-1})+\|v\|_{H^1}^2+\|v\|_{H^1}^p\\
\label{phispatN}
&\les \ep^2+\ep \|D_\la\phi_S \pa_t v\|_{L^2}\les
\ep^2+\|D_\la\phi_S \pa_t(\e^{i\Th}v)\|_{L^2}+\|D_\la\phi_S\dot\Th v\|_{L^2} \\
\notag
&\les\ep^2+\ep\|w\|_{L^2}\les \ep^2.
\end{align}
We must remark here that using the decomposition \eqref{decomp} we have
\[
 |\pa_t v|\les |w|+|\dot \Th v|.
\]
Although $\dot\Th$ depends on $z$, $|\dot \Th D_\la\phi_S|$ decays
exponentially in $z$ by Theorem \ref{propoff}.

For the second term $<\dot a_1 D_\la\phi_S, \phi_{tt}>_{dx}$, we use equations \eqref{FEQUATION}, \eqref{defofH} and then use integration by parts.
We can bound
\begin{align*}
 |<\dot a_1 D_\la\phi_S, \phi_{tt}>_{dx}|&=|<\pa_t\ln a_1 D_\la\phi_S, a^{k\mu}\pa_{\mu k}\phi+b^\mu\pa_\mu\phi+\Delta\phi
-b(m^2\phi-|\phi|^{p-1}\phi)>_{dx}|\\
&\les\ep^2+|<\pa_k(a^{k\mu}\pa_t\ln a_1 D_\la\phi_S), \pa_\mu\phi>_{dx}|+|<\nabla_x(\pa_t\ln a_1 D_\la\phi_S), \nabla_x\phi>_{dx}|\\
&\les\ep^2+\|\pa^2 g^\ep D_\la\phi_S\|_{L^2}\les \ep^2+\|\pa^2(g^\ep-h^\ep)\|_{L^2}+\|\pa^2 h^\ep D_\la\phi_S\|_{L^2}\les \ep^2
\end{align*}
by using the assumption \eqref{bapsi} and the fact that $h^\ep(t,x)=h(\ep t, \ep x)\in C^2$.

The third term $<\ddot{a_0}D_\la\psi_S, \phi-\phi_S>_{dx}$ can be estimated similarly. We can show that
\[
 |<\ddot{a_0}D_\la\psi_S, \phi-\phi_S>_{dx}|\les \ep^2+\|\pa^2 g^\ep D_\la\phi_S\|_{L^2}\les \ep^2.
\]
Therefore, we have shown \eqref{paF}. Hence $|\ddot\ga|\les \ep^2$.

\bigskip

Having proven $|\ddot\ga|\les \ep^2$, to estimate of the third order derivative of $\ga$, differentiate the modulation
equations \eqref{modeq} twice
\[
(D+D_1+D_2)\pa_{t}^3
\ga+2\pa_t(D+D_1+D_2)\ddot{\ga}+\pa_{tt}(D+D_1+D_2)\dot{\ga}=\pa_{tt}F(t;\la(t)).
\]
Similarly to estimating $\pa_t F(t;\la(t))$ carried out above, we can show that
\[
\left|\pa_{tt}(D+D_1+D_2)\right|\les |\ddot \la|+|\dot\la|+\|\pa^2(g^\ep-h^\ep)\|_{L^2}\les1.
\]
It hence suffices to estimate $\pa_{tt}F(t;\la(t))$. The strategy is similar to that of $\pa_t F(t;\la(t))$. The main
term is that with derivative hitting on the nonlinearity
$<\e^{-i\Th}bD_\la\phi_S, \pa_{tt}\mathcal{N}(\la)>$, which by Lemma
\ref{lemnonlest} and by inequalities \eqref{Xfom}, \eqref{phispatN}, can be estimated as follows
\begin{align*}
&\left|<\e^{-i\Th}bD_\la\phi_S, \pa_{tt}\mathcal{N}(\la)>_{dx}\right|\les
\left|<\e^{-i\Th}bD_\la\phi_S,
\pa_{tt}\mathcal{N}(\la)>_{dx}\right|\\
&\les|<\e^{-i\Th}bD_\la\phi_S, X^2\mathcal{N}(\la)>_{dx}|+|<\e^{-i\Th}bD_\la\phi_S, (-(u_0\nabla_x)^2+2u_0\nabla_x X)\mathcal{N}(\la)>_{dx}|\\
&\les \ep^2(1+\|v\|_{L^\infty})+\ep\|D_\la\phi_S X^2v\|_{L^2}+\|D_\la\phi_S \pa^2 g^\ep (|v|^2+|v|^p)\|_{L^1}\\
&\quad +|<(u_0\nabla_x)^2(\e^{-i\Th}bD_\la\phi_S), \mathcal{N}(\la)>_{dx}|+|<u_0\nabla_x (\e^{-i\Th}bD_\la\phi_S),X
\mathcal{N}(\la)>_{dx}|\\
&\les \ep^2(1+\|v\|_{L^\infty})+\ep\|D_\la\phi_S X^2v\|_{L^2}+\|\pa^{2}(g^\ep-h^\ep)\|_{L^2}\|v\|_{H^1}^2\\
&\les \ep^2(1+\|v\|_{L^\infty})+\ep\|D_\la\phi_S X^2v\|_{L^2}.
\end{align*}
For those terms when the derivative hits on $\pa_t\phi$, we rely on the equation \eqref{FEQUATION} together with the identity
 \eqref{defofH} and then use integration by parts to pass the derivative to the metric $g^\ep$. We hence can show that
\begin{align*}
|\pa_{tt}F(t;\la(t))|&\les \ep^2(1+\|v\|_{L^\infty})+\ep\|D_\la\phi_S X^2v\|_{L^2}+
\||\pa^3(g^\ep)|(|D_\la\phi_S|+|D_\la\psi_S|)\|_{L^1}\\
&\les\ep^2(1+\|v\|_{L^\infty})+\ep\|D_\la\phi_S X^2v\|_{L^2}.
\end{align*}
Then Lemma \ref{nondegD} yields the estimate
\[
|\pa_{t}^3\ga|\les \ep^2(1+\|v\|_{L^\infty})+\ep\|D_\la\phi_S
X^2v\|_{L^2}.
\]
\end{proof}

\subsubsection{Linearized Equation for $\tilde{v}$ and Energy Estimates}
Using the modified curve $\tilde{\la}(t)$ and the corresponding decomposition \eqref{decompmod}, we can find the equation for
$\tilde{v}$
\begin{equation}
\label{eqtildev} L_\ep\tilde{v}+\mathcal{N}(\tilde{\la})+\tilde{F}=0,
\end{equation}
where
\begin{align*}
\tilde{F}=\e^{-i\Th(\tilde{\la})}(\Box_{g^\ep}\phi_S(\tilde{\la};x)-m^2\phi_S+|\phi_S|^{p-1}\phi_S)+i\Box_{g^\ep}\Th\cdot \tilde{v}.
\end{align*}
For any complex function $v(t, x)=v_1(t, x)+iv_2(t, x)$, the linear operator $L_\ep$ is defined as follows
\begin{equation}
\label{defofLep}
L_{\ep}v=\Box_{g^\ep}v+A(\tilde\ga) v+2i\pa_{\mu}\Th\cdot\pa^\mu
v-m^2v+f_{\om_0}^{p-1}(z)v+(p-1)f_{\om_0}^{p-1}(z)v_1
\end{equation}
 with
\[
A(\tilde\ga)=-(g^\ep)^{\mu\nu}\pa_\mu\Th(\tilde\ga)\pa_\nu\Th(\tilde\ga)=\e^{-i\Th}\Box_{g^\ep}\e^{i\Th}-i\Box_{g^\ep}\Th.
\]
We have the following energy estimates for the linear operator $L_\ep$.
\begin{lem}
\label{lemenergyL}
For all $t\leq T/\ep$, we have
\begin{equation*}
\|\pa v\|_{L^2}(t)^2\les \|\pa v(0, x)\|_{L^2}^2+\|v(0,
x)\|_{L^2}^2+\ep^{-1}\int_{0}^{t}\|L_\ep v\|_{L^2}^2(s)ds+\sup\limits_{0\leq
s\leq t}\|v\|_{L^2}^2.
\end{equation*}
\end{lem}
\begin{proof}
Recall the energy momentum tensor $\tilde{T}_{\mu\nu}[v]$ for the operator $\Box_{g^\ep}$
\[
 \tilde{T}_{\mu\nu}[v]=<\pa_\mu v,\pa_\nu v>-\f12 g^\ep_{\mu\nu}<\pa^\ga v, \pa_\ga v>.
\]
For any vector field $Y$, we have the identity
\[
 D^\mu(\tilde{T}_{\mu\nu}[v]Y^\nu)=\tilde{T}^{\mu\nu}[v]\pi^Y_{\mu\nu}+<\Box_{g^\ep}v, Y(\phi)>.
\]
Take $Y=X=\pa_t+u_0^k\pa_k$. Then integrate on the region
$[0, t]\times \mathbb{R}^3$. We obtain
\begin{align*}
\int_{\mathbb{R}^3}\tilde{T}_{\mu\nu}[v]X^\mu n^\nu
d\si(t)=\int_{\mathbb{R}^3}\tilde{T}_{\mu\nu}[v]X^\mu n^\nu d\si(0)+
\int_{0}^{t}\int_{\mathbb{R}^3}
\tilde{T}^{\mu\nu}[v]\pi^X_{\mu\nu}+<\Box_{g^\ep}v, X(v)>d\vol.
\end{align*}
By replacing $\Box_{g^\ep}v$ with $L_\ep v$, we must estimate the other terms respectively. First, consider the term
$<2i\pa_\mu\Th\cdot \pa^\mu v, X(v)>_{d\si}$. We use integration by parts. We can write
\begin{align*}
&<2i\pa_\mu\Th\cdot\pa^\mu v, X(v)>d_\ep^2=d_\ep^2\pa^\mu\Th<2i\pa_\mu v, X v>=d_\ep^2\pa^\mu\Th\left(\pa_\mu<iv, Xv>-X<iv, \pa_\mu v>\right)\\
&=\pa_\mu\left(d_\ep^2\pa^\mu\Th<iv, Xv>\right)-X\left(d_\ep^2\pa^\mu\Th<iv, \pa_\mu v>\right)-\pa_\mu(d_\ep^2\pa^\mu\Th)
<iv, Xv>+X(d_\ep^2\pa^\mu\Th)<iv, \pa_\mu v>.
\end{align*}
Recall the definition of $\tilde{\la}$. We can compute
\begin{align*}
 \pa_t\Th&=\frac{\om}{\rho}+\rho_0 \om_0(|u_0|^2+u_0 u)=\rho_0\om_0+\rho_0\om_0 u_0 u+\frac{\om}{\rho}-\frac{\om_0}{\rho_0},\\
\nabla_x \Th&=-\rho_0\om_0u_0.
\end{align*}
Therefore
\begin{align*}
& \left|\int_{0}^{t}\int_{\mathbb{R}^3}<2i\pa_\mu\Th\cdot \pa^\mu v, X v>d\vol\right|= \left|\int_{0}^{t}\int_{\mathbb{R}^3}
<2i\pa_\mu\Th\cdot \pa^\mu v, X v>d_\ep^2dxds\right|\\
& \les\left|\left.\int_{\mathbb{R}^3}<iv, \pa^0\Th X v-\pa^\mu\Th \pa_\mu v>d\si\right|_{0}^{t}\right|+
\int_{0}^{t}\int_{\mathbb{R}^3}|\pa(d_\ep^2\pa^\mu\Th) v\pa v|d\vol\\
&\les \|v\|_{L^2}(t)\|\pa v\|_{L^2}(t)+\|v\|_{L^2}(0)\|\pa v\|_{L^2}(0)+\ep\int_{0}^{t}\|v\|_{L^2}^2+\|\pa v\|_{L^2}^2ds.
\end{align*}
For the other terms, notice that for any real function $F_1$, we have
\begin{align*}
2\int_{0}^{t}\int_{\mathbb{R}^3}<F_1 v,
X(v)>d\vol&=\int_{0}^{t}\int_{\mathbb{R}^3}F_1
X(|v|^2)d_\ep^2 dxdt\\
&=\left.\int_{\mathbb{R}^3}F_1
|v|^2d\si\right|_{0}^{t}-\int_{0}^{t}\int_{\mathbb{R}^3}X(F_1
d_\ep^2)|v|^2dxdt.
\end{align*}
When $F_1=f_{\om_0}^{p-1}$, since $|X F_1|\les \ep$ by \eqref{Xfom}, we have
\begin{align*}
\|X(f_\om^{p-1} d_\ep^2)
|v|^2\|_{L^1(\mathbb{R}^3)}\les \ep\|v^2\|_{L^1}\les\ep\|v\|_{H^1}^2.
\end{align*}
When $F_1=A(\tilde\ga)$, since $|\pa\Th|\les 1$, $|\pa^2\Th|\les |\dot \ga|\les \ep^2$, we have
\begin{align*}
\|X(A(\ga) d_\ep^2)v^2\|_{L^1}=\|X((g^\ep)^{\b\nu}\pa_\b\Th\pa_\nu\Th d_\ep^2)|v|^2
\|_{L^1}\les\ep\|v^2\|_{L^1}\les \ep\|v\|_{L^2}^2.
\end{align*}
Combine all these together. We have shown that
\begin{align*}
&\left|\int_{0}^{t}\int_{\mathbb{R}^3}<
A(\ga)v-m^2v+f_{\om_0}^{p-1}(z)v+(p-1)f_{\om_0}^{p-1}(z)v_1,
X(v)>d\vol\right|\\
&\les \|v\|_{L^2}(t)\|v\|_{H^1}(t)+\|v\|_{L^2}(0)\|v\|_{H^1}(0)+\ep\int_{0}^{t}\|v\|_{H^1}^2ds.
\end{align*}
Now, since the unit normal vector field
$(-(h^\ep)^{00})^{-\f12}(h^\ep)^{0\mu}\pa_\mu$ to the hypersurface $\mathbb{R}^3$ as well as the vector field $X=\pa_t+u_0^k\pa_k$
 are timelike with respect to
the metric $h^\ep$, we can conclude that the unit normal vector field $n=(-(g^\ep)^{00})^{-\f12}(g^\ep)^{0\mu}\pa_\mu$
together with the vector field $X=\pa_t+u_0^k\pa_k$ are also timelike for the metric $g^\ep$ if $\ep$ is sufficiently
small. Therefore, there exists a positive constant $c$, depending only $h, u_0$, such that
\[
 T_{\mu\nu}[v]X^\mu n^\nu\geq c|\pa v|^2.
\]
Since $|\pi_{\mu\nu}^X|\les \ep$, the energy identity then implies that
\begin{align*}
\|\pa v\|_{L^2}^2(t)&\les \|\pa v(0, x)\|_{L^2}^2+\|v(0,
x)\|_{L^2}^2+\ep\int_{0}^{t}\|\pa
v\|_{L^2}^2(s)ds+\ep^{-1}\int_{0}^{t}\|L_\ep v\|_{L^2}^2ds\\
&+\|v\|_{L^2}^2+\ep\int_{0}^{t}\|v\|_{H^1}^2(s)ds,\quad \forall t\leq T/\ep.
\end{align*}
The Lemma then follows by using Gronwall's inequality.
\end{proof}

\subsubsection{$H^2$ Estimates}
Commuting the equation \eqref{eqtildev} with the vector field $X=\pa_t+u_0^k\pa_k$, we obtain
\begin{equation}
\label{eqpatv} L_\ep X\tilde{v}+[X,
L_\ep]\tilde{v}+X\mathcal{N}(\tilde{v})+X\tilde{F}=0.
\end{equation}
Since we have computed
\[
|\pa \Th|\les 1,\quad |\pa^2\Th|\les |\dot\ga|\les\ep^2, \quad |\pa A(\tilde\ga)|\les \ep,\quad |Xf_{\om}|\les \ep,
\]
using Lemma \ref{lemsmallrad}, we can estimate the commutator
\begin{align*}
\|[X, L_\ep]\tilde{v}\|_{L^2}&\les \|[X,
\Box_{g^\ep}]\tilde{v}\|_{L^2}+\|X\pa^\mu\Th \cdot \pa_\mu
\tilde{v}\|_{L^2}+\|X
f_{\om_0}^{p-1}(\tilde{v}+(p-1)\tilde{v}_1)\|_{L^2}+\|X A(\ga) \tilde{v}\|_{L^2}\\
&\les \ep\|\pa^2 \tilde{v}\|_{L^2}+\|\pa^2(g^\ep)\pa
\tilde{v}\|_{L^2}+\ep(\|\pa\tilde v\|_{L^2}+\|\tilde v\|_{L^2})\\
&\les
\ep\|\pa^2\tilde{v}\|_{L^2}+\|\pa^2(g^\ep-h^\ep)\|_{L^6}\|\pa\tilde{v}\|_{L^3}+\ep^2\\
&\les\ep\|\pa^2\tilde{v}\|_{L^2}+\|\pa^2\psi^\ep\|_{H^1}\|\pa\tilde{v}\|_{H^1}+\ep^2\\
&\les
\ep\|\pa^2\tilde{v}\|_{L^2}+\ep^2\|\pa\tilde{v}\|_{H^1}+\ep^2\les \ep\|\pa^2\tilde{v}\|_{L^2}+\ep^2.
\end{align*}
For the nonlinearity $X\mathcal{N}(\tilde \la)$, Lemma
\ref{lemnonlest} yields the estimates
\begin{align*}
\|X\mathcal{N}(\tilde{\la})\|_{L^2}&\les\|(|\tilde{v}|+|\tilde{v}|^{p-1})(|X
f_{\om_0}|+|X \tilde{v}|)\|_{L^2}\\
&\les\||\tilde{v}|+|\tilde{v}|^{p-1}\|_{L^3}(\ep+\|X\tilde{v}\|_{L^6})\\
&\les
(\|\tilde{v}\|_{H^1}+\|\tilde{v}\|_{H^1}^{p-1})(\ep+\|\pa^2\tilde{v}\|_{L^2})\les \ep^2+\ep\|\pa^2\tilde{v}\|_{L^2}.
\end{align*}
As for $X\tilde{F}$, first using the identity \eqref{idenofphiS} we have
\[
 \Box{\phi_S(x;\tilde{\la})}-m^2\phi_S(x;\tilde{\la})+|\phi_S|^{p-1}\phi_S=V(\tilde{\la}) D_\la^2 \phi_S
 V(\tilde{\la})-V(\la)D_\la^2\phi_S V(\la)-D_\la\phi_S\cdot \pa_t V(\la).
\]
By Theorem \ref{thmfix}, we can show that
\[
 |V(\la)-V(\tilde{\la})|=|(0, \frac{\om}{\rho}, u+u_0, 0)-(0, \frac{\om_0}{\rho_0}, u_0, 0)|\les \ep.
\]
The key observation is that for all $k\leq 4$, we have
\[
 \|X \nabla^k f_\om\|_{L^2}\les\ep,
\]
which can be proven similarly to \eqref{Xfom}.
In particular, we have
\[
 \|X(\e^{-i\Th}D_\la^2\phi_S)\|_{L^2}\les\ep.
\]
Therefore, by Proposition \ref{hiorderga}, we can estimate $X\tilde F$ as follows
\begin{align*}
\|X \tilde{F}\|_{L^2}&\les \|X\left(V(\tilde{\la})\e^{-i\Th} D_\la^2 \phi_S
 V(\tilde{\la})-V(\la)\e^{-i\Th}D_\la^2\phi_S V(\la)-\e^{-i\Th}D_\la\phi_S\cdot \pa_t V(\la)\right)\|_{L^2}\\
&\qquad +\|X(\e^{-i\Th}(\Box_{g^\ep}-\Box)\phi_S)\|_{L^2}+\|X(\Box_{g^\ep}\Th\cdot \tilde{v})\|_{L^2}\\
 &\les |\dot \ga|+|\ddot{\ga}|+\ep\|X(\e^{-i\Th}D_\la^2\phi_S)\|_{L^2}+\|\pa^2(g^\ep)(|\nabla^2 f_{\om_0}|+|\nabla
f_{\om_0}|+f_{\om_0})\|_{L^2}\\
&\qquad+\|(g^\ep-m_0)(|\nabla^3 f_{\om_0}|+|\nabla^2 f_{\om_0}|+|\nabla f_{\om_0}|+f_{\om_0})\|_{L^2}+\|\pa^2(g^\ep)\tilde{v}\|_{L^2}+\|\pa(g^\ep)\pa\tilde{v}\|_{L^2}\\
&\les
\ep^2+\|\pa^2\psi^\ep\|_{H^1}\|\tilde{v}\|_{H^1}+\ep\|\pa\tilde{v}\|_{L^2}\les
\ep^2.
\end{align*}
The energy estimate Lemma \ref{lemenergyL} then implies that
\begin{align}
\label{energyH20}
 \|\pa X\tilde{v}\|_{L^2}^2&\les \max\{\ep^2,
\ep_1^2\}+\ep^{-1}\int_{0}^{t}\|[L_\ep,
X]\tilde{v}\|_{L^2}^2+\|X\mathcal{N}(\tilde{\la})\|_{L^2}^2+\|X\tilde{F}\|_{L^2}^2ds\\
\notag
 &\les \max\{\ep^2,
\ep_1^2\}+\ep\int_{0}^{t}\|\pa^2\tilde{v}\|_{L^2}^2ds
\end{align}
for all $t\leq T/\ep$, where $\ep_1$ is the size of the
initial data given in Proposition \ref{prophigherSobnorm}.

\bigskip

To derive the full estimates for $\|\pa^2 \tilde{v}\|_{L^2}$, merely commuting the equation with
 the vector field $X$ is not sufficient. The key point to obtain estimates \eqref{energyH20}
is that the soliton $\phi_S$ travels along the timelike geodesic $(t, u_0t)$ or quantitatively the vector field
 $X=\pa_t+u_0\nabla_x$ acting on $f_{\om}$ leads to the estimates $|Xf_{\om}|\les \ep$. For
general vector field, estimates \eqref{energyH20} may not hold. To retrieve the full estimates $\|\pa^2 \tilde{v}\|_{L^2}$,
we rely on the following elliptic estimates.
\begin{lem}
\label{lemellipt} Let $A^{ij}(x)\in C^{\a}(\mathbb{R}^3)$ for some
positive constant $0<\a<1$. Assume that $A^{ij}$ is uniformly
elliptic. That is $\exists K$ such that
\[
K^{-1}|y|^2\leq A^{ij}(x)y_iy_j\leq K|y|^2,\quad \forall x, y\in
\mathbb{R}^3.
\]
Then there exists a constant $C$ such that
\[
\|\phi\|_{H^2}\leq
C\|A\|_{C^{\a}}^2(\|A^{ij}\pa_{ij}\phi\|_{L^2}+\|\phi\|_{L^2})
\]
for any $\phi\in H^2(\mathbb{R}^3)$. Here
$\|A\|_{C^{\a}}=\sup\limits_{i,j}\|A^{ij}\|_{C^{\a}}$.
\end{lem}
\begin{proof}
Let $\chi$ be a cut-off function supported in the ball $B_2$ with
radius 2 and equal to 1 in the unit ball $B_{1}$. Then
the elliptic estimates show that
\[
\|\phi\|_{{H^2}(B_{1})}\leq
C\|A\|_{C^\a}(\|A^{ij}\pa_{ij}(\chi\phi)\|_{L^2}+\|\phi\|_{{L^2}(B_2)})\leq
C\|A\|_{C^\a}(\|A^{ij}\pa_{ij}\phi\|_{L^2(B_2)}+\|\phi\|_{H^1(B_2)})
\]
for some constant $C$ independent of $\phi$, $\chi$. The above
estimate holds for any ball $B_2$. We cover the whole space
$\mathbb{R}^3$ with radius 1 balls such that every point is covered
for at most 10 times. Add all the estimates, we conclude that
\[
\|\phi\|_{H^2}\leq
C\|A\|_{C^\a}(\|A^{ij}\pa_{ij}\phi\|_{L^2}+\|\phi\|_{H^1}).
\]
Interpolating between $H^2$ and $L^2$, we have
\[
\|\phi\|_{H^1}\leq
C\|\phi\|_{H^2}^{\f12}\|\phi\|_{L^2}^{\f12}\leq\f12
C^{-1}\|A\|_{C^\a}^{-1}\|\phi\|_{H^2}+2C\|A\|_{C^\a}\|\phi\|_{L^2}.
\]
Plug this into the above inequality. We get
\[
\|\phi\|_{H^2}\leq
C\|A\|_{C^\a}^2(\|A^{ij}\pa_{ij}\phi\|_{L^2}+\|\phi\|_{L^2}).
\]
\end{proof}
Using this lemma and estimates \eqref{energyH20}, we are able
obtain the $H^2$ estimates of $\tilde{v}$. First, write the wave operator $\Box_{g^\ep}$ as follows
\[
\Box_{g^\ep}\tilde{v}=((g^\ep)^{kl}+(g^\ep)^{00}u_0^k u_0^l-2(g^\ep)^{0k}u_0^l)\pa_{kl}\tilde{v}+((g^\ep)^{00}(\pa_t-u_0\nabla)+
2(g^\ep)^{0k}\pa_k)X\tilde{v}+d_\ep^{-2}
\pa_{\mu}(d_\ep^2 (g^\ep)^{\mu\nu})\pa_{\nu}\tilde{v}.
\]
Since $X=\pa_t+u_0^k\pa_k$ is timelike with respect to the metric $g^\ep$, we conclude that the $3\times 3$ matrix
\[
 (g^\ep)^{kl}+(g^\ep)^{00}u_0^k u_0^l-2(g^\ep)^{0k}u_0^l
\]
is positive definite. That is there exists a constant $K$, depending only on $h, u_0$ such that
\[
K^{-1}|y|^2\leq\sum\limits_{1\leq k, l\leq 3} y_k y_l\left((g^\ep)^{kl}+(g^\ep)^{00}u_0^k u_0^l-2(g^\ep)^{0k}u_0^l\right)
\leq K |y|^2,\quad \forall t\leq T/\ep, \quad x, y\in\mathbb{R}^3.
\]
Hence by Lemma \ref{lemellipt} and equation \eqref{eqtildev}, we can show that
\begin{align*}
\|\tilde{v}\|_{H^2}\les \ep+\|\pa X\tilde{v}\|_{L^2}.
\end{align*}
Then inequality \eqref{energyH20} implies that
\[
\|\pa X\tilde{v}\|_{L^2}^2\les \max\{\ep^2,
\ep_1^2\}+\ep\int_{0}^{t}\ep^2+\|\pa X\tilde{v}\|_{L^2}^2ds\les\max\{\ep_1^2,\ep^2\}+\ep\int_{0}^{t}\|\pa X\tilde{v}\|_{L^2}^2ds,\quad
\forall t\leq T/\ep.
\]
By using Gronwall's inequality, we conclude that
\begin{equation}
\label{energyH2} \|\pa^2\tilde{v}\|_{L^2}^2\les
\max\{\ep^2,\ep_1^2\}+\|\pa X\tilde{v}\|_{L^2}^2\les \max\{\ep^2, \ep_1^2\},\quad \forall t\leq T/\ep.
\end{equation}

\subsubsection{$H^3$ Estimates}
Having the $H^2$ estimates, we first can improve the estimate of
$\pa_t^3 \ga$ obtained in Proposition \ref{hiorderga}. According to the decomposition \eqref{decomp} corresponding to
the curve $\la(t)$, we can show that
\begin{align*}
\|v\|_{L^\infty}&\les\|v\|_{H^2}\les\|a_0\e^{-i\Th(\la)}(\phi_S(\la;x)-\phi_S(\tilde{\la};x)-\e^{i\Th(\tilde{\la})}\tilde{v})\|_{H^2}\\
&\les|\la(t)-\tilde{\la}(t)|+\|\tilde{v}\|_{H^2}+\|\pa^2 a_0 \tilde{v}\|_{L^2}]\\
&\les\max\{\ep,\ep_1\}+\|\tilde{v}\|_{H^1}\|\pa^2 a_0\|_{H^1}\les \max\{\ep,\ep_1\}.
\end{align*}
We also need to estimate $X^2v$, which, as having pointed out
previously, does not follow directly from the estimates of
$\tilde{v}$. However, notice that
\begin{align*}
 &\|D_\la\phi_S X^2(a_0 \e^{-i\Th(\la)})\|_{L^2}\les 1+\|X^2 a_0\|_{L^2}\les 1,\\
& \|\phi_S(\la;x)-\phi_S(\tilde{\la};x)-\e^{i\Th(\tilde{\la})}\tilde{v}\|_{L^\infty}\les \|\tilde{v}\|_{L^\infty}+|\la(t)-
\tilde{\la}(t)|\les \max\{\ep, \ep_1\}.
\end{align*}
By the decompositions \eqref{decomp},
\eqref{decompmod}, we can show that
\begin{align*}
 \|D_\la\phi_S X^2 v\|_{L^2}&\les\|D_\la\phi_S X^2\left(a_0\e^{-i\Th}(\phi_S(\la;x)-\phi_S(\tilde{\la};x)
-\e^{i\Th(\la)(\tilde{\la})}\tilde{v})\right)\|_{L^2}\\
&\les \max\{\ep, \ep_1\}+|\la(t)-\tilde{\la}(t)|+|\dot \ga|+|\ddot\ga|\les\max\{\ep, \ep_1\}.
\end{align*}
Thus Proposition \ref{hiorderga} implies that
\begin{equation}
\label{gadottte} |\pa_t^3 \ga|\les
\ep^2(1+\|v\|_{L^\infty})+\ep\|D_\la\phi_S X^2v\|_{L^2}\les
\ep\max\{\ep, \ep_1\}.
\end{equation}

\bigskip

We proceed to estimate the $H^3$ norm of $\tilde{v}$. Commute the equation \eqref{eqpatv} with
 the vector field $X=\pa_t+u_0^k\pa_k$ again. We have the equation for $X^2\tilde{v}$
\begin{equation}
\label{eqpattv} L_\ep X^2\tilde{v}+[X^2,
\Box_{g^\ep}]\tilde{v}+X^2\mathcal{N}(\tilde{\la})+X^2\tilde{F}=0.
\end{equation}
By Lemma \ref{lemnonlest}, we can estimate the nonlinearity
\begin{align*}
\|X^2\mathcal{N}(\tilde{\la})\|_{L^2}&\les\ep^2+\||X\tilde{v}|^2+|\tilde{
v}|^2+(|\tilde{v}|+|\tilde{v}|^{p-1})(|\dot{\ga}|+|X^2 \tilde{v}|)+|X^2a_0|(|\tilde v|^2+|\tilde{v}|^p)\|_{L^2}\\
&\les \ep^2+\|X\tilde{v}\|_{L^\infty}\|X\tilde{v}\|_{L^2}+\|\tilde{v}\|_{L^4}^2+\ep\|X^2\tilde{v}\|_{L^6}+\|X^2a_0\|_{H^1}
\|\tilde{v}\|_{H^1}^2\\
&\les\ep^2+\ep\|X\tilde{v}\|_{H^2}+\|\tilde{v}\|_{H^1}^2+ \ep\|X^2\tilde{v}\|_{H^1}\\
&\les\ep\max\{\ep, \ep_1\}+\ep\|\pa^2 X\tilde{v}\|_{L^2}.
\end{align*}
For the commutator, using \eqref{Xfom}, we can show that
\begin{align*}
\|[X^2, L_\ep]\tilde{v}\|_{L^2}&\les\|[X^2,
\Box_{g^\ep}]\tilde{v}\|_{L^2}+\|[X^2, A(\ga)]
\tilde{v}\|_{L^2}+\|[X^2, \pa^\mu\Th]\pa_\mu\tilde{v}\|_{L^2}+\|[X^2, f_{\om_0}^{p-1}]\tilde{v}\|_{L^2}\\
&\les  \|\pa(g^\ep)\pa^2 X\tilde{v}\|_{L^2}+\|\pa^2(g^\ep)\pa^2\tilde{v}\|_{L^2}+\|\pa^3(g^\ep)\pa\tilde{v}\|_{L^2}
+\|\pa^2 (g^{\ep})\tilde{v}\|_{L^2}\\
&\quad+\|\pa^2(g^\ep)\pa\tilde{v}\|_{L^2}+\|X^2 f_{\om_0}^{p-1}\tilde{v}\|_{L^2}
+\|Xf_{\om_0}^{p-1}X\tilde{v}\|_{L^2}\\
&\les
\ep\|\pa^2 X\tilde{v}\|_{L^2}+\|\pa^2(g^\ep-h^\ep)\|_{L^6}\||\pa^2\tilde{v}|+|\pa\tilde{v}|+|\tilde{v}|\|_{L^3}+
\|\pa^3(g^\ep-h^\ep)\|_{L^2}
\|\pa\tilde{v}\|_{L^\infty}+\ep^2\\
&\les
\ep\|\pa^2 X\tilde{v}\|_{L^2}+\|\pa^2\psi^\ep\|_{H^1}(\ep+\|\pa^2\tilde{v}\|_{H^1})
+\|\pa^3\psi^\ep\|_{L^2}\|\pa\tilde{v}\|_{H^2}+\ep^2\\
&\les\ep\|\pa^2 X\tilde{v}\|_{L^2}+\ep^2\|\pa^3\tilde{v}\|_{L^2}+\ep^2.
\end{align*}
Using the improved estimate \eqref{gadottte}, similarly to $X\tilde{F}$, we can estimate $X^2\tilde{F}$ as follows
\begin{align*}
\|X^2\tilde{F}\|_{L^2}&\les\ep^2+|\pa_t^3\ga|+\|\pa^2g^\ep \pa^2\phi_S\|_{L^2}+\|\pa^3g^\ep\pa\phi_S\|_{L^2}+\|\pa^2g^\ep(|\pa\tilde{v}|+|\tilde v|)\|_{L^2}+\|\pa^3g^\ep \tilde{v}\|_{L^2}\\
&\les\ep\max\{\ep,
\ep_1\}+\|\pa^3(g^\ep-h^\ep)\|_{L^2}+\|\pa^2(g^\ep-h^\ep)\|_{H^1}\\
&\les\ep\max\{\ep, \ep_1\}.
\end{align*}
Thus apply Lemma \ref{lemenergyL} to equation \eqref{eqpattv}. We get
\begin{equation*}
\|\pa X^2\tilde{v}\|_{L^2}^2\les \max\{\ep^2,
\ep_1^2\}+\ep^{-1}\int_{0}^{t}\ep^2\|\pa^2X\tilde{v}\|_{L^2}^2+\ep^4\|\pa^3\tilde{v}\|_{L^2}^2ds
\les\max\{\ep^2,\ep_1^2\}+\ep\int_{0}^{t}\|\pa^3\tilde{v}\|_{L^2}^2ds.
\end{equation*}
To retrieve the full estimates $\|\pa^3 \tilde{v}\|_{L^2}$, we apply Lemma \ref{lemellipt} to the equation
\[
L_\ep \pa\tilde{v}+[\pa,
L_\ep]\tilde{v}+\pa\mathcal{N}(\tilde{\la})+\pa \tilde{F}=0.
\]
We can show that
\begin{align*}
\|\pa\tilde{v}\|_{H^2}&\les\max\{\ep,\ep_1\}+\|\pa^2 X\tilde{v}\|_{L^2}+\|[\pa,
L_\ep]\tilde{v}+\pa\mathcal{N}(\tilde{\la})+\pa \tilde{F}\|_{L^2}\\
&\les\max\{\ep,\ep_1\}+\|\pa^2X\tilde{v}\|_{L^2}+\|(|\tilde{v}|+|\tilde{v}|^{p-1})(|\pa
f_{\om_0}|+|\pa\tilde{v}|)\|_{L^2}\\
&\les \max\{\ep,\ep_1\}+\|\pa^2X\tilde{v}\|_{L^2}.
\end{align*}
In particular, we have
\begin{align*}
\|X\tilde{v}\|_{H^2}\les
\max\{\ep,\ep_1\}+\|\pa X^2\tilde{v}\|_{L^2}.
\end{align*}
Therefore from the estimates for $\pa X^2\tilde{v}$ we have obtained above, we can show that
\begin{align*}
\|\pa X^2\tilde{v}\|_{L^2}^2&\les
\max\{\ep^2,\ep_1^2\}+\ep\int_{0}^{t}\|\pa^3\tilde{v}\|_{L^2}^2ds\les
\max\{\ep^2,\ep_1^2\}+\ep\int_{0}^{t}\|\pa^2 X\tilde{v}\|_{L^2}^2ds\\
&\les\max\{\ep^2,\ep_1^2\}+\ep\int_{0}^{t}\|\pa X^2\tilde{v}\|_{L^2}^2ds,\quad \forall t\leq T/\ep.
\end{align*}
Then Gronwall's inequality implies that
\begin{equation*}
\|\pa^3\tilde{v}\|_{L^2}\les\max\{\ep,\ep_1\}+\|\pa X^2\tilde{v}\|_{L^2}\les\max\{\ep,\ep_1\},\quad
\forall t\in[0, T/\ep].
\end{equation*}
This together with estimates \eqref{energyH2} proves Proposition \ref{prophighsobred}.

Finally, the estimates \eqref{gadottte} imply that
\begin{align*}
 \|\pa^s(\phi-\phi_S(x;\la(t)))\|_{L^2}&\les \|\pa^s(\phi_S(x;\tilde{\la}(t))+\e^{i\Th(\tilde{\la})}-\phi_S(x;\la(t)))\|_{L^2}\\
&\les \max\{\ep, \ep_1\}+|\pa_t^2\ga|\les \max\{\ep, \ep_1\}
\end{align*}
for all $|s|\leq 3$. Hence we have finished proving Proposition \ref{prophigherSobnorm}.

\section{Proof of the Main Theorem}
We use bootstrap argument to prove the main Theorem \ref{thmbaby}. Using the Fermi coordinate system, we have shown the existence of
solution $\phi$ of equation \eqref{FEQUATION} as well as its properties in Theorem \ref{thmfix} and Proposition \ref{prophigherSobnorm}
under the assumption \eqref{bapsi}, which could be viewed as a bootstrap assumption for the full reduced Einstein equations
\eqref{EQUATIONRED}. We consider the equations
of $\psi^\ep=g^\ep-h^\ep$ to improve this bootstrap assumption and thus to conclude the main Theorem \ref{thmbaby}.

\subsection{Estimates of the Metric $g^\ep$}
Let $(g^\ep, \phi)$ be a solution of the system \eqref{EQUATIONRED} with initial data satisfying conditions \eqref{inicondgphi},
\eqref{inicondgf} on the space $([0, T/\ep]\times\mathbb{R}^3, h^\ep)$. We have shown in Lemma \ref{lemredeq}
that the difference $\psi^\ep=g^\ep-h^\ep$ satisfies the
following hyperbolic system
\begin{equation*}
-(g^\ep)^{\a\b}\pa_{\a\b}\psi^\ep_{\mu\nu}+\delta P_{\mu\nu}+\delta Z_{\mu\nu}+\delta Q_{\mu\nu}
=2\delta^2(T_{\mu\nu}-\f12
tr T\cdot g^\ep_{\mu\nu}),
\end{equation*}
where $\delta Q_{\mu\nu}$, $\delta Z_{\mu\nu}$, $\delta P_{\mu\nu}$ are given in \eqref{defofPQ}.
 We show in this subsection that
\begin{prop}
 \label{proppsiep}
If $\psi^\ep=g^\ep-h^\ep$ satisfies condition \eqref{bapsi}, then
\[
 \|\pa^{s+1}\psi^\ep\|_{L^2}(t)\les\delta^2,\quad \forall t\leq T/\ep, |s|\leq 2.
\]
\end{prop}
The key observation that allows $\delta\leq \ep^q$, $q>1$ is based on the fact that  the energy momentum tensor $T_{\mu\nu}[\phi]$ splits
 into soliton part, which travels along the timelike geodesic $(t, u_0 t)$, and the error term which is small by
Proposition \ref{prophigherSobnorm}. When doing energy estimate, we
multiply the equations by $X\psi^\ep=(\pa_t+u_0\nabla_x)\psi^\ep$.
By using integration by parts, we can pass the derivative $X$ to the
soliton part of $T_{\mu\nu}[\phi]$. This, according to \eqref{Xfom},
allows us to prove Proposition \ref{proppsiep} for all $\delta \leq
\ep^q$, $q>1$.
\begin{proof}
Since the initial data $(\phi_0, \phi_1)$ satisfy condition
\eqref{inicondgphi}, we conclude according to Theorem \ref{thmfix}
and Proposition \ref{prophigherSobnorm} that $\phi$ decomposes as
\eqref{decompmod} associated to
 the modified curve $\tilde{\la}(t)$ such that the remainder $\tilde{v}$ satisfies the estimates
\[
 \|\pa^s\tilde{v}\|_{L^2}(t)\les \ep,\quad \forall |s|\leq 3, \quad t\leq T/\ep.
\]
Using the modified decomposition \eqref{decompmod}, we can write
\[
 T_{\mu\nu}-\f12 tr T \cdot g^\ep_{\mu\nu}=<\pa_\mu\phi, \pa_\nu\phi>+\mathcal{V}(\phi)g^\ep_{\mu\nu}=T_{\mu\nu}^S+T_{\mu\nu}^R
\]
with the soliton part given by
\begin{align*}
 T_{\mu\nu}^S&=<\pa_\mu\phi_S(x;\tilde{\la}(t)), \pa_\nu\phi_S(x;\tilde{\la}(t))>+\mathcal{V}(\phi_S(x;\tilde{\la}(t)))g^\ep_{\mu\nu}\\
&=\pa_\mu f_{\om_0} \pa_\nu f_{\om_0}+\pa_\mu\Th(\tilde{\la}) \pa_\nu\Th(\tilde{\la}) f_{\om_0}^2
 +\mathcal{V}(f_{\om_0})g_{\mu\nu}^\ep.
\end{align*}
The error term $T_{\mu\nu}^R$ is small by Proposition \ref{prophighsobred}. In fact, we can show that
\begin{align*}
 \sum\limits_{|s|\leq2}\|\pa^s T_{\mu\nu}^R\|_{L^2}(t)&\les \ep+\sum\limits_{|s|\leq 2}\|\pa^{s+1}\tilde{v}
\|_{L^2}+\|\pa^s(\mathcal{V}(\phi)-\mathcal{V}(\phi_S))\|_{L^2}+\|\pa^2g^\ep(\mathcal{V}(\phi)-\mathcal{V}(\phi_S))\|_{L^2}\\
&\les\ep+\|\pa^2(g^\ep-h^\ep)\|_{L^2}+\|\mathcal{V}^{''}(\phi)-\mathcal{V}^{''}(\phi_S)\|_{L^2}
+\|\mathcal{V}^{'}(\phi)-\mathcal{V}^{'}(\phi_S)\|_{L^2}\\
&\les \ep+\|\phi-\phi_S\|_{L^2}\les \ep, \quad \forall t\leq T/\ep.
\end{align*}
Here we recall that $\mathcal{V}(\phi)$ is given in line \eqref{defofmV} and $p\geq 2$.

Since $X=\pa_t +u_0^k\pa_k$ is timelike, apply formula \eqref{enerestfor} with $Y=X,\b=1$ to the above hyperbolic system
for $\psi_{\mu\nu}^\ep$ commuting with the vector field $\pa^s$. We obtain the energy estimates
\begin{equation*}
 \begin{split}
\|\pa\pa^s\psi^\ep\|_{L^2}^2(t)&\les \|\pa\pa^s\psi^\ep\|_{L^2}^2(0)+\int_{0}^{t}(\|\pa^2 g^\ep\cdot \pa^2\psi^\ep\|_{L^2}
+\|\pa^s(\delta P_{\mu\nu}+\delta Z_{\mu\nu}+\delta Q_{\mu\nu})
\|_{L^2})\|X\pa^s\psi^\ep\|_{L^2}ds\\
&+
\ep\int_{0}^{t}
\|\pa\pa^{s_1}\psi^\ep\|_{L^2}^2ds
+\delta^2\int_{0}^{t}\|\pa^s T^R_{\mu\nu}\|_{L^2}\|X \pa^s\psi^\ep\|_{L^2}ds+\delta^2\left|\int_{0}^{t}\int_{\mathbb{R}^3}
\pa^s T^S_{\mu\nu} \cdot X\pa^s\psi^\ep_{\mu\nu}
d\vol\right|,
 \end{split}
\end{equation*}
where $|s_1|\les |s|$. Recall that $h^\ep(t, x)=h(\ep t, \ep x)$. By assumptions \eqref{condfix}, we conclude that
\[
 \||x|\pa^{s+2}h^\ep\|_{L^\infty}+\||x|(\pa h^\ep)(\pa h^\ep)\|_{L^\infty}+
\|\pa^{s+1} h^\ep\|_{L^\infty}\les \ep,\quad |\psi^\ep|+|\pa\psi^\ep|\les \ep.
\]
By the definitions of $\delta Q_{\mu\nu}$, $\delta Z_{\mu\nu}$, $\delta P_{\mu\nu}$ given in \eqref{defofPQ}, for $|s|\leq 2$,
we can estimate
\begin{align*}
\|\pa^s\delta  Z_{\mu\nu}\|_{L^2}&\les \ep\sum\limits_{s_1\leq s-1}\|\pa^{s_1+1}\psi^\ep\|_{L^2}+\sum\limits_{s_2\leq s}\||x|
\pa^{s_2+2}
h^\ep \cdot |x|^{-1}\psi^\ep\|_{L^2}\les \ep\sum\limits_{s_1\leq s-1}\|\pa^{s_1+1}\psi^\ep\|_{L^2},\\
\|\pa^s \delta P_{\mu\nu}\|_{L^2}&\les \ep\sum\limits_{s_1\leq s}\|\pa^{s_1+1} \psi^\ep\|_{L^2}+\||x|\pa^s(\pa h^\ep \cdot \pa h^\ep)
\cdot |x|^{-1}\psi^\ep\|_{L^2}\les \ep\sum\limits_{s_1\leq s}\|\pa^{s_1+1} \psi^\ep\|_{L^2},\\
\|\pa^s \delta Q_{\mu\nu}\|_{L^2}&\les \ep\sum\limits_{s_1\leq s}\|\pa^{s_1+1} \psi^\ep\|_{L^2}+\||x|\pa^s(\pa h^\ep \cdot \pa h^\ep)
\cdot |x|^{-1}\psi^\ep\|_{L^2}+\|\pa^2\psi^\ep\cdot\pa^2\psi^\ep\|_{L^2}\\
&\les \ep\sum\limits_{s_1\leq s}\|\pa^{s_1+1} \psi^\ep\|_{L^2} +\|\pa^2\psi^\ep\|_{H^1}^2.
\end{align*}
Here we use Lemma \ref{lemhardyineq} to bound $\||x|^{-1}\psi^\ep\|_{L^2}$.

For the soliton part, first notice that with the modified curve $\tilde{\la}$, we can compute
\[
 X f_{\om_0}(z(x;\tilde{\la}))=-\nabla_z f_{\om_0} A_{u_0} u,\quad |u(t)|\les \int_{0}^{t}|\dot \ga|ds\les \ep.
\]
Therefore by Proposition \ref{hiorderga} and inequality \eqref{gadottte}, we have
\[
 \|X \pa^s f_{\om_0}(z)(1+|x|)\|_{L^2}(t)\les |u|+|\dot \ga|+|\ddot \ga|+|\pa_t^3 \ga|\les \ep, \quad \forall t\leq T/\ep,
\quad |s|\leq 4.
\]
Using integration by parts and Lemma \ref{lemhardyineq}, we can show that
\begin{align*}
 &\left|\int_{0}^{t}\int_{\mathbb{R}^3}\pa^s T^S_{\mu\nu} \cdot X\pa^s\psi^\ep_{\mu\nu} d\vol\right|\\
&\les
\|\pa^s T^S_{\mu\nu} \cdot \pa^s\psi^\ep_{\mu\nu}\|_{L^1}(t)+\|\pa^s T^S_{\mu\nu} \cdot \pa^s\psi^\ep_{\mu\nu}\|_{L^1}(0)
+\int_{0}^{t}\|X \pa^s T_{\mu\nu}^S
\cdot \pa^s\psi^\ep_{\mu\nu}\|_{L^1}ds\\
&\les\|\pa^s T_{\mu\nu}^S(1+|x|)\|_{L^2}\|\pa^s\psi^\ep_{\mu\nu}(1+|x|)^{-1}\|_{L^2}(t)+\|\pa^s T_{\mu\nu}^S(1+|x|)
\|_{L^2}\|\pa^s\psi^\ep_{\mu\nu}(1+|x|)^{-1}\|_{L^2}(0)\\
&\quad +\int_{0}^{t}\|X \pa^s T_{\mu\nu}^S(1+|x|)\|_{L^2}\|\pa^s\psi^\ep_{\mu\nu}(1+|x|)^{-1}\|_{L^2}ds\\
&\les \|\pa^{s+1}\psi^\ep\|_{L^2}(t)+\|\pa^{s+1}\psi^\ep\|_{L^2}(0)+\ep\int_{0}^{t}\|\pa^{s+1}\psi^\ep\|_{L^2}ds.
\end{align*}
For $\|\pa^2 g^\ep\cdot \pa^2 \psi^\ep\|_{L^2}$, note that
\[
 \|\pa^2 g^\ep\cdot \pa^2 \psi^\ep\|_{L^2}\les \ep^2\|\pa^2\psi^\ep\|_{L^2}+\|\pa^2\psi^\ep\cdot \pa^2\psi^\ep\|_{L^2}\les
\ep^2\|\pa^2\psi^\ep\|_{L^2}+\|\pa^2\psi^\ep\|_{H^1}^2.
\]
Hence by the conditions \eqref{inicondgf}, we can show that
\begin{equation*}
 \begin{split}
\sum\limits_{|s|\leq2}\|\pa\pa^s\psi^\ep\|_{L^2}^2(t)&\les \delta^4+\ep\sum\limits_{|s|\leq2}\int_{0}^{t}\|\pa^{s+1}\psi^\ep\|_{L^2}^2ds
+\int_{0}^{t}\|\pa^{s+1}\psi^\ep\|_{L^2}^3ds+\delta^2\ep\int_{0}^{t}\|\pa^{s+1}\psi^\ep\|_{L^2}ds\\
&\quad+\delta^2\|\pa^{s+1}\psi^\ep\|_{L^2}(t).
 \end{split}
\end{equation*}
Since we have assumed
\[
 \|\pa^{s+1}\psi^\ep\|_{L^2}\leq 2 \ep, \quad \forall t\leq T/\ep,
\]
Gronwall's inequality then implies that
\[
 \|\pa\pa^s\psi^\ep\|_{L^2}(t)\les \delta^2,\quad \forall t\leq T/\ep,\quad |s|\leq 2.
\]
Thus the proposition follows.
\end{proof}

\subsection{Proof of Theorem \ref{thmbaby}}
The foliation of the spacetime is not preserved under the change of
coordinate system constructed in Lemma \ref{propfermiCord}. Since the argument of Theorem \ref{thmfix}
relies on the new Fermi coordinate system, we first extend the given vacuum spacetime
$([0, T]\times\mathbb{R}^3, h)$ to $([0,
T+2\delta_1]\times\mathbb{R}^3, h)$ for some small positive constant
$\delta_1$. This can be obtained due to the assumptions \eqref{condfix} together with the local
existence result for Einstein equations \cite{LocalEx_bruhat}. The
metric $h$ still satisfies condition \eqref{condfix} but with some new
constant $K_0$. Therefore Lemma \ref{propfermiCord} implies that one can
choose a Fermi coordinate system $(s, y)\in[0,
c_0(T+\delta_1)]\times \mathbb{R}^3$ on a subspace $M$ such that
\begin{equation}
\label{MofM}
\mathcal{M}=[0, T]\times\mathbb{R}^3\subset M\subset [0,
T+2\delta_0]\times\mathbb{R}^3,
\end{equation}
where $c_0=C(u_h(0), h)$, depending only on the initial data
$\la_0\in\La_{\textnormal{stab}}(0)$(see the definition in the proof of Lemma \ref{propfermiCord}).
 We now can identify $M$ with the space $[0,
c_0(T+\delta_1)]\times \mathbb{R}^3$ with the Fermi coordinate
system $(s, y)$. Then on the rescaled space $[0,
c_0(T+\delta_1)/\ep]\times\mathbb{R}^3$, the hyperbolic system
\eqref{EQUATIONRED} with initial data described in
 Lemma \ref{lemredeq} admits a unique solution $(g^\ep, \phi)$ on $[0, t^*)\times\mathbb{R}^3$ for some small time $t^*$.
Moreover, Proposition
\ref{proppsiep} implies that under the assumption
\[
 \sup\limits_{0\leq s\leq t^*}\sum\limits_{|\a|\leq 2}\|\pa^\a\psi^\ep\|_{L^2}(s)\leq 2\ep^2,
\]
we in fact can show that
\[
 \sup\limits_{0\leq s\leq t^*}\sum\limits_{|\a|\leq 2}\|\pa^\a\psi^\ep\|_{L^2}(s)\leq C_3\delta^2\leq C_3\ep^{2q}(\textnormal{or }
C_3\ep_0^2 \ep^2 \textnormal{ if }\delta=\ep_0 \ep)
\]
for some constant $C_3$ independent of $\ep$. Note that $q>1$.
Additional to the requirement on $\ep$ in Theorem \ref{thmfix}, if
we choose $\ep$(or $\ep_0$) such that
\[
C_3 \ep^{2q-2}(\textnormal{or }C_3\ep_0^2)\leq 1,
\]
then we can improve the bootstrap assumption \eqref{bapsi}. This also implies that the solution $(g^\ep, \phi)$ of \eqref{EQUATIONRED}
can be extended to the whole space $[0, c_0(T+\delta_1)/\ep]\times\mathbb{R}^3$ such that
\begin{align*}
& \sum\limits_{1\leq |\b|\leq 3}\|\pa^\b(g^\ep-h^\ep)\|_{L^2}(s)\les \delta^2,\quad \forall s\leq c_0(T+\delta_1)/\ep,\\
&\sum\limits_{|\b|\leq 3}\|\pa^\b\left(\phi-\phi_S(y;\la^\ep(s))\right)\|_{L^2}(s)\les\ep, \quad \forall s\leq c_0(T+\delta_1)/\ep,
\end{align*}
where we define the modulation curve $\la^\ep(s)$ as follows
\[\la^\ep(s)=(\om(\ep s), \ep^{-1}\th( \ep s),
\ep^{-1}\zeta(\ep s)+u_0  s, u_0+u(\ep s
))\in\La_{\textnormal{stab}}.
\]
Moreover, this curve is close to the given time like geodesic $(s, u_0s)$ in the sense that
\[
|\zeta(s)|+ |\om(s)-\om(0)|+|u(s)|\les\ep,
\quad \forall s\leq c_0(T+\delta_1).
\]

\bigskip

We now use these results to construct solutions of the Einstein
equations \eqref{EQUATION}. Under the Fermi coordinate system,
rescale the spacetime $([0, c_0(T+\delta_1)/\ep]\times\mathbb{R}^3],
g^\ep, \phi)$ to
 $([0, c_0(T+\delta_1)]\times\mathbb{R}^3, g, \phi^\ep)$ in the following way
\[
 g(s, y)=g^\ep(s/\ep, y/\ep), \quad \phi^\ep(s, y)=\delta \phi(s/\ep, y/\ep).
\]
Making use of Lemma \ref{lemredeq}, we conclude that $([0, c_0(T+\delta_1)]\times\mathbb{R}^3, g, \phi^\ep)=(M, g, \phi^\ep)$ solves the
Einstein equations \eqref{EQUATION} and satisfies the initial data $(\Si_0, \bar g, \bar K, \phi_0^\ep, \phi_1^\ep)$ given in
Theorem \ref{thmbaby}. The Fermi coordinate system $(s, y)$ on $M$ leads to a foliation $\Si_\tau$ of $M$ in a natural way
$
 \Si_\tau:=\{s=\tau\}
$ such that for such foliation we have the following estimates for the solution $(g, \phi^\ep)$
\begin{align*}
&\sum\limits_{1\leq|\b|\leq 3}\ep^{\b-\frac{3}{2}}\|\pa^\b(g-h)\|_{L^2}(s)\les \delta^2,\quad \forall s\leq c_0(T+\delta_1),\\
&\sum\limits_{|\a|\leq 3}\ep^{\a-\frac{3}{2}}\|\pa^\a\left(\delta^{-1}\phi^\ep-\phi_S(y/\ep;\la^\ep(s/\ep))\right)\|_{L^2}
(s)\les\ep, \quad \forall s\leq c_0(T+\delta_1).
\end{align*}
Let the $C^1$ curve
\[
 \la(s)=(\om(s), \th(s), \zeta(s)+u_0 s, u_0+u(s))\in\La_{\textnormal{stab}}
\]
be defined from $\la^\ep(s)$, which has been given above. Recall Definition \ref{defofsolfix} for $\phi_S^\ep(y;\la(s))$. We note that
\[
 \phi_S(y/\ep;\la^\ep(s/\ep))=\phi_S^\ep(y;\la(s))
\]
if $h(s, \zeta(s)+u_0 s)=m_0$. However, since
\[
 |\zeta(s)|\les \ep,\quad h(s, u_0s)=m_0,
\]
we conclude from \eqref{pofhm} that
\[
\sum\limits_{|\a|\leq 3}\ep^{\a-\frac{3}{2}}\|\pa^\a\left(\phi_S^\ep(y;\la(s))-\phi_S(y/\ep;\la^\ep(s/\ep))\right)\|_{L^2}(s)
\les|h(s, \zeta(s)+u_0s)-h(s, u_0s)|\les\ep^2.
\]
In particular, we have shown that
\begin{align*}
&\|\pa(g-h)\|_{H^2_\ep}\les \ep^{-1}\delta^2\les\ep,\\
 &\|\delta^{-1}\phi^{\ep}(s, y)-\phi_S^\ep(y;\la(s))\|_{H^3_\ep}+\ep\|\delta^{-1}\phi_t^\ep(s, y)-\psi_S^\ep(y;\la(s))\|_{H_\ep^2}\les
\ep, \quad \forall s\leq c_0(T+\delta_1).
\end{align*}
This proves estimates \eqref{metricclose}, \eqref{solitonclose}.

Finally, since the space $M$(diffeomorphic to $[0, c_0(T+\delta_1)]$) can be viewed as an extension of $\mathcal{M}$ by \eqref{MofM},
restricting the solution $(M, g, \phi^\ep)$ to $\mathcal{M}$, we obtain a solution $(\mathcal{M}, g, \phi^\ep)$ of \eqref{EQUATION}
as well as a $C^1$ curve $\la(s)=(\om(s), \th(s), \zeta(s)+u_0s, u_0+u(s))$ such that \eqref{geoclose} holds. This solution is unique
up to diffeomorphism by a result in \cite{Uniqu_bruhat}. This completes the proof of the main Theorem \ref{thmbaby}.

\section{Existence of Initial Data}
In this section, we discuss the existence of the initial data $(\mathbb{R}^3, \bar g, \bar{K}, \phi^\ep_0,
 \phi^\ep_1)$ satisfying the conditions in Theorem \ref{thmbaby}.

\bigskip

Let $(\mathbb{R}^3, \bar h, \bar k)$ be the given
initial data for the vacuum spacetime $(\mathcal{M}, h)$, satisfying the vacuum constraint equations
\begin{equation}
 \label{vaconeq}
R(\bar h)-|\bar k|^2+(\tr \bar k)^2=0,\quad \nabla^j \bar k_{ij}-\nabla_i \tr\bar k=0,
\end{equation}
where $R$ denotes the scalar curvature on $(\mathbb{R}^3, \bar h)$, $\nabla$ is the covariant derivative with
 respect to $\bar h$. Let $\{x|(x_1, x_2, x_3)\}$ be a coordinate system on $\mathbb{R}^3$. Assume $\phi_0^\ep$, $\phi^\ep_1$ are given
functions on $\mathbb{R}^3$ of the form
\[
 \phi^\ep_0(x)=\delta\phi_0(x/\ep)=\delta\phi_0(\cdot/\ep),\quad \phi^\ep_1(x)=\delta\ep^{-1}\phi_1(x/\ep)
=\delta\ep^{-1}\phi_1(\cdot/\ep),
\]
where $\delta=\ep^q$, $q>1$ or $\delta=\ep_0 \ep$.
We want to show that there exists a Riemannian metric
$\bar g$ and a symmetric two tensor $\bar K$ on $\mathbb{R}^3$ satisfying the Einstein constraint equations
\begin{equation}
\label{coneq}
\begin{cases}
  \bar R(\bar g)-|\bar K|^2+(tr \bar K)^2=\delta^2\ep^{-2}(|\phi_1|^2+|\bar \nabla \phi_0|^2+2\mathcal{V}(\phi_0))(\cdot/\ep),\\
 \bar{\nabla}^j \bar{K}_{ij}-\bar{\nabla}_i tr\bar K=\delta^2\ep^{-2}<\phi_1, \bar \nabla_i \phi_0>(\cdot /\ep),
\end{cases}
\end{equation}
as well as the estimates
\begin{equation}
\label{initialdata}
 \|\nabla(\bar g-\bar h)\|_{H_\ep^2(\mathbb{R}^3)}+\|\bar K-\bar k\|_{H_\ep^2(\mathbb{R}^3)}
\leq C(\phi_0, \phi_1, \bar h, \bar k)\delta^2 \ep^{-1}.
\end{equation}
Here $\bar \nabla $ is the covariant derivative for the unknown metric $\bar g$ and the function $\mathcal{V}$
is defined in \eqref{defofmV}.

\bigskip

We define the weighted Sobolev space $H^{s, w}$ on $\mathbb{R}^3$
\begin{equation*}
 \|\phi\|_{H^{s, w}}:=\left(\sum\limits_{0\leq m\leq s}\int_{\mathbb{R}^3}|\pa^m \phi|^2(1+|x|^2)^{w+m}dx\right)^{\f12},
\end{equation*}
and weighted H$\ddot{o}$lder space $C^{0, w}$
\[
\|\phi\|_{C^{0, w}}:= \sup\limits_{x}\left\{(1+|x|)^\om |\phi(x)|\right\}.
\]
We define the metric space $M^{s, w}$ on $\mathbb{R}^3$ as follows:
\[
 M^{s, w}:=\{\textnormal{Riemannian metric $g$};g_{ij}-(m_0)_{ij}\in H^{s, w}\},
\]
where $m_0$ is the Euclidean metric on $\mathbb{R}^3$, that is, $(m_0)_{ii}=1$, $(m_0)_{ij}=0$ if $i\neq j$. We have the following
existence result of the initial data $(\mathbb{R}^3, \bar g, \bar K, \phi^\ep_0, \phi_1^\ep)$.
\begin{thm}
 \label{exdata}
Let $(\Si_0, \bar h, \bar k)$ be the initial data for the vacuum spacetime $(\mathcal{M}, h)$ such that
$\bar h\in M^{4, -1}$, $\bar k\in H^{3, 0}$, $\tr \bar k=0$. Assume the matter field $\phi_0\in H^{3, -1}$, $\phi_1\in
 H^{2, 0}$. Then there exists
$\ep_0>0$ such that for all $\ep<\ep_0$, there exists a Riemannian metric $\bar g$ and a symmetric two tensor $\bar K$ satisfying
the constraint equations \eqref{coneq} and the estimates \eqref{initialdata}.
\end{thm}
\begin{remark}
The method here also applies to the case $\tr \bar k=constant$ which has been studies in \cite{Imfunc_bruhat}. However, the
assumption $\bar k\in H^{3, 0}$ together with $\tr\bar k=constant$ imply that $\tr \bar k=0$. For the general case $\tr \bar k\neq 0$,
see the work of Corvino and Schoen \cite{corvino_schoen}.
\end{remark}

The existence of $\bar g$, $\bar K$ has been shown in \cite{GREE_bruhat}, \cite{Imfunc_bruhat} by using implicit
function theorem. For completeness, we repeat the proof. However, the difficulty here
is to show that $\bar g$, $\bar K$ obey the estimates \eqref{initialdata} for all $\delta=\ep^q$, $q>1$ or $\delta=\ep_0\ep$, in particular
the estimate
\[
 \|\nabla(\bar g-\bar h)\|_{L^2}+\|\bar K-\bar k\|_{L^2}\les C\delta^2\ep^{\f12}.
\]
The approach of previous works can only imply
\[
 \|\nabla(\bar g-\bar h)\|_{L^2}+\|\bar K-\bar k\|_{L^2}\les C\delta^2\ep^{-\f12}.
\]
We improve this estimate by relying on a Hardy's inequality(see Lemma \ref{lemhardyineqV}) for the first order linear operator
 $L_V$(defined below), which does not have any nontrivial kernel in the class $H^{2, -1}$.
Before proving this theorem, we make a convention that $A\les B$ means
$A\leq CB$ for some constant $C$ depending on $\bar h$, $\phi_0$, $\phi_1$.
\begin{remark}
Suggested by our argument for the main theorem, the previous results
in \cite{GREE_bruhat}, \cite{Imfunc_bruhat} may imply the above
existence theorem if we consider the rescaled constraint equations.
However, the estimates depends on the $H^{4, -1}$ norm of the given
metric $\bar h$. After scaling, the $H^{4, -1}$ norm of the scaled
metric $\bar h^\ep(x)=\bar h(\ep x)$ depends on $\ep$. In fact, it
can be shown that the scaled norm has size
$\ep^{-\frac{3}{2}+w}=\ep^{-\f12}$. Hence considering the scaled
constraint equations can not lead to the above theorem directly.
\end{remark}

\bigskip

We denote the \textbf{Hamiltonian constraint}
\[
 \mathcal{H}((\bar g, \bar K), (\phi_0, \phi_1))=R(\bar g)-|\bar K|^2+(tr \bar K)^2-\delta^2\ep^{-2}(|\phi_1|^2+|\bar\nabla\phi_0|^2-
2\mathcal{V}(\phi_0))(\cdot/\ep),
\]
and the \textbf{momentum constraint}
\[
 \mathcal{M}((\bar g, \bar K), (\phi_0, \phi_1))=\bar{\nabla} \bar K_{\cdot }-\bar{\nabla} tr\bar K-\delta^2\ep^{-2}<\phi_1,
\bar \nabla \phi_0>(\cdot/\ep),
\]
where $\bar \nabla$ is the covariant derivative with respect to $\bar g$. Define the spaces
\begin{align*}
 &X:=\{(\bar g-m_0, \bar K)|\bar g\in M^{3, -1},\quad \bar K\in H^{2, 0}\},\\
&Y:=\{(\phi_0, \phi_1)|\phi_0\in H^{3, -1},\quad \phi_1\in H^{2, 0}\},\\
&Z:=\{(\rho, J)|\rho\in H^{1, 1},\quad J\in H^{1, 1}\},
\end{align*}
where $\rho$ is scalar function, $J$ is vector valued function on $\mathbb{R}^3$. We define the constraint map
\begin{align*}
 \Phi: X\times Y&\rightarrow Z,\\
(\bar g-m_0, \bar K)\times (\phi_0, \phi_1)&\mapsto (\mathcal{H}((\bar g, \bar K), (\phi_0, \phi_1)), \mathcal{M}(
(\bar g, \bar K), (\phi_0, \phi_1))).
\end{align*}
The fact that $\Phi$ is a map from $X\times Y$ to $Z$ follows from
the multiplication and embedding properties of the weighted Sobolev
spaces. We state Lemma 2.4 and Lemma 2.5 in \cite{Bruhat_christo}
here.
\begin{lem}
 \label{mulweiSob}
We have
\begin{align*}
 &\|fg\|_{H^{s,w}}\leq C\|f\|_{H^{s_1, w_1}}\|g\|_{H^{s_2, w_2}}, \quad s_1+s_2>s+\frac{3}{2}, \quad w_1+w_2>w-\frac{3}{2},\\
&\|f\|_{C^{0, w'}}\leq C\|f\|_{H^{s, w}},\quad s>\frac{3}{2},\quad w'<w+\frac{3}{2},
 \end{align*}
where the constant $C$ depends only on $s$, $w$, $s_1$, $w_1$, $s_2$, $w_2$, $w'$.
\end{lem}

Let
\[
x_0=(\bar h, \bar k), \quad y_0=(0, 0), \quad x=(\bar g, \bar K), \quad y=\delta^2\ep^{-2}(\phi_0, \phi_1)(\cdot/\ep).
\]
Then the vacuum constraint equations \eqref{vaconeq} become
$\Phi(x_0, y_0)=0$. We define the linear map
\begin{align*}
 D\Phi(x_0, y_0): X&\rightarrow Z,\\
 (g, K)&\mapsto (D\mathcal{H}, D\mathcal{M})
\end{align*}
as the linearization of $\Phi(x, y)$ at the point $(x_0, y_0)$, which can be computed as follows
\begin{align*}
&D\mathcal{H}=-\Delta_{\bar h}(tr_{\bar h}g)+div_{\bar h}(div_{\bar h}g)-g\cdot Ric(\bar h)-2\bar k\cdot K+2(tr_{\bar h}\bar
k)(tr_{\bar h}K), \\
&D\mathcal{M}=\nabla_j K_{i}^{j}-\nabla_i (tr_{\bar h} K)+\f12k_i^l \nabla_l tr_{\bar h}g-\f12 k^{jl}\nabla_i g_{jl}.
\end{align*}
Here the covariant derivative $\nabla$ is for the metric $\bar h$. The above formulae
could be found in \cite{GREE_bruhat}
or can be obtained by straightforward computations, noticing that the linearization of the connection is given by
\[
 \delta \Ga_{ij}^l=\f12 \left(\nabla_i g_j^l+\nabla_j g_i^l -\nabla^l g_{ij}\right), \quad g_{i}^{j}=g_{il}\bar h^{lj}
\]
for symmetric two tensor and $g_{ij}$. Moreover, using Lemma \ref{mulweiSob}, we can show that
\[
 \|\Phi((g+\bar h, K+\bar k), y_0)-\Phi((\bar h, \bar k), y_0)-(D \mathcal{H}, D\mathcal{M})\|_{Z}\les\|(g, K)\|_{X}^2.
\]
To apply the implicit function theorem, we must show that the linear map $D\Phi(x_0, y_0)$ is surjective from $X$ to $Z$, that is,
for any $z=(\rho, J)\in Z$, the equations
\begin{equation}
\label{surjeq}
 D\mathcal{H}=\rho, \quad D\mathcal{M}=J
\end{equation}
have at least one solution $(g, K)\in X$. Notice that the above equations are underdetermined. The linear map $D\Phi(x_0, y_0)$
 has nontrivial kernel. We are instead looking for a solution of the form
\begin{equation}
\label{gkspecialform}
 g=\frac{1}{3}\la \bar h,\quad K=L_{V}\bar h-div(V) \bar h-\f12 \la \bar k
\end{equation}
for some real function $\la$ and vector fields $V$ on $\mathbb{R}^3$. Here $L_V\bar h$ is the deformation tensor of the vector
fields $V$ on $(\mathbb{R}^3, \bar h)$ defined as follows
\[
 (L_V\bar h)_{ij}=\nabla_i V_j+\nabla_j V_i.
\]
Note that $tr \bar k=0$ on $\mathbb{R}^3$. Using \eqref{vaconeq}, the equations \eqref{surjeq} are reduced to
the following elliptic systems
\begin{align}
\label{scalarellp}
-\Delta_{\bar h}\la+|\bar k|^2\la&=3\bar k\cdot L_{V}\bar h+\frac{3}{2}\rho,\\
\label{vectorellp}
div(L_{V}\bar h)&=J.
\end{align}
We must show that the above elliptic systems have a unique solution $(\la, V)\in H^{s, w}$ for any $z=(\rho, J)\in Z$. Since the
systems are splitting. We first
 consider the second equation \eqref{vectorellp} which is independent of $\la$. The following lemma indicates that the
 operator $div(L_{(\cdot) }\bar h)$ is injective
from $H^{2, -1}$ to $H^{0, 1}$.
\begin{lem}
 \label{leminjective}
Let $\bar h\in M^{4, -1}$. If $div(L_{V}\bar h)=0$, $V\in H^{2, -1}$, then $V=0$.
\end{lem}
This result has been proven in \cite{maximal_bruhat}. From a geometric point of view, a killing vector field $V$ is uniquely determined by
$V|_{p}$, $\nabla V|_{p}$. The condition $div(L_{V}\bar h)=0$, $V\in H^{2, -1}$ implies that $V$ is killing and
$V$, $\nabla V$ vanish at infinity. Hence $V$ vanishes everywhere.
However, we give another proof inspired by the method in \cite{boost_christ}, see Theorem 3.3 there.
\begin{proof}
 Since $V\in H^{2, -1}$, we have
\[
 0=\int_{\mathbb{R}^3}\bar h(div(L_{V}\bar h),  V)d\si=-\f12\int_{\mathbb{R}^3}|L_{V}\bar h|^2d\si
\]
by approximating $V$ with vectorfields $V_n\in C_{0}^{\infty}$.
Hence $L_{V}\bar h=0$, that is, $V$ is killing. In local
coordinates, we have $\nabla_{\pa_i}V_j+\nabla_{\pa_j}V_i=0$.
 In particular, we can compute
\begin{equation}
\label{vveq}
 \nabla_{\pa_i}\nabla_{\pa_j}V_{k}=-R(V, \pa_i, \pa_j, \pa_k),
\end{equation}
where $\pa_i$ is the vector field $\pa_{x_i}$, $R$ is the Riemann
curvature tensor defined as follows
\[
 R(\pa_l, \pa_i, \pa_j, \pa_k)=\bar h(\nabla_{\pa_l}\nabla_{\pa_i}\pa_j-\nabla_{\pa_i}\nabla_{\pa_l}\pa_j, \pa_k).
\]
Using Lemma \ref{mulweiSob} by taking $w'=\frac{9}{4}<-1+2+\frac{3}{2}$, we have
 \[
 |R(\pa_l, \pa_i, \pa_j, \pa_k)|\les |\nabla^2 \bar h|\les
 (1+|x|)^{-\frac{9}{4}}, \quad |V|\les (1+|x|)^{-\frac{1}{4}}.
 \]
Hence by Lemma \ref{lemhardyineq}(or Poincar$\acute{e}$ inequality), we can show that
\begin{align*}
 \|(1+|x|)\nabla^2 V\|_{L^2(\mathbb{R}^3/B_{R_0})}&\les (1+R)^{-\frac{1}{4}}\|(1+|x|)^{-1}V\|_{L^2(\mathbb{R}^3/B_{R_0})}\\
&\les (1+R)^{-\frac{1}{4}}\|(1+|x|)\nabla^2 V\|_{L^2(\mathbb{R}^3/B_{R_0})}
\end{align*}
for any ball $B_{R_0}$ with radius $R_0$. Choose $R_0$ large enough. We can conclude that $V$ is vanishing outside the ball
$B_{R_0}$.

\bigskip

Now consider the set
\[
 S:=\{x\in\mathbb{R}^3,\quad V|_{x}=0,\quad \nabla V|_{x}=0\}.
\]
We show that the set $S$ is open. In fact, let $x\in S$. Notice that $V\in C^{0}$. We have
\begin{align*}
 \int_{0}^{r}|V|^2(x+s\om )ds\les r^2\int_{0}^{r}|\nabla V|^2 ds\les r^4\int_{0}^{r}|\nabla^2 V|(x+s\om)ds\les r^4\int_{0}^{r}
|V|^2(x+s\om) ds,\quad \forall \om \in \mathbb{S}^2,
\end{align*}
where we have used the equation \eqref{vveq}. Choosing $r$ small enough, we can show that the ball $B_r(x)\subset S$. Hence $S$ is open.
Notice that $S$ is closed and nonempty. We conclude that $V\equiv0$ on $\mathbb{R}^3$.
\end{proof}

This lemma also implies that the first order linear operator $L_{(\cdot)}\bar h$ is injective. We prove a Hardy's inequality
for this operator, which will be used to improve estimates for $\|(\nabla\la, \nabla V)\|_{L^2}$.
\begin{lem}
 \label{lemhardyineqV}
Assume $\bar h\in M^{4, -1}$, $V\in H^{2, -1}$. Then
\[
 \||x|^{-1}V\|_{L^2}\les \|L_{V}\bar h\|_{L^{2}}.
\]
\end{lem}
\begin{proof}
 Choose $R_0$ such that
\[
 |Ric|\leq \frac{1}{10}(1+|x|)^{-2},\quad |x|\geq R_0.
\]
We claim that
\begin{equation}
\label{hardyineqbaby}
 \|V\|_{L^2(B_{R_0})}\les \|L_{V}\bar h\|_{L^2}.
\end{equation}
In fact, if the above inequality does not hold, assume $V_{n}\in H^{2, -1}$ such that
\[
 1=\|V_n\|_{L^2(B_{R_0})}\geq n\|L_{V_n}\bar h\|_{L^2}.
\]
Integration by parts, we have
\begin{align*}
 \f12\|L_{V}\bar h\|_{L^2}^2&=\int_{\mathbb{R}^3}\nabla^iV^j \nabla_i V_j+\nabla^j V^i \nabla_i V_j d\si\\
&=\int_{\mathbb{R}^3}\nabla^iV^j \nabla_i V_j- V^i \nabla^j\nabla_i V_j +V^i \nabla_i\nabla^j V_j-V^i \nabla_i\nabla^j V_jd\si\\
&=\int_{\mathbb{R}^3}|\nabla V|^2+|div (V)|^2-Ric(V, V)d\si.
\end{align*}
Hence we conclude from the assumptions on $V_n$ that
\[
 \|\nabla V_n\|_{L^2}\leq 1+\|Ric\|_{C^0}\|V_n\|_{L^2(B_{R_0})}+\frac{1}{10}\|(1+|x|)^{-1}V_n\|_{L^2(\mathbb{R}^3/B_{R_0})},
\]
which, by using Lemma \ref{lemhardyineq}, implies that
\[
 \|\nabla V_n\|_{L^2}\leq 2+2\|Ric\|_{C^0}.
\]
Therefore, there exists vectorfields $V$ such that(up to a
subsequence)
\[
 \nabla V_n\rightharpoonup \nabla V \textnormal{ weakly},\quad V_n\rightarrow V \textnormal{ strongly in }L^{2}(B_{0}).
\]
In particular, we have
\[
 \|L_V\bar h\|_{L^2}\leq \varliminf\limits_{n\rightarrow \infty} \|L_{V_n}\bar h\|_{L^2}=0, \quad \|V\|_{L^2(B_{R_0})}=1.
\]
That is $V$ is killing vector field. Moreover, by \eqref{vveq}, we have $V\in H^{2, -1}$. Then Lemma \ref{leminjective} implies that
$V\equiv 0$, which contradicts to $\|V\|_{L^2(B_{R_0})}=1$. Therefore the desired inequality \eqref{hardyineqbaby} holds. Thus
, for $V\in H^{2, -1}$, we can show that
\begin{align*}
 \||x|^{-1}V\|_{L^2}^2&\les\|\nabla V\|_{L^2}^2=2\|L_{V}\bar h\|_{L^2}^2+\int_{\mathbb{R}^3}Ric(V, V)-|div(V)|^2d\si\\
&\leq 2\|L_{V}\bar h\|_{L^2}^2+\|Ric\|_{C^0}\|V\|_{L^2(B_{R_0})}+\frac{1}{10}\|(1+|x|)^{-1}V\|_{L^2(\mathbb{R}^3/B_{R_0})},
\end{align*}
which implies that
\[
 \||x|^{-1}V\|_{L^2}^2\les 2\|L_{V}\bar h\|_{L^2}^2+\|Ric\|_{C^0}\|V\|_{L^2(B_{R_0})}\les \|L_{V}\bar h\|_{L^2}^2.
\]
\end{proof}

\begin{remark}
 An alternative approach for Lemma \ref{leminjective} and Lemma \ref{lemhardyineqV} is to use continuity argument. We sketch the prove
here. Let
\[
 L_t(V)=L_{V}(t \bar h+(1-t)m_0),\quad t\in[0, 1].
\]
If for some $t_0\in[0, 1]$ such that
\[
 \||x|^{-1}V\|_{L^2}\leq C_0 \|L_{t_0}(V)\|_{L^2},\quad \forall V\in H^{1, -1},
\]
then we can show that
\begin{align*}
 \|L_{t_0}(V)\|_{L^2}&\leq C(\bar h)\|L_{t}(V)\|_{L^2}
\end{align*}
for $t$ close to $t_0$. Since Lemma \ref{lemhardyineqV} follows from Lemma \ref{lemhardyineq} if $\bar h=m_0$, we thus conclude
that Lemma \ref{lemhardyineqV} holds for all $\bar h\in M^{2, -1}$, $V\in H^{2, -1}$. In particular, Lemma \ref{leminjective}
follows from Lemma \ref{lemhardyineqV}.

\end{remark}

We now proceed to prove Theorem \ref{exdata}. By Lemma \ref{leminjective}, the operator $L_{(\cdot)}(t\bar h+(1-t)m_0)$
is injective from $H^{s+2, -1}$ to $H^{s, 1}$, for $s=0$ or 1.
Now, for $t=0$, the operator $L_{(\cdot )}m_0=2\Delta$ is a diffeomorphism from $H^{s+2, -1}$ to $H^{s, 1}$, see
Theorem 5.1 in \cite{Bruhat_christo}. Hence the method of continuity \cite{elliptic} implies that the operator $L_{(\cdot)}\bar h$
is a diffeomorphism from $H^{s+2, -1}$ to $H^{s, 1}$. In particular, there exists a unique solution $V\in H^{s+2, -1}$ of
 \eqref{vectorellp} such that
\[
 \|V\|_{H^{s+2, -1}}\les \|J\|_{H^{s, 1}}.
\]
Moreover, by Lemma \ref{lemhardyineqV}, we can show that
\[
 \|L_{V}\bar h\|_{L^2}^2=-2\int_{\mathbb{R}^3}\bar h(V, div(L_{V}\bar h))d\si\les \||x|^{-1}V\|_{L^2}\cdot\||x| J\|_{L^2}
\les \|L_{V}\bar h\|_{L^2}\||x|J\|_{L^2},
\]
which implies that
\begin{equation}
\label{LVineq}
 \|L_{V}\bar h\|_{L^2}\les \||x|J\|_{L^2}.
\end{equation}
Having obtained $V$, for equation \eqref{scalarellp}, by Theorem 6.6 in \cite{Bruhat_christo},
the operator $-\nabla_{\bar h}+|k|^2$ is a diffeomorphism
from $H^{s+2, -1}$ to $H^{s, 1}$, for $s= 0$ or 1. By Lemma \ref{mulweiSob}, $k\cdot L_{V}\bar h\in H^{1, 1}$. Therefore, there
exists a unique solution $\la$ of \eqref{scalarellp} such that
\[
\|\la\|_{H^{s+2, -1}} \les \|k\cdot L_{V}\bar h\|_{H^{s, 1}}+\|\rho\|_{H^{s, 1}}\les \|\rho\|_{H^{s, 1}}+\|J\|_{H^{s, 1}},
\quad s=0, 1.
\]
Moreover, multiply equation
\eqref{scalarellp} by $\la$ and integration by parts, we can show that
\[
 \|\nabla \la\|_{L^2}^2+\|k\la\|_{L^2}^2=-\int_{\mathbb{R}^3}\la\left(3\bar k\cdot L_{V}\bar h+\frac{3}{2}\rho\right)d\si
\les \||x|^{-1}\la\|_{L^2}\left(\||x|\rho\|_{L^2}+\|L_{V}\bar h\|_{L^2}\right),
\]
where $||x| \bar k|\les 1$ by the assumption that $\bar k\in H^{3, 0}$. Using \eqref{LVineq}, we get a similar estimate
\begin{equation}
 \label{laineq}
\|\nabla \la\|_{L^2}\les \||x|\rho\|_{L^2}+\||x|J\|_{L^2}.
\end{equation}
Now we have shown that the linear map $D\Phi(x_0, y_0)$ is surjective from $X$ to $Z$. Hence the Banach space
$X$ can be decomposed as $X=X_1+X_2$ such that
$D\Phi(x_0, y_0)(X_2)=0$, $D\Phi(x_0, y_0)$ is a diffeomorphism from $X_1$ to $Z$. In particular, for
$y=\delta^2\ep^{-2}(\phi_0, \phi_1)(\cdot /\ep)$, the implicit function theorem shows that
there is a solution
\[
 (\bar g, \bar K)=(\bar h+g, \bar k+K)\in X
\]
of \eqref{coneq} if $\ep$(or $\ep_0$) is sufficiently small, depending only on $\bar h$, $\phi_0$, $\phi_1$. Moreover, we can require that $(g, K)$ is of the form
\eqref{gkspecialform} for $(\la, V)\in Z$. Therefore
\[
 \|(\la, V)\|_{H^{2, -1}}\les\delta^2\ep^{-2} \|(|\phi_1|^2+|\bar \nabla\phi_0|^2-2\mathcal{V}(\phi_0))(\cdot/\ep)\|_{H^{0, 1}}+
\|<\phi_1, \bar \nabla\phi_0>(\cdot/\ep)\|_{H^{0, 1}}\les \delta^2\ep^{-\frac{1}{2}},
\]
which implies that
\[
 \ep^{1-\frac{3}{2}}\left(\|\nabla^2(\bar g-\bar h)\|_{L^2}+\|\nabla(\bar K-\bar k)\|_{L^2}\right)\les
\ep^{-\f12}\|(\la, V)\|_{H^{2, -1}}\les \delta^2 \ep^{-1}.
\]
We must remark here that in local coordinate $\{x\}$ we have $\phi_0(\cdot/\ep)=\phi_0(x/\ep)$.
Similarly
\[
  \|(\la, V)\|_{H^{3, -1}}\les\delta^2\ep^{-2} \|(|\phi_1|^2+|\bar \nabla\phi_0|^2-2\mathcal{V}(\phi_0))(\cdot/\ep)\|_{H^{1, 1}}+
\|<\phi_1, \bar \nabla\phi_0>(\cdot/\ep)\|_{H^{0, 1}}\les \delta^2\ep^{-\frac{3}{2}},
\]
which shows that
\[
 \ep^{2-\frac{3}{2}}\left(\|\nabla^3(\bar g-\bar h)\|_{L^2}+\|\nabla^2(\bar K-\bar k)\|_{L^2}\right)\les
\ep^{\f12}\|(\la, V)\|_{H^{3, -1}}\les\delta^2 \ep^{-1}.
\]
It remains to estimate $\|\nabla(\bar g-\bar h)\|_{L^2}+\|\bar K-\bar k\|_{L^2}$, which we rely on \eqref{LVineq},
 \eqref{laineq}. Note that
\[
 |\nabla \bar h|+|\bar k|\les |x|^{-1}.
\]
 Making use of the multiplication properties
 of $H^{s, w}$ in Lemma \ref{mulweiSob}, we can show that
\begin{align*}
 &\|\nabla(\bar g-\bar h)\|_{L^2}+\|\bar K-\bar k\|_{L^2}\les\|\nabla\la \|_{L^2}+\|L_{V}\bar h\|_{L^2}\\
&\les \delta^{2}\ep^{-2}\left(\||x|(|\phi_1|^2+|\bar \nabla \phi_0|^2-2\mathcal{V}(\phi_0))(\cdot/\ep)\|_{L^2}+\||x|<\phi_1, \bar \nabla \phi_0>(\cdot/
\ep)\|_{L^2}\right)+\|\mathcal{N}(g, K)\|_{H^{0, 1}}\\
&\les \delta^{2}\ep^{\f12}+\|g\|_{H^{2, -1}}^2+\|K\|_{H^{1, 0}}^2\les \delta^2\ep^{\f12}+
(\delta^2\ep^{-\f12})^2\les \delta^2\ep^{\f12},
\end{align*}
where $\mathcal{N}(g, K)$ is the nonlinear term in equation \eqref{coneq} for $\bar g=\bar h+g$, $\bar K=\bar k+K$. Hence
\[
 \|\nabla(\bar g-\bar h)\|_{H_\ep^2}+\|\bar K-\bar k\|_{H^2_\ep}=\sum\limits_{|\a|\leq 2}\ep^{|\a|-\frac{3}{2}}\left(
\|\nabla^{\a+1}(\bar g-\bar h)\|_{L^2}+\|\nabla^\a(\bar K-\bar k)\|_{L^2}\right)\les \delta^2\ep^{-1}.
\]
That is the solution $(\bar g, \bar K)$ satisfies the constraint equations \eqref{coneq} as well as
 the estimate \eqref{initialdata}.

\bigskip

\textbf{Acknowledgments} The author is deeply indebted to his
advisor Igor Rodnianski for suggesting this problem and for sharing
numerous valuable thoughts, as well as tremendous indispensable
help.

\bibliography{shiwu}{}
\bibliographystyle{plain}

 \end{document}